\documentclass[11pt]{article}
\usepackage{amssymb,latexsym}
\usepackage{amsmath,amscd}
\usepackage{theorem}
\usepackage[all]{xy}

\title{\bf Dirac geometry, quasi-Poisson actions and $D/G$-valued moment maps}

\author{Henrique Bursztyn\thanks{{\tt henrique@impa.br}}\\[0.1cm]
        Instituto Nacional de Matem\'atica Pura e Aplicada\\
        Estrada Dona Castorina 110,\\
        Rio de Janeiro, 22460-320, Brazil.
        \\[0.2cm]
        Marius Crainic  \thanks{{\tt crainic@math.uu.nl}}
        \\[0.1cm]
         Department of Mathematics\\
        Utrecht University, P.O. Box 80.010, 3508 TA\\
        Utrecht, The Netherlands}
\date{}

\newcommand{\rmap}{\longrightarrow}

\newcommand{\id}         {{\mathrm {Id}}}

\newcommand{\Ker}        {{\mathrm {ker}}}
\newcommand{\ad}         {{\mathrm {ad}}}

\newcommand{\Ad}         {{\mathrm {Ad}}}
\newcommand{\gra}        {{\mathrm {graph}}}

\newcommand{\pr}         {{\mathrm{pr}}}

\newcommand{\M}          {\overline{\mathcal{M}}}
\newcommand{\Mp}         {\mathcal{M}}
\newcommand{\Id}         {\mathrm{Id}}
\newcommand{\I}          {\mathcal{I}}
\newcommand{\op}         {\mathrm{op}}

\newcommand{\SP} [1]     {{\left\langle {{#1}} \right\rangle}}
\newcommand{\SPd} [1]     {{\left\langle {{#1}} \right\rangle_\mathfrak{d}}}
\newcommand{\Bi} [1]     {{\left\langle {{#1}} \right\rangle_{\mathfrak{g}}}}
\newcommand{\Bic} [1]     {{\left\langle {{#1}} \right\rangle_{\mathbb{C}}}}

\newcommand{\frakg}     {\mathfrak{g}}
\newcommand{\frakh}     {\mathfrak{h}}
\newcommand{\frakd}     {\mathfrak{d}}

\newcommand{\gstar}     {\mathfrak{g}^*}
\newcommand{\gc}        {\mathfrak{g}^\mathbb{C}}
\newcommand{\Gc}        {G^\mathbb{C}}

\newcommand{\tri}       {\chi}
\newcommand{\cobr}      {{F}}
\newcommand{\X}         {\mathfrak{X}}
\newcommand{\rma}         {\mathfrak{r}}
\newcommand{\T}         {\mathbb{T}}
\newcommand{\TS}        {\T S}
\newcommand{\TM}        {\T M}

\newcommand{\Sop}       {S^{\mathrm{op}}}
\newcommand{\Mop}       {M^\mathrm{op}}
\newcommand{\Lop}       {L^\mathrm{op}}
\newcommand{\D}         {\delta}
\newcommand{\Inv}       {\mathrm{Inv}}

\newcommand{\sigmav}   {\overline{\sigma}}
\newcommand{\rhov}     {\overline{\rho}}

\newcommand{\grd}         {\mathcal{G}}
\newcommand{\sour}        {\mathsf{s}}
\newcommand{\tar}         {{\mathsf{t}}}

\newcommand{\Cour}[1]      {[\![#1]\!]}

\newcommand{\Lie}        {\mathcal L}

\newtheorem{lemma} {Lemma} [section]
\newtheorem{proposition} [lemma] {Proposition}

\newtheorem{theorem} [lemma] {Theorem}
\newtheorem{corollary} [lemma] {Corollary}
\newtheorem{definition}[lemma] {Definition}

\theorembodyfont{\rm} 

\newtheorem{example}[lemma] {Example}

\newtheorem{remark}[lemma]{Remark}

\newenvironment{proof}{{\sc Proof:}}{{\hspace*{\fill} $\square$\\}}

\numberwithin{equation}{section}

\begin{document}

\maketitle

\begin{abstract}
We study Dirac structures associated with Manin pairs
$(\frakd,\frakg)$ and give a Dirac geometric approach to
Hamiltonian spaces with $D/G$-valued moment maps, originally
introduced by Alekseev and Kosmann-Schwarzbach \cite{AK} in terms
of quasi-Poisson structures. We explain how these two distinct
frameworks are related to each other, proving that they lead to
isomorphic categories of Hamiltonian spaces. We stress the
connection between the viewpoint of Dirac geometry and equivariant
differential forms. The paper discusses various examples,
including q-Hamiltonian spaces and Poisson-Lie group actions,
explaining how presymplectic groupoids are related to the notion
of ``double'' in each context.
\end{abstract}


\tableofcontents

\section{Introduction}

In this paper, we study Dirac structures \cite{Cou90,CouWe,SeWe01}
associated with Manin pairs $(\frakd,\frakg)$ and develop a theory
of Hamiltonian spaces with $D/G$-valued moment maps based on Dirac
geometry. Our approach is parallel to the one originally
introduced by Alekseev and Kosmann-Schwarzbach \cite{AK} to treat
Hamiltonian quasi-Poisson actions, and one of our goals is to
explain how these two points of view are related.

This paper is largely motivated by questions, set forth by
Weinstein \cite{We}, concerning the existence of a unified
geometric framework in which recent generalizations of the notion
of \textit{moment map} (including \cite{AK,AKM,AMM,Lu,MiWe}) would
naturally fit. As it turns out, a fruitful step to address these
questions consists in passing from Poisson geometry, which
describes classical moment maps, to Dirac structures, and our
guiding principle is that generalized moment maps should be seen
as morphisms between Dirac manifolds. Building on \cite{BC}, we
illustrate in this paper how Dirac geometry underlies moment map
theories arising from Manin pairs, providing a natural arena for
their unified treatment.

Our work was also stimulated by the theory of $G$-valued moment
maps,  introduced in \cite{AKM,AMM} in order to give a
finite-dimensional account of the Poisson geometry of moduli
spaces of flat $G$-bundles over surfaces \cite{AB}. A
characteristic feature of $G$-valued moment maps is that they
admit two distinct geometrical formulations: the original approach
of \cite{AMM} is based on \textit{twisted 2-forms} and fits
naturally into the framework of Dirac geometry, see
\cite{ABM,BC,BCWZ}, whereas the description in \cite{AKM} involves
\textit{quasi-Poisson bivector fields}. Although each approach
relies on a different type of geometry, they turn out to be
equivalent, see \cite[Sec.~10]{AKM} and \cite[Sec.~3.5]{BC}. This
paper grew out of our attempt to explain the geometric origins of
these two formulations of $G$-valued moment maps as well as the
equivalence between them. We prove in this paper that the
equivalence between the Dirac geometric and quasi-Poisson
approaches to Hamiltonian spaces holds at the more general level
of $D/G$-valued moment maps. This elucidates, in particular, the
connection between the Hamiltonian quasi-Poisson spaces of
\cite{AK} and the symmetric-space valued moment maps, described by
closed equivariant 3-forms, studied in \cite{Lei}.

The paper is organized as follows.

In Section \ref{sec:generalmoment}, we review the basics of Dirac
geometry, including the integration of Dirac structures to
presymplectic groupoids, and the relationship between Dirac
structures and equivariant cohomology \cite{BCWZ}. In particular,
for a given Dirac manifold $S$, we recall the general notion of
\textit{Hamiltonian space with $S$-valued moment map} \cite{BC}.

We consider Dirac structures associated with Manin pairs in
Section \ref{sec:geomManin}. Given a Manin pair $(\frakd,\frakg)$
integrated by a group pair $(D,G)$ (the definitions are recalled
in Section~\ref{subsec:manin}), we consider, following \cite{AK},
the homogeneous space $S:=D/G$. We view $S$ as a $G$-manifold with
respect to the \textit{dressing action}, induced by the left
multiplication of $G$ on $D$. While the theory of quasi-Poisson
actions \cite{AK} is based on the additional choice of an
isotropic complement of $\frakg$ in $\frakd$ (not necessarily a
subalgebra), making the Manin pair into a Lie quasi-bialgebra
\cite{K-S}, our starting point consists of a distinct choice. We
instead consider the principal $G$-bundle $D\to D/G$, with respect
to the action by right multiplication, and choose an
\textit{isotropic connection} $\theta\in \Omega^1(D,\frakg)$,
i.e., a principal connection whose horizontal distribution is
isotropic in $TD$ (with respect to the invariant pseudo-riemannian
metric induced by $\frakd$). As it turns out, such connection
$\theta$ defines a closed 3-form $\phi_S\in \Omega^3(S)$ as well
as a $\phi_S$-twisted Dirac structure $L_S \subset TS\oplus T^*S$
on $S$. This Dirac structure is best understood in terms of
Courant algebroids: as observed by \v{S}evera \cite{Se} and
Alekseev-Xu \cite{AX}, the trivial bundle $\frakd_S:=\frakd\times
S$ over $S$ is naturally an exact Courant algebroid, and fixing
$\theta$ is in fact equivalent to a choice of identification
$\frakd_S \cong TS\oplus T^*S$ (under which $L_S$ corresponds to
$\frakg$). Upon an extra invariance assumption on $\theta$, the
Dirac structure $L_S$ turns out to be determined by a closed
equivariant extension of the 3-form $\phi_S$.

Starting from a Manin pair $(\frakd,\frakg)$ together with the
choice of an isotropic connection $\theta \in \Omega^1(D,\frakg)$,
we investigate in Section \ref{sec:dirac} the Hamiltonian theory
associated with the Dirac manifold $(S,L_S,\phi_S)$. In this
theory, moment maps are given by suitable morphisms $J:M\to S$
from Dirac manifolds $M$ into $S$ \cite{BC,ABM}. We discuss how
specific examples of Manin pairs equipped with particular choices
of connections lead to many known moment maps theories, including
$G$-valued and $P$-valued moment maps \cite{AMM}, $G^*$-valued
moment maps \cite{Lu}, as well as symmetric-space valued moment
maps \cite{Lei}. This section also includes an explicit
description of presymplectic groupoids integrating the Dirac
manifold $(S,L_S,\phi_S)$, explaining how they lead to the
appropriate notion of ``double'' in different examples. A final
important observation in this section is that the connection
$\theta$ determines an interesting 2-form $\omega_D$ on the Lie
group $D$. This 2-form makes $D$ into a Morita bimodule (in the
sense of \cite{Xu}) between the Dirac manifold $S$ and its
opposite $S^\mathrm{op}$. We use this Morita equivalence to define
a nontrivial involution in the category of Hamiltonian $G$-spaces
with $S$-valued moment maps. In the particular cases where $S=G$
or $S=G^*$, this involution agrees with the one induced by the
inversion map on $G$ or $G^*$.

In Section \ref{sec:quasip}, we revisit the quasi-Poisson theory
developed in \cite{AK}. The main new ingredient in our point of
view is the construction of a Lie algebroid describing
quasi-Poisson actions (particular examples of this Lie algebroid
have appeared in \cite{Lu2} and \cite{BC}, and an alternative
construction was discussed in \cite{BCS}, see also \cite{Ter}).

The aim of Section \ref{sec:equiv} is to relate the two approaches
to Hamiltonian theories associated with Manin pairs. The set-up is
a Manin pair $(\frakd,\frakg)$ together with two extra choices: an
isotropic connection $\theta\in \Omega^1(D,\frakg)$, which leads
to a category of Hamiltonian spaces via Dirac geometry, and an
isotropic complement $\frakh$ of $\frakg \subset \frakd$, which
leads to a category of Hamiltonian spaces described by
quasi-Poisson structures. These extra choices can always be made,
and they are independent of each other. We give an explicit
geometric construction of an isomorphism between the two types of
Hamiltonian categories, generalizing \cite[Thm.~3.16]{BC}
(following the methods in \cite{ABM}). As we will see, the link
between Dirac structures and quasi-Poisson bivector fields lies in
the theory of Lie quasi-bialgebroids \cite{Dima2}. Under the
identification $\frakd_S \cong TS\oplus T^*S$ induced by $\theta$,
the subspace $\frakh\subset \frakd$ defines an almost Dirac
structure $C_S$ on $S$ transverse to $L_S$. Given a moment map
$J:(M,L)\to (S,L_S)$ in the Dirac geometric setting, the pull-back
image of $C_S$ under $J$ (in the sense of Dirac geometry, see e.g.
\cite{ABM,BuRa}) defines an almost Dirac structure $C$ on $M$
transverse to $L$, so  the pair $L, C \subset TM\oplus T^*M$ is a
Lie quasi-bialgebroid. The bivector field on $M$ naturally induced
by this Lie quasi-bialgebroid makes it into a quasi-Poisson space.
This procedure can be reversed and establishes the desired
equivalence of viewpoints. We note that many features of the
Hamiltonian spaces are independent of any of the choices involved,
including the construction of reduced spaces. Lastly, the main
facts about Courant algebroids and Lie quasi-bialgebroids used in
the paper are collected in the Appendix.

There is a more conceptual explanation for the equivalence between
the quasi-Poisson and Dirac geometric viewpoints to Hamiltonian
spaces associated with Manin pairs: as shown in \cite{BIS}, there
is an abstract notion of Hamiltonian space canonically associated
with a Manin pair; when additional (noncanonical) choices are
made, these abstract Hamiltonian spaces take the two concrete
forms studied in this paper.

\noindent{\bf Acknowledgments:} We would like to thank A.
Alekseev, D. Iglesias Ponte, Y. Kosmann-Shwarzbach, J.-H. Lu, E.
Meinrenken, P. \v{S}evera, A. Weinstein and P. Xu for helpful
discussions, as well as the referee for his/her comments. H. B.
thanks CNPq and the Brazil-France cooperation agreement for
financial support, and Utrecht University, the Fields Institute,
and the University of Toronto for their hospitality during various
stages of this project. The research of M. C. was supported by
 the Dutch Science Foundation (NWO) through the Open Competitie project no. 613.000.425.
We thank the Erwin Schr\"odinger Institute for hosting us
when this paper was being completed.

\noindent{\bf Notation}: Given a Lie group $G$ with algebra
$\frakg$ (defined by \textit{right}-invariant vector fields), the
left/right vector fields in $G$ defined by $u\in\frakg$ are
denoted by $u^l, u^r \in \X(G)$. Left and right translations by
$g\in G$ are denoted by $l_g$ and $r_g$, and we write $(u^r)_g=
dr_g(u)$, or simply $r_g(u)$ if there is no risk of confusion
(similarly for left translations). The left/right Maurer-Cartan
1-forms on $G$ are denoted by $\theta^L, \theta^R \in
\Omega^1(G,\frakg)$ and defined by
$\theta^L(u^l)=\theta^R(u^r)=u$.

\section{Dirac geometry and Hamiltonian actions} \label{sec:generalmoment}

In this section, we briefly recall the basics of Dirac geometry
\cite{Cou90,CouWe,SeWe01} and describe how to associate a category
of Hamiltonian spaces to a given Dirac manifold $S$, obtaining a
general notion of \textit{$S$-valued moment map}. We will mostly
follow \cite{BC,BCWZ}.

\subsection{Dirac geometry}\label{subsec:basics}

Let $M$ be a smooth manifold. Consider the bundle $\TM:=TM\oplus
T^*M$ equipped with the symmetric pairing
\begin{equation}\label{eq:symm}
\SP{(X_1,\alpha_1),(X_2,\alpha_2)}:=\alpha_2(X_1) + \alpha_1(X_2).
\end{equation}
An \textbf{almost Dirac structure} on $M$ is a subbundle $L
\subset \TM$ which is \textit{lagrangian} (i.e., maximal
isotropic) with respect to \eqref{eq:symm}. Since the pairing has
split signature, it follows that $\mathrm{rank}(L)=\dim(M)$.
Simple examples of almost Dirac structures include 2-forms
$\omega\in \Omega^2(M)$ and bivectors fields $\pi \in \X^2(M)$,
realized as subbundles of $\TM$ via the graphs of the maps $TM\to
T^*M$, $X\mapsto i_X\omega$ and $T^*M\to TM$, $\alpha\mapsto
i_\alpha\pi$.

Let $\phi \in \Omega^3(M)$ be a closed 3-form on $M$. A {\bf
$\phi$-twisted Dirac structure} \cite{SeWe01} on $M$ is an almost
Dirac structure $L\subset \TM$ satisfying the following
integrability condition: the space of sections $\Gamma(L)$ is
closed under the $\phi$-twisted Courant bracket
\begin{equation}\label{eq:cou}
\Cour{(X_1,\alpha_1),(X_2,\alpha_2)}_{\phi}:=
([X_1,X_2],\Lie_{X_1}\alpha_2-i_{X_2}d\alpha_1+i_{
X_2}i_{X1}\phi),
\end{equation}
where $X_1, X_2\in \X(M)$ and $\alpha_1, \alpha_2 \in
\Omega^1(M)$. For a 2-form $\omega\in \Omega^2(M)$, the
integrability condition amounts to $d\omega+\phi=0$, and for a
bivector field $\pi\in \X^2(M)$ it gives
$\frac{1}{2}[\pi,\pi]=\pi^\sharp(\phi)$ (here $[\cdot,\cdot]$
denotes the Schouten bracket). We will see many other examples
later in this paper. We denote a Dirac manifold by the triple
$(M,L,\phi)$, or simply $(M,L)$ if the 3-form is clear from the
context. Given an $\phi$-twisted Dirac structure $L$ on $M$, we
define its \textbf{opposite} as
\begin{equation}\label{eq:Lop}
\Lop:=\{(X,\alpha)\in \TM\,|\, (X,-\alpha)\in L\},
\end{equation}
which is a $-\phi$-twisted Dirac structure. We often denote
$(M,\Lop,-\phi)$ simply by $\Mop$.

The bracket \eqref{eq:cou}, although not skew-symmetric, becomes a
Lie bracket when restricted to the space of sections of a Dirac
structure $L$. The vector bundle $L\to M$ inherits the structure
of a \textit{Lie algebroid} over $M$, with bracket
$\Cour{\cdot,\cdot}_{\phi}|_{\Gamma(L)}$ and anchor given by
$\pr_{TM}|_{L}$, where $\pr_{TM}:\TM\to TM$ is the natural
projection. As a result, the distribution $\pr_{TM}(L)\subset TM$
is integrable 
and defines a singular
foliation on $M$. Each leaf $\iota: \mathcal{O}\hookrightarrow M$
inherits a 2-form $\omega_\mathcal{O}\in \Omega^2(\mathcal{O})$,
defined at each point $x\in \mathcal{O}$ by
\begin{equation}
{\omega_{\mathcal{O}}}(X_1,X_2)= \alpha(X_2),
\end{equation}
where $\alpha \in T^*_xM$ is such that $(X_1,\alpha)\in L_x$ (the
value of $\omega_\mathcal{O}$ is independent of the particular
choice of $\alpha$). The integrability of $L$ implies that
$d\omega_{\mathcal{O}}+\iota^*\phi=0$. This singular foliation,
equipped with the leafwise 2-forms, is referred to as the
\textbf{presymplectic foliation} of $L$. Note that the leafwise
2-forms are nondegenerate if and only if $L$ is the graph of a
bivector field, and the Lie algebroid of $L$ is
\textit{transitive} (i.e., with surjective anchor) if and only if
$L$ is the graph of a 2-form.

Given manifolds $M$ and $S$ and a smooth map $J:M \to S$, we say
that the elements $(X,\alpha)\in \TM_x$ and $(Y,\beta)\in
\TS_{J(x)}$ are \textbf{$J$-related} at $x$ if
$$
Y=(dJ)_x(X)\;\; \mbox{ and } \;\; \alpha = (dJ)_x^*\beta.
$$
A direct calculation shows the following (see \cite[Sec.~2]{ABM}
and \cite{XS}):
\begin{lemma}\label{lem:Jrel}
If $(X_i,\alpha_i)$ and $(Y_i,\beta_i)$ are $J$-related at $x$,
$i=1,2$, then
$\SP{(X_1,\alpha_1),(X_2,\alpha_2)}_x=\SP{(Y_1,\alpha_1),(Y_2,\alpha_2)}_{J(x)}.$
Also, if $(X_i,\alpha_i)\in \Gamma(\TM)$ and $(Y_i,\beta_i)\in
\Gamma(J^*\TS)$ are $J$-related at all points in a neighborhood of
$x\in M$, then $\Cour{(X_1,\alpha_1),(X_2,\alpha_2)}_{J^*\phi_S}$
is $J$-related to $\Cour{(Y_1,\alpha_1),(Y_2,\alpha_2)}_{\phi_S}$
at $x$.
\end{lemma}
Here $J^*\TS$ denotes the pull-back vector bundle over $M$.

Consider the subbundle $\Gamma_J\subset J^*\TS\oplus \TM$ defined
by
$$
\Gamma_J:=\{ ((Y,\beta),(X,\alpha))\in J^*\TS\oplus \TM\,|\,
(X,\alpha) \mbox{ is $J$-related to } (Y,\beta)\}.
$$
Then $\Gamma_J$ is lagrangian in $J^*\overline{\TS}\oplus \TM$,
where $\overline{\TS}$ is equipped with minus the pairing
\eqref{eq:symm}. We use $\Gamma_J$ to define morphisms of almost
Dirac structures using composition of lagrangian relations,
following \cite{GS,We} (c.f. \cite{ABM,BuRa}). Let $L$ and $L_S$
be almost Dirac structures on $M$ and $S$. We say that $L_S$ is
the \textbf{forward image} of $L$ if $\Gamma_J\circ L=J^*L_S$ at
each $x\in M$, that is,
$$
(L_S)_{J(x)}= \{((dJ)_x(X),\beta)\;|\; X\in T_xM,\; \beta \in
T_{J(x)}^*S,\,\mbox{ and }\, (X,(dJ)_x^*(\beta))\in L_x\},\;
\forall x\in M.
$$
In this case we call $J$ a \textbf{forward Dirac map} (or simply
f-Dirac map). Similarly, we say that $L$ is the \textbf{backward
image} of $L_S$ if $L=(J^*L_S)\circ \Gamma_J$ at each $x\in M$,
which amounts to
$$
L_x = \{ (X,(dJ)_x^*(\beta))\;|\; X\in T_xM,\; \beta \in
T_{J(x)}^*S,\, \mbox{ and }\, ((dJ)_x(X),\beta)\in (L_S)_{J(x)}
\}.
$$
When both $L$ and $L_S$ are (graphs of)
2-forms, the notion of backward image reduces to the usual notion
of pull-back; on the other hand, when both $L$ and $L_S$ are
bivector fields, then the forward image amounts to the
push-forward relation.

Just as Poisson manifolds are infinitesimal versions of symplectic
groupoids \cite{Wesym} (c.f. \cite{CaXu,CrFe}), Dirac manifolds
also have global counterparts. The objects integrating
$\phi$-twisted Dirac structures are \textbf{$\phi$-twisted
presymplectic groupoids} \cite{BCWZ,Xu}, i.e., Lie groupoids
$\grd$ over a base $M$ equipped with a  2-form $\omega \in
\Omega^2(\grd)$ such that:
\begin{enumerate}

\item[$i)$] $\omega$ is \textit{multiplicative}, i.e.,
$m^*\omega=p_1^*\omega+p_2^*\omega$, where $m:\grd\times_M\grd\to
\grd$ is the groupoid multiplication and $p_i:\grd\times_M\grd \to
\grd$, $i=1,2$, are the natural projections onto the first and
second factors.

\item[$ii)$] $d\omega=\sour^*\phi-\tar^*\phi$, where $\sour,\tar$
are source, target maps on $\grd$, and $\phi\in \Omega_{cl}^3(M)$,

\item[$iii)$] $\dim(\grd)=2\dim(M)$,

\item[$iv)$] $\ker(\omega)_x\cap\ker(d\sour)_x\cap
\ker(d\tar)_x=\{0\}$, for all $x\in M$.
\end{enumerate}
 If $\omega$ satisfies condition $ii)$, then we say
that it is \textbf{relatively $\phi$-closed}. Conditions $i)$ and
$ii)$ together are equivalent to $\omega + \phi$ being a 3-cocycle
in the \textit{bar-de Rham complex} of the Lie groupoid $\grd$
\cite{BSS} (see also \cite{Xu}), i.e., the total complex of the
double complex $\Omega^p(\grd_q)$ (here $\grd_q$ denotes the space
of composable sequence of $q$-arrows) computing the cohomology of
$B\grd$.

As proven in \cite{BCWZ}, any $\phi$-twisted presymplectic
groupoid $\grd$ over $M$ defines a canonical $\phi$-twisted Dirac
structure $L$ on $M$, uniquely determined by the fact that $\tar$
is an f-Dirac map (whereas $\sour$ is anti f-Dirac). Moreover,
there is an explicit identification (as Lie algebroids) of $L$
with the Lie algebroid of $\grd$. In this context, we say that the
presymplectic groupoid is an \textbf{integration} of the Dirac
manifold $(M,L,\phi)$. Conversely, if a $\phi$-twisted Dirac
structure $L$ on $M$ is integrable (as a Lie algebroid), then the
corresponding $\sour$-simply-connected groupoid admits a unique
$\phi$-twisted presymplectic structure integrating $L$. We will
see many concrete examples of presymplectic groupoids in Section
\ref{subsec:doubles}.

\subsection{The Hamiltonian category of a Dirac manifold}\label{subsec:hamcat}

The fact that moment maps in symplectic geometry are
\textit{Poisson maps} indicates that moment maps in Dirac geometry
should be represented by \textit{f-Dirac maps}. Indeed, we need a
special class of f-Dirac maps to play the role of moment maps.

Given Dirac manifolds $(M,L,\phi)$ and $(S,L_S,\phi_S)$, we call a
smooth map $J:M\to S$ a \textbf{strong Dirac map} if
\begin{enumerate}
\item $\phi = J^*\phi_S$,

\item $J$ is an f-Dirac map: $\Gamma_J\circ L = J^*L_S$,

\item Denoting $\ker(L):=L\cap TM$, the following transversality
condition holds:
\begin{equation}\label{eq:transv}
\ker(dJ)_x\cap \ker(L)_x=\{0\},\;\; \forall x\in M,
\end{equation}
(This is equivalent to the composition $\Gamma_J\circ L$ being
transversal.)
\end{enumerate}
Strong Dirac maps are alternatively called \textit{Dirac
realizations} in \cite{BC} (see also \cite{ABM}). In particular,
we will refer to a strong Dirac map $J:M\to S$ for which the Dirac
structure on $M$ is a 2-form as a \textbf{presymplectic
realization} of $S$. An immediate example of a presymplectic
realization is the inclusion $\iota:\mathcal{O}\hookrightarrow S$
of a presymplectic leaf. More generally, since the composition of
strong Dirac maps is a strong Dirac map, the restriction of a
strong Dirac map $M\to S$ to leaves of $M$ define presymplectic
realizations of $S$.

\begin{definition}
The \textbf{Hamiltonian category} of a Dirac manifold
$(S,L_S,\phi_S)$ is the category $\M(S,L_S,\phi_S)$ whose objects
are strong Dirac maps $J:M\to S$ and morphisms are smooth maps
$\varphi:M\to M'$ which are f-Dirac maps and such that $J'\circ
\varphi = J$. We denote by $\Mp(S,L_S,\phi_S)$ the subcategory of
presymplectic realizations.
\end{definition}

This definition will be justified by general properties of strong
Dirac maps as well as concrete examples. First of all, a strong
Dirac map $J:M\to S$ induces a canonical action on $M$. Indeed,
the properties of $J$ define a smooth bundle map \cite{ABM,BC}
\begin{equation}\label{eq:rhoh}
\rho_M:J^*L_S\to TM,
\end{equation}
where $X=\rho_M(Y,\beta)$ is uniquely determined by the conditions
\begin{equation}\label{eq:actc}
(dJ)_x(X)=Y \;\; \mbox{ and }\;\; (X, (dJ)_x^*(\beta))\in L.
\end{equation}
Let us also consider the bundle map
\begin{equation}\label{eq:rhoh1}
\widehat{\rho}_M:J^*L_S\to L, \;\;\; (Y,\beta)\mapsto
(\rho_M(Y,\beta),(dJ)_x^*(\beta)).
\end{equation}
A direct computation shows that, at the level of sections, the
induced map $\widehat{\rho}_M:\Gamma(L_S)\to \Gamma(L)$ preserves
Lie brackets, and hence so does $\rho_M:\Gamma(L_S)\to \X(M)$. As
a result, we have

\begin{proposition}\label{prop:act}
If $J:M\to S$ is a strong Dirac map, then the map $\rho_M$
\eqref{eq:rhoh} defines a Lie algebroid action of $L_S$ on $M$.
\end{proposition}

More on Lie algebroid actions can be found e.g. in \cite{Mac}.

It immediately follows from \eqref{eq:actc} that the action
${\rho}_M$ is tangent to the presymplectic leaves of $M$. In
particular, when $L$ is defined by a 2-form $\omega$ on $M$ (i.e.,
$J$ is a presymplectic realization), then the conditions in
\eqref{eq:actc}, relating $J$ and the action ${\rho}_M$, take the
form
$$
dJ({\rho}_M(Y,\beta))=Y\;\;\mbox{ and }\;\;
i_{{\rho}_M(Y,\beta)}\omega=J^*\beta,
$$
which can be interpreted as an \textit{equivariance condition} for
$J$ (with respect to the canonical action of $L_S$ on $S$)
together with a \textit{moment map condition}. In this sense, we
think of $M$ as carrying a \textit{Hamiltonian action} and $J:M\to
S$ as an \textit{$S$-valued moment map}. Various properties of
usual Hamiltonian spaces are present in the framework of strong
Dirac maps. For example, as discussed in \cite[Sec.~4.4]{BC},
there is a natural reduction procedure generalizing
Marsden-Weinstein's reduction \cite{MW} (a particular case of
which will be recalled in Section~\ref{subsec:DG1}).

Let us recall some examples of Hamiltonian spaces defined by
strong Dirac maps \cite{BC}.

\begin{example}
The identity map $\Id:S\to S$ is always an object in
$\M(S,L_S,\phi_S)$, whereas inclusions of presymplectic leaves
$\iota:\mathcal{O}\hookrightarrow S$ are objects in
$\Mp(S,L_S,\phi_S)$.
\end{example}

\begin{example}\label{ex:grd}
If $(\grd,\omega)$ is a presymplectic groupoid integrating
$(S,L_S,\phi_S)$, then
$$
(\tar,\sour):\grd \to S\times S^\op
$$
is a strong Dirac map, i.e.,
it is an object in $\Mp(S\times S, L_S\times L^\op_S,\
\phi_S\times (-\phi_S))$.
\end{example}

\begin{example}\label{ex:poisson}
If $L_S$ is the graph of a Poisson structure $\pi_S$ and $J:(M,L)
\to (S,\pi_S)$ is a strong Dirac map, then the transversality
condition \eqref{eq:transv} implies that $L$ must be (the graph
of) a Poisson structure. Hence $\M(S,L_S)$ is simply the category
of Poisson maps into $S$ (whereas $\Mp(S,\pi_S)$ is the category
of Poisson maps from \textit{symplectic} manifolds into $S$).
These are the infinitesimal versions of the Hamiltonian spaces
studied by Mikami and Weinstein \cite{MiWe} in the context of
symplectic groupoid actions. For the specific choice of
$S=\gstar$, equipped with its canonical linear Poisson structure
$\pi_{\gstar}$, then $\M(S,L_S)$ (resp. $\Mp(S,L_S)$) is the
category of Poisson (resp. symplectic) Hamiltonian $\frakg$-spaces
in the classical sense.
\end{example}

\begin{example}\label{ex:quasi}
Let $G$ be a Lie group whose Lie algebra $\frakg$ carries an
$\Ad$-invariant, symmetric, nondegenerate bilinear form
$\Bi{\cdot,\cdot}$. We consider $G$ equipped with the Cartan-Dirac
structure (see e.g. \cite{ABM,BCWZ,SeWe01})
\begin{equation}\label{eq:Cartan-Dirac}
L_G:= \{( \rho(v), \sigma(v))\; |\; v\in \frakg \}\subset \T G,
\end{equation}
where $\rho(v)=v^r-v^l$ and
$\sigma(v):=\frac{1}{2}\Bi{\theta^L+\theta^R,v}$. This Dirac
structure is integrable with respect to $-\phi_G$, where $\phi_G$
is the bi-invariant Cartan 3-form, defined by
$$
\phi_G(u,v,w):=\frac{1}{2}\Bi{[u,v],w}, \;\;\; u,v,w\in \frakg.
$$
As shown in \cite{BC,BCWZ} (see also \cite{ABM}),
$\Mp(G,L_G,\phi_G)$ is the category of q-Hamiltonian
$\frakg$-spaces in the sense of Alekseev-Malkin-Meinrenken
\cite{AMM} (more general objects in $\M(G,L_G,\phi_G)$ correspond
to foliated spaces whose leaves are q-Hamiltonian
$\frakg$-spaces).
\end{example}
We will return to Example \ref{ex:quasi} in Section
\ref{sec:dirac}.

We finally observe the behavior of the Hamiltonian category under
gauge transformations. There is a natural action of $\Omega^2(S)$,
the abelian group of 2-forms on $S$, on the set of Dirac
structures on $S$: if $L_S$ is a $\phi_S$-twisted Dirac structure
and $B\in \Omega^2(S)$, we define
$$
\tau_B(L_S):=\{(Y,\beta + i_Y B)\;|\; (Y,\beta) \in L_S \} \subset
\TS,
$$
which is a $(\phi_S-dB)$-twisted Dirac structure on $S$. We refer
to $\tau_B$ as a \textbf{gauge transformation} by $B$. We use the
notation $\tau_B(S)$ for the Dirac manifold
$(S,\tau_B(L_S),\phi_S-dB)$.

\begin{proposition}\label{prop:gauge}
Let $B\in \Omega^2(S)$, and let $(M,L,\phi)$, $(S,L_S,\phi_S)$ be
Dirac manifolds. Then $J:M\to L$ is a strong Dirac map if and only
if $J:\tau_{J^*B}(M)\to \tau_B(S)$ is a strong Dirac map.
Moreover, this correspondence defines an isomorphism of
Hamiltonian categories
$$
\I_B:\M(S,L_S)\stackrel{\sim}{\longrightarrow} \M(S,\tau_B(L_S)),
$$
which restricts to an isomorphism $\Mp(S,L_S)\cong
\Mp(S,\tau_B(L_S))$.
\end{proposition}

The proof is a direct verification using the definitions (see also
\cite{BuRa}).

\begin{remark}(Global actions)\label{rem:global}

We have only defined the Hamiltonian category of a Dirac manifold
at the infinitesimal level. The global counterparts of the
$S$-valued Hamiltonian spaces in $\Mp(S,L_S,\phi_S)$ are manifolds
$M$ equipped with a 2-form $\omega_M\in\Omega^2(M)$ and carrying a
left $\grd$-action $\rho_M:\grd\times_S M\to M$ (where
$(\grd,\omega)$ is a presymplectic groupoid integrating $L_S$)
along a smooth map $J:M\to S$ such that $d\omega_M + J^*\phi_S=0$,
$\Ker(dJ)\cap \Ker(\omega_M)=\{0\}$ and
\begin{equation}\label{eq:global}
\rho_M^*\omega_M=\pr_M^*\omega_M + \pr_\grd^*\omega.
\end{equation}
Here $\pr_M, \pr_\grd$ are the natural projections from
$\grd\times_S M$ on $M$ and $\grd$. Condition \eqref{eq:global} is
the global version of $J$ being an f-Dirac map
\cite[Sec.~7]{BCWZ}. These global Hamiltonian spaces are studied
in \cite{Xu}. The global counterparts of the more general objects
in $\M(S,L_S)$ are similar, but now $M$ carries a Dirac structure
and \eqref{eq:global} holds leafwise \cite[Sec.~4.3]{BC}. In this
paper, we will be mostly concerned with the infinitesimal
Hamiltonian category (but all results have global versions that
can be obtained by standard integration procedures).
\end{remark}

\subsection{Dirac structures and equivariant cohomology}\label{subsec:equiv}

Let $(\grd,\omega)$ be a $\phi$-twisted presymplectic groupoid
integrating a Dirac manifold $(M,L,\phi)$. As recalled in Section
\ref{subsec:basics}, $\omega +\phi$ defines a 3-cocycle in the
bar-de Rham complex of $\grd$. Let us assume that $\grd=G\ltimes
M$ is an \textit{action groupoid}, relative to an action of a Lie
group $G$ on $M$ (i.e., $\sour(g,x)=x$, $\tar(g,x)=g.x$ and
$m((h,y),(g,x))=(hg,x)$). In this case, the bar-de Rham complex of
$\grd$ becomes the total complex of the double complex
$\Omega^p(G^q\times M)$, which computes the equivariant cohomology
of $M$ in the Borel model (see e.g. \cite{BSS}). In particular,
$\omega+\phi$ defines an equivariant 3-cocycle. We now discuss the
infinitesimal counterpart of this picture, which relates Dirac
structures to equivariant 3-cocycles in the Cartan model (see
\cite[Sec.~6.4]{BCWZ}).

Let $A$ be a Lie algebroid over $M$, with bracket
$[\cdot,\cdot]_A$ and anchor $\rho:A\to TM$. As proven in
\cite{BCWZ}, the infinitesimal version of a multiplicative,
relatively $\phi$-closed 2-form on a Lie groupoid is a pair
$(\sigma,\phi)$ where $\sigma: A\to T^*M$ is a bundle map,
$\phi\in \Omega^3(M)$ is closed, and such that
\begin{align}
\SP{\sigma(a),\rho(a')} &=-\SP{\sigma(a'),\rho(a)},\label{eq:IM1}\\
\sigma([a,a']_A)&= \Lie_{\rho(a)}\sigma(a')-i_{\rho(a')}d\sigma(a)
+ i_{\rho(a)\wedge\rho(a')}\phi \label{eq:IM2},
\end{align}
for all $a, a' \in \Gamma(A)$. Let us assume that the bundle map
$(\rho,\sigma):A\to \TM$ has constant rank, and let
$L:=(\rho,\sigma)(A)\subset \TM$. Then \eqref{eq:IM1} says that
$L$ is isotropic, whereas \eqref{eq:IM2} means that the space of
section $\Gamma(L)$ is closed under the Courant bracket
$\Cour{\cdot,\cdot}_\phi$. It immediately follows that if
$\mathrm{rank}(L)=\dim(M)$, then $L$ is a Dirac structure on $M$.

In this paper, we will be particularly interested in the following
special case of this construction. Suppose that a manifold $S$
carries an action of a Lie algebra $\frakg$, denoted by
$\rho:\frakg\to \X(S)$. Let $A=\frakg\ltimes S$ be the associated
action Lie algebroid, whose anchor is $\rho$ and Lie bracket on
$\Gamma(A)=C^\infty(S,\frakg)$ is uniquely defined by the bracket
on $\frakg$ (viewed as constant sections) and the Leibniz rule
(see Lemma \ref{lem:Lieds}). Assume that we are given a bundle map
$\sigma: \frakg \times S \to T^*S$ and a closed 3-form $\phi_S\in
\Omega^3(S)$ satisfying \eqref{eq:IM1} and \eqref{eq:IM2}. Let us
also suppose that
\begin{equation}\label{eq:nondegcond}
\dim(\frakg)=\dim(S), \;\;\; \mbox { and }\;\;\; \ker(\rho)\cap
\ker(\sigma)=\{0\}.
\end{equation}
The last two conditions guarantee that
$\mathrm{rank}(L_S)=\dim(S)$, hence
\begin{equation}\label{eq:LS}
L_S:=\{(\rho(v),\sigma(v))\;|\; v\in \frakg \}
\end{equation}
is a $\phi_S$-twisted Dirac structure on $S$. By construction,
$L_S$ is isomorphic to $A=\frakg\ltimes S$ as a Lie algebroid (via
$(\rho,\sigma):A\to L_S$).

\begin{proposition}\label{cor:act}
Given a strong Dirac map $J:M\to S$, we have an induced
$\frakg$-action $\rho_M:\frakg\to \X(M)$ uniquely determined by
the conditions
\[
dJ\circ \rho_M = \rho, \;\; \mbox{ and } \;\; (\rho_M(v),
J^*\sigma(v))\in L,\ \ \forall\ v\in\frakg.
\]
\end{proposition}

This is a direct consequence of Prop.~\ref{prop:act}: the action
$\rho_M$ is just the restriction of \eqref{eq:rhoh} to $\frakg$,
viewed as constant sections in $\Gamma(L_S)\cong
C^\infty(S,\frakg)$. We also have the associated
bracket-preserving map
\begin{equation}\label{eq:rhoh2}
\widehat{\rho}_M:\frakg \to \Gamma(L), \;\;
\widehat{\rho}_M(v)=(\rho_M(v),J^*\sigma(v)).
\end{equation}

We can use the action $\rho_M$ to give an alternative description
of presymplectic realizations of $S$, phrased only in terms of the
maps $\sigma:\frakg \to \Omega^1(S)$ and $\rho:\frakg\to \X(S)$,
without any explicit reference to Dirac structures:

\begin{proposition}\label{prop:explicitham}
Let $M$ be equipped with a 2-form $\omega$. Equip $S$ with the
Dirac structure $L_S$ of \eqref{eq:LS}, and let $J:M\to S$ be a
smooth map. Then $J$ is a presymplectic realization of
$(S,L_S,\phi_S)$ if and only if the following is satisfied:
\begin{enumerate}
\item[i)] $d\omega +J^*\phi_S=0$;

\item[ii)] At each $x\in M$, $\Ker(\omega)_x= \{ \rho_{M}(v)_x:
v\in \Ker(\sigma)\}$;

\item[iii)] The map $J:M\to S$ is $\frakg$-equivariant and
satisfies the moment map condition
$$
i_{\rho_M(v)}\omega= J^*\sigma(v),\ \
\forall \ v\in \ \frakg.
$$
\end{enumerate}

\end{proposition}

\begin{proof}
The only condition that remains to be checked is $ii)$, which
follows from the transversality condition \eqref{eq:transv}. The
proof is identical to the one in \cite[Thm.~7.6]{BCWZ} (c.f.
\cite[Sec.~5]{ABM}).
\end{proof}

An immediate consequence of conditions $i)$ and $iii)$ in
Prop.~\ref{prop:explicitham} is that
\begin{equation}\label{eq:omegainv}
\Lie_{\rho_M(v)} \omega = J^*(d(\sigma(v))- i_{\rho(v)}\phi_S).
\end{equation}
Hence the 2-form $\omega$ on $M$ will not be $\frakg$-invariant in
general unless $\sigma:\frakg\to \Omega^1(S)$ and
$\phi_S\in\Omega^3_{cl}(S)$ satisfy the extra condition
\begin{equation}\label{eq:eq1}
d(\sigma(v))= i_{\rho(v)}\phi_S.
\end{equation}
In this case, it immediately follows from \eqref{eq:IM2} that
\begin{equation}\label{eq:eq2}
\sigma([u,v])=\Lie_{\rho(u)}\sigma(v), \;\;\;\; v \in \frakg.
\end{equation}
Note that \eqref{eq:IM1} implies that $i_{\rho(v)}\sigma(v)=0$
and, by \eqref{eq:eq1}, $\Lie_{\rho(v)}\phi_S=0$. These conditions
together precisely say that $\sigma + \phi_S$ defines an
equivariantly closed 3-form, i.e., a 3-cocycle in the Cartan
complex
\begin{equation}\label{eq:cartan}
\Omega^k_G(S):=\left (\oplus_{2i+j=k} S^i(\gstar)\otimes
\Omega^j(S)\right )^G, \;\;\; d_G(P)=d(P(v))-i_{\rho(v)}(P(v)),
\end{equation}
where $P$ is viewed as a $G$-invariant $\Omega^\bullet(S)$-valued
polynomial on $S$, and $G$ is a connected Lie group integrating
$\frakg$. Conversely, we see that an equivariantly closed 3-form
$\sigma + \phi_S$ on $S$ satisfying \eqref{eq:nondegcond} defines
a particular type of Dirac structure $L_S$ by \eqref{eq:LS}. We
will see in this paper many concrete examples of this interplay
between Dirac structures and equivariant 3-forms.

\begin{remark}
When a Dirac structure $L_S$ is determined by an equivariantly
closed 3-form $\sigma+\phi_S$, then a gauge transformation of
$L_S$ by an invariant 2-form $B$ changes the equivariant 3-form by
an equivariant coboundary: $\sigma +\phi_S \mapsto (\sigma +
i_\rho B) + (\phi_S-dB)$.
\end{remark}

 As we will see in
Section~\ref{subsec:doubles}, one has explicit formulas for the
multiplicative 2-forms on $\grd=G\ltimes S$ arising via
integration of Dirac structures defined by \eqref{eq:LS}, and
these formulas are particularly simple when $\sigma +\phi_S$ is an
equivariant 3-form. In this case, the integration procedure for
Dirac structures gives a concrete realization of the natural map
from the cohomology of the Cartan complex \eqref{eq:cartan} into
the equivariant cohomology of $M$ in degree three.

\section{Manin pairs and isotropic connections}\label{sec:geomManin}

\subsection{Manin pairs}\label{subsec:manin}

This section recalls the basic definitions in \cite{AK} and fixes
our notation.

A \textbf{Manin pair} is a pair $(\frakd,\frakg)$, where $\frakd$
is a Lie algebra of dimension $2n$, equipped with an Ad-invariant,
nondegenerate, symmetric bilinear form $\SP{\cdot,\cdot}_{\frakd}$
of signature $(n,n)$, and $\frakg\subset \frakd$ is a Lie
subalgebra which is also a maximal isotropic subspace. (In
Appendix~\ref{subsec:app2} we discuss the more general notion of
Manin pair \textit{over a manifold $M$}.)

Throughout this paper we assume that a Manin pair
$(\frakd,\frakg)$ is integrated by a \textbf{group pair} $(D,G)$,
where $D$ is a connected  Lie group whose Lie algebra is $\frakd$,
and $G$ is a connected, closed Lie subgroup of $D$ whose Lie
algebra is $\frakg$. Given a group pair $(D,G)$, one considers the
quotient space
\[
S= D/G
\]
with respect to the $G$-action on $D$ by right multiplication. The
action of $D$ on itself by left multiplication induces an action
of $D$ on $S$, called the \textbf{dressing action}. We denote by
\begin{equation}
\label{seq-linear} \rho_S: \frakd \to \mathfrak{X}(S)
\end{equation}
the induced infinitesimal action, and by
\begin{equation}
\label{eq:rho} \rho: \frakg\rmap \mathfrak{X}(S)
\end{equation}
its restriction to $\frakg$.
The following are two key examples from \cite{AK}.

\begin{example}\label{dual-Lie}
Let $\frakg$ be a Lie algebra and consider $\frakd= \frakg\oplus
\frakg^*$, with Lie bracket given by
\[
[(u, \mu), (v, \nu)]_\frakd = ([u, v], \ad^*_{u}(\nu)-
\ad^*_{v}(\mu)),\,
\]
i.e. $\frakd=\frakg\ltimes\gstar$ is the semi-direct product Lie
algebra with respect to the coadjoint action.
If we set the pairing $\SP{\cdot,\cdot}_\frakd$ to be the
canonical one,
\begin{equation}\label{can-pair}
\SP{(u, \mu), (v, \nu)}_{\frakd}:= \SP{(u, \mu), (v,
\nu)}_{can}=\nu(u)+ \mu(v),
\end{equation}
then $(\frakd, \frakg)$ is a Manin pair. The Lie group integrating
$\frakd$ is $D= G\ltimes \frakg^*$, the semi-direct product Lie
group with respect to the coadjoint action of $G$ on $\frakg^*$.
Hence $S= \frakg^*,$ and the infinitesimal dressing action of
$\frakg$ on $S$ is the coadjoint action.
\end{example}

\begin{example}\label{group-valued}\rm \
Let $\frakg$ be a Lie algebra equipped with a symmetric,
nondegenerate, ad-invariant bilinear form $\Bi{\cdot,\cdot}$.
Consider the direct sum of Lie algebras $\frakd=\frakg \oplus
\frakg$, together with the pairing
$$
\SP{(u_1,v_1),(u_2,v_2)}_{\frakd} := \Bi{u_1,u_2}- \Bi{v_1,v_2}.
$$
We also write $\frakg\oplus \overline{\frakg}$ to denote $\frakd$
with the pairing above. If we consider $\frakg$ as a subalgebra of
$\frakd$ through the diagonal embedding $\frakg \hookrightarrow
\frakd$, $v \mapsto (v,v)$, then $(\frakd,\frakg)$ is a Manin
pair. The associated group pair is $(D=G\times G,G)$, where $G$ is
identified with the diagonal of $D$. In this case $S= (G\times
G)/G \cong G$ via the map $[(a,b)]\mapsto ab^{-1}$. Under this
identification, the dressing action of $D$ on $S$ is
$$
(a,b)\cdot g = agb^{-1},
$$
and, infinitesimally, we have
$$
\rho_S: \frakd \to TG, \;\; (u,v) \mapsto u^r-v^l,
$$
so the dressing action restricted to $G\subset D$ is the action by
conjugation.
\end{example}

\subsection{Connections and differential forms}\label{subsec:connection}

We now introduce certain differential forms on $S=D/G$ and $D$
which arise once a connection on the principal $G$-bundle (with
respect to right multiplication)
\begin{equation}\label{eq:Dbundle}
p: D \longrightarrow S
\end{equation}
is chosen. These differential forms play a central role in the
definition of $D/G$-valued Hamiltonian spaces in Section
\ref{sec:dirac}.


A principal connection on the bundle \eqref{eq:Dbundle} is called
\textbf{isotropic} if its horizontal spaces are isotropic in $TD$
(with respect to the bi-invariant pseudo-riemannanian metric
defined by $\SP{\cdot,\cdot}_\frakd$).

\begin{proposition}\label{lem:sconn}
A connection on \eqref{eq:Dbundle} is equivalent to the choice of
a 1-form
$$
s \in \Omega^1(S,\frakd)
$$
satisfying $\rho_S(s(X)) = X$, for all $X\in TS$, and the
connection is isotropic if and only if $s$ has isotropic image in
$\frakd$.
\end{proposition}

\begin{proof}
A principal connection is a $G$-equivariant bundle map $H:p^*TS
\to TD$ such that $dp\circ H=\Id$. We relate $H$ and $s$ by
trivializing $TD$ using \textit{right} translations:
$H(X,a)=dr_a(s(X))$. Since the dressing action on $S$ is
$\rho_S(u)_{p(a)}=dp(dr_a(u))$, we have $\rho_S\circ s = dp\circ
H$. The last assertion in the lemma follows from the invariance of
$\SP{\cdot,\cdot}_\frakd$.
\end{proof}

A connection on \eqref{eq:Dbundle} can also be given in terms of a
1-form $\theta \in \Omega^1(D,\frakg)$ satisfying
$$
\theta(dl_a(v))=v, \;\;\; \theta_{ag}dr_g=\Ad_{g^{-1}}\theta_a,
$$
for $a\in D$, $g\in G$ and $v\in \frakg$, and it is isotropic if
and only if
\begin{equation}\label{eq:isotropy}
\SP{\theta(X),\theta^L(Y)}_\frakd+\SP{\theta(Y),\theta^L(X)}_\frakd=\SP{X,Y}_\frakd,\;\;
X,Y\in TD.
\end{equation}
The 1-forms $\theta\in \Omega^1(D,\frakg)$ and $s\in
\Omega^1(S,\frakd)$ are related by
\begin{equation}\label{eq:stheta}
\theta_a=\theta^L_a - \Ad_{a^{-1}}(p^*s), \;\;\; a\in D.
\end{equation}

Once an isotropic connection is fixed, we have the following
induced differential forms on $S$:
\begin{equation}\label{eq:phisconn}
\phi_S\in \Omega^3(S),\;\;\; \phi_S:= \frac{1}{2}\SP{ds,s}_\frakd
+ \frac{1}{6}\SP{[s,s]_\frakd,s}_\frakd,
\end{equation}
and a $\gstar$-valued 1-form
\begin{equation}\label{eq:sigmaconn}
\sigma\in \Omega^1(S,\gstar),\;\;\;
\sigma(X)(u):=\SP{s(X),u}_\frakd,
\end{equation}
which we may alternatively view as a map $\frakg\to \Omega^1(S)$.
 We will also write $\phi_S^s$,
$\sigma_s$ if we want to stress the dependence of these forms on
the given connection $s\in \Omega^1(S,\frakd)$.

These forms satisfy many nice properties, as illustrated below.
\begin{proposition}\label{prop:properties}
The 3-form $\phi_S$ is closed, and $\sigma$ satisfies conditions
\eqref{eq:IM1} and \eqref{eq:IM2}. Moreover, viewing $S$ as a
$\frakg$-manifold with respect to the dressing action
$\rho:\frakg\to \X(S)$, conditions \eqref{eq:nondegcond} hold, and
hence $L_S=\{(\rho(u),\sigma(u))\,|\, u\in \frakg\}$ is a
$\phi_S$-twisted Dirac structure on $S$.
\end{proposition}

The proof of Prop.\ref{prop:properties} will be postponed to
Section \ref{subsec:split}.

We conclude that the choice of an isotropic connection on $p:D\to
S$ places us in the context of Section \ref{subsec:equiv}, leading
to a category of Hamiltonian spaces with $D/G$-valued moment maps.

An isotropic connection, given by $\theta\in \Omega^1(D,\frakg)$,
also induces an important 2-form $\omega_D\in \Omega^2(D)$,
\begin{equation}\label{eq:omega_D}
\omega_D:=\frac{1}{2}\left(\SP{\theta^R,\Inv^*\theta}_\frakd
-\SP{\theta^L,\theta}_\frakd\right),
\end{equation}
where $\Inv:D\to D$ denotes the inversion on the Lie group $D$.
Since $\Inv^*\theta^R=-\theta^L$, we have
$\Inv^*\omega_D=\omega_D$. Let us consider $\overline{p}:=p\circ
\Inv:D\to S$. The main property of $\omega_D$, to be proven in
Section \ref{subsec:DG1}, is that
$$
(p,\overline{p}):(D,\omega_D)\to (S\times S,L_S\times L_S)
$$
is a presymplectic realization (i.e., it is an $S\times S$-valued
Hamiltonian space).

In order to prove the various properties of the differential forms
introduced in this section, we will resort to the theory of
Courant algebroids and Dirac structures (see the Appendix).

\subsection{The Courant algebroid of a Manin pair}\label{subsec:courmanin}

In this section we recall how a Manin pair $(\frakd,\frakg)$ gives
rise to a Courant algebroid over $S=D/G$. This fact goes back to
unpublished work of \v{S}evera \cite{Se} and Alekseev-Xu
\cite{AX}.

Given a Manin pair $(\frakd,\frakg)$, let
$\frakd_{S}:=\frakd\times S $ be the trivial vector bundle over
$S$ with fiber $\frakd$. The space of sections
$\Gamma(\frakd_S)=C^\infty(S,\frakd)$ contains $\frakd$ as the
constant sections. There are several ways to extend the Lie
bracket on $\frakd$ to $\Gamma(\frakd_S)$, and each way produces a
different structure on $\frakd_S$. The simplest possibility of
extension is to use the bracket $[\cdot, \cdot]_{\frakd}$ of
$\frakd$ pointwise. This makes $\frakd_{S}$ into a \textit{bundle
of Lie algebras}. The second possibility takes into account the
infinitesimal action $\rho_S$ of $\frakd$ on $S$ and makes
$\frakd_S$ into a \textit{Lie algebroid}. This is described by the
following well-known construction.

\begin{lemma}\label{lem:Lieds} There is a unique extension of the Lie bracket of $\frakd$ to
a Lie bracket on $\Gamma(\frakd_S)$, denoted by $[\cdot,
\cdot]_{Lie}$, which makes $\frakd_S$ into a Lie algebroid over
$S$ with anchor $\rho_S$.
\end{lemma}


\begin{proof} Uniqueness follows from Leibniz identity. For the existence,
we give the explicit formula at $x\in S$:
\begin{equation}\label{eq:Liebrk}
[u, v]_{Lie}(x):= [u(x), v(x)]_{\frakd}+
\mathcal{L}_{\rho_S(u(x))}(v)(x)-
\mathcal{L}_{\rho_S(v(x))}(u)(x),
\end{equation}
for $u, v \in C^\infty(S,\frakd)$.
\end{proof}

Finally, as observed in \cite{AX,Se}, taking into account $\rho_S$
as well as the bilinear form $\SP{\cdot, \cdot}_{\frakd}$ on
$\frakd$, we can view $\frakd_S$ as a \textit{Courant algebroid}
(see Sec.~\ref{subsec:app1}). Analogously to Lemma
\ref{lem:Lieds}, we have

\begin{lemma} \label{lem:Courds} There is a unique extension of the Lie bracket of $\frakd$ to
a bilinear bracket on $\Gamma(\frakd_S)$, denoted by $\Cour{\cdot,
\cdot}_{\frakd}$, which makes $\frakd_S$ into a Courant algebroid
over $S$ with anchor $\rho_S:\frakd_S\to TS$ and symmetric pairing
$\SP{\cdot, \cdot}_{\frakd}$.
\end{lemma}

\begin{proof} As in the previous lemma, the uniqueness follows from the Leibniz identity (condition C5) in
Section~\ref{subsec:app1}). For the existence, we have the
explicit formula
\[
\SP{\Cour{u, v}_{\frakd},w}_\frakd:= \SP{[u,
v]_{Lie},w}_\frakd+\SP{\Lie_{\rho_S(w)}(u), v}_{\frakd}, \;\qquad
u, v, w \in \Gamma(\frakd_S).
\]
Note that conditions C1)-C4) in Sec.~\ref{subsec:app1} follow from
the fact that each formula is $C^{\infty}(S)$-linear in its
arguments, and they are clearly satisfied on constant sections.
\end{proof}

Let us consider the trivial bundle $\frakg_S=\frakg\times S$ over
$S$ associated with the Lie subalgebra $\frakg\subset \frakd$.

\begin{proposition}The following holds:
\begin{itemize}
\item[i)]$\frakg_{S}$ is a Dirac structure in the Courant
algebroid $\frakd_S$.

\item[ii)] The Courant algebroid $\frakd_S$ is exact, i.e., the
sequence
$$
0\longrightarrow T^*S\stackrel{\rho_S^*}{\longrightarrow} \frakd_S
\stackrel{\rho_S}{\longrightarrow} TS \longrightarrow 0
$$
is exact (see Sec.~\ref{subsec:app5}).
\end{itemize}
\end{proposition}

\begin{proof}
Using the Leibniz rule, one immediately checks that the space of
sections $\Gamma(\frakg_S)\subset \Gamma(\frakd_S)$ is closed
under any of the extensions of the Lie bracket on $\frakd$ to
$\frakd_S$. In particular, $\frakg_S$ is a Dirac structure in
$\frakd_S$.

To prove $ii)$, note that $\rho_S$ is onto. On the other hand, we
have that $\mathrm{Im}(\rho_S^*)\subseteq \mathrm{Ker}(\rho_S)$,
see \eqref{eq:rhorho*}. Since $\SP{\cdot,\cdot}_{\frakd}$ has
signature $(n,n)$, it follows that $\dim(\frakg)=n$, so $S=D/G$
has dimension $n$. The rank of $\mathrm{Ker}(\rho_S)$ is $n$,
which agrees with the rank of
$\mathrm{Im}(\rho_S^*)=\rho_S^*(T^*S)$. Hence
$\mathrm{Im}(\rho_S^*)= \mathrm{Ker}(\rho_S)$.
\end{proof}

\subsection{Invariant connections and equivariant
3-forms}\label{subsec:split}

Given a Manin pair $(\frakd,\frakg)$, let us consider its
associated Courant algebroid $\frakd_S$ as in Lemma
\ref{lem:Courds}. Let us fix an isotropic splitting $s:TS \to
\frakd_S$ of the exact sequence
$$
0\longrightarrow T^*S\stackrel{\rho_S^*}{\longrightarrow} \frakd_S
\stackrel{\rho_S}{\longrightarrow} TS \longrightarrow 0.
$$
(Isotropic splitings always exist, see Sec.~\ref{subsec:app2}.)
Since, according to Prop.~\ref{lem:sconn},
$s\in\Omega^1(S,\frakd)$ is equivalent to the choice of an
isotropic connection on the bundle $p:D\to S$, we refer to $s$ as
a \textbf{connection splitting}.

It is a general fact about exact Courant algebroids (see
Sec.~\ref{subsec:app5}) that an isotropic splitting $s$ determines
a \textit{closed} 3-form $\phi_S^s \in \Omega^3(S)$ by
\begin{equation}\label{eq:phis2}
\phi^s_S(X,Y,Z):= \SP{\Cour{s(X),s(Y)}_\frakd,s(Z)}_{\frakd}, \;\;
X,Y,Z \in \X(S).
\end{equation}
If $s$ is clear from the context, we simplify the notation by just
writing $\phi_S$ for this form.

\begin{lemma}
The 3-form $\phi_S$ in \eqref{eq:phis2} agrees with
\eqref{eq:phisconn}.
\end{lemma}

\begin{proof}
By the definition of $\Cour{\cdot,\cdot}_\frakd$ in terms of the
brackets $[\cdot,\cdot]_\frakd$ and $[\cdot,\cdot]_{Lie}$, we have
\begin{eqnarray}
\phi_S(X,Y,Z) &=& \SP{[s(X),s(Y)]_{Lie},s(Z)}_\frakd +
\SP{\Lie_Z(s(X)),s(Y)}_\frakd \label{eq:phislie}\\
 &=& \SP{[s(X),s(Y)]_\frakd
+\Lie_X(s(Y)) -\Lie_Y(s(X)),s(Z)}_\frakd +
\SP{\Lie_Z(s(X)),s(Y)}_\frakd.\nonumber
\end{eqnarray}

Using that $s$ is isotropic, we find the expression
\begin{equation}\label{eq:phis}
\phi_S(X,Y,Z)= \SP{[s(X),s(Y)]_\frakd,s(Z)}_\frakd +
\sum_{cycl}\SP{\Lie_X(s(Y)),s(Z)}_\frakd,
\end{equation}
where $\sum_{cycl}$ denotes cyclic sum in $X, Y$ and $Z$.

On the other hand, using again that $s$ is isotropic, we have
\begin{eqnarray*}
\SP{ds(X,Y),s(Z)}_\frakd&=&\SP{\Lie_X (s(Y)),
s(Z)}_\frakd-\SP{\Lie_Y(s(X)),s(Z)}_\frakd\\
&=& \SP{\Lie_X (s(Y)),s(Z)}_\frakd +
\SP{\Lie_Y(s(Z)),s(X)}_\frakd,
\end{eqnarray*}
and it follows that $\SP{ds,s}_\frakd(X,Y,Z)= 2
\sum_{cycl}\SP{\Lie_X(s(Y)),s(Z)}_\frakd$. Similarly, using that
$$
\SP{[s,s]_\frakd(X,Y),s(Z)}_\frakd=2\SP{[s(X),s(Y)]_\frakd,s(Z)},
$$
we obtain that
$\SP{[s,s]_\frakd,s}_\frakd(X,Y,Z)=6\SP{[s(X),s(Y)]_\frakd,s(Z)}$.
Now \eqref{eq:phisconn} follows from \eqref{eq:phis}.
\end{proof}

The previous lemma explains why the 3-form \eqref{eq:phisconn} is
closed. To finish the proof of Prop.~\ref{prop:properties}, note
that the connection splitting $s$ induces an identification of
Courant algebroids
\begin{equation}\label{eq:ident}
(\rho_S,s^*):\frakd_S\stackrel{\sim}{\longrightarrow} TS\oplus
T^*S,
\end{equation}
where $TS\oplus T^*S$ is equipped with the $\phi_S$-twisted
Courant bracket (see Sec.~\ref{subsec:app5}). In particular, the
image of $\frakg_S$ under \eqref{eq:ident} is a $\phi_S$-twisted
Dirac structure $L^s_S$ on $S$. Defining
\begin{equation}\label{eq:sigma}
\sigma_s:=s^*|_{\frakg}:\frakg_S \to T^*S,
\end{equation}
where $s^*:\frakd\to T^*S$ is dual to $s$ after the identification
$\frakd\cong \frakd^*$, we can write
\begin{equation}\label{eq:LS2}
L^s_S=\{(\rho(u),\sigma_s(u))\,|\, u\in \frakg\}.
\end{equation}
It is clear that the presymplectic leaves of $L_S$ are the
dressing $\frakg$-orbits. To simplify the notation, we may omit
the dependence on $s$. Note that \eqref{eq:nondegcond} holds, and
the integrability of $L_S$ implies that $\sigma$ satisfies
\eqref{eq:IM1} and \eqref{eq:IM2}, as claimed in
Prop.~\ref{prop:properties}.

We now discuss when $\sigma + \phi_S$ is an equivariantly closed
3-form with respect to the $\frakg$-action $\rho$.

Let us suppose that the connection we have fixed on $p:D\to G$ is
invariant with respect to the action of $G$ on $D$ by left
multiplication. This is equivalent to the connection splitting
$s:TS\to \frakd_S$ being $G$-equivariant, where the $G$-action on
$\frakd_S$ is given by
$$
g.(x,u)=(gx, \Ad_g(u)), \;\; g\in G,\, x\in S,\, u\in \frakd.
$$
Infinitesimally, the equivariance of $s$ becomes
\begin{equation}\label{eq:sequiv}
\Lie_{\rho(v)}(s(X)) + [v,s(X)]_{\frakd} -s([\rho(v),X])=0,\;\;\;
\forall \; X\in \mathfrak{X}(S),\; v\in \frakg.
\end{equation}

\begin{proposition}\label{prop:equivariant}
Suppose that the connection splitting $s:TS\to \frakd_S$ is
equivariant. Then $\sigma+\phi_S$ defines an equivariantly closed
3-form on $S$.
\end{proposition}

\begin{proof}
Let us first show that $\sigma$ is $\frakg$-equivariant, i.e.,
$\sigma([v,w])=\Lie_{\rho(v)}\sigma(w)$. Note that
\begin{equation}\label{eq:equiv1}
\Lie_{\rho(v)}\SP{\sigma(w),X}_\frakd=\Lie_{\rho(v)}\SP{s(X),w}_\frakd=
\SP{\Lie_{\rho(v)}(s(X)),w}_\frakd.
\end{equation}
On the other hand,
$\Lie_{\rho(v)}\SP{\sigma(w),X}_\frakd=
\SP{\Lie_{\rho(v)}\sigma(w),X}_\frakd+\SP{w,s([\rho(v),X])}_\frakd$.
Using \eqref{eq:sequiv} and the invariance of
$\SP{\cdot,\cdot}_\frakd$, we obtain
$$
\Lie_{\rho(v)}\SP{\sigma(w),X}_\frakd=\SP{\Lie_{\rho(v)}(\sigma(w)),X}_\frakd
+\SP{\Lie_{\rho(v)}(s(X)),w}_\frakd -
\SP{s(X),[v,w]_\frakd}_\frakd.
$$
Comparing with \eqref{eq:equiv1}, the equivariance of $\sigma$
follows.

Since $d\phi_S=0$, in order to check that $\sigma + \phi_S$ is an
equivariantly closed 3-form, it remains to prove that
\begin{equation}\label{eq:equivcond}
i_{\rho(v)}\sigma(v)=0,\;\;\;\mbox{ and }\;\;\;
i_{\rho(v)}\phi_S-d \sigma(v)=0.
\end{equation}
The equation on the left is a consequence of the fact that
$\frakg_S$ sits in $\frakd_S\cong TS\oplus T^*S$ as an isotropic
subbundle. For the equation on the right, first note that
\begin{eqnarray}\label{eq:dsigma}
d(\sigma(v))(X,Y)&=&\Lie_X\SP{v,s(Y)}_\frakd-\Lie_Y\SP{v,s(X)}_\frakd-\SP{v,s([X,Y])}_\frakd\nonumber\\
&=&
\SP{v,\Lie_X(s(Y))}_\frakd-\SP{v,\Lie_Y(s(X))}_\frakd-\SP{v,s([X,Y])}_\frakd.
\end{eqnarray}
Using \eqref{eq:phislie}, we have
$$
\phi_S(X,Y,\rho(v)) = \SP{[s(X),s(Y)]_{Lie}, s(\rho(v))}_\frakd +
\SP{\Lie_{\rho(v)}(s(X)),s(Y)}_\frakd.
$$
Since $s \rho_S=\id - \rho_S^*  s^*$ and $s$ is isotropic, we use
\eqref{eq:Liebrk} to write the previous expression as
\begin{align}\label{eq:express}
&\SP{\Lie_X(s(Y)),v}_\frakd - \SP{\Lie_Y (s(X)),v}_\frakd
+\SP{[s(X),s(Y)]_\frakd,v}_\frakd - \nonumber \\
&\SP{[s(X),s(Y)]_{Lie},\rho^*_S(s^*(v))}_\frakd+\SP{\Lie_{\rho(v)}(s(X)),s(Y)}_\frakd.
\end{align}

Note that pairing \eqref{eq:sequiv} with $s(Y)$ and using that $s$
is isotropic and $\SP{\cdot,\cdot}_\frakd$ is invariant, we obtain
that $\SP{\Lie_{\rho(v)}(s(X)),s(Y)}_\frakd
+\SP{[s(X),s(Y)]_\frakd,v}_\frakd=0$. On the other hand, using
that $\rho_S\circ s=\id$ and that $\rho_S:C^\infty(S,\frakd)\to
\mathfrak{X}(S)$ is a Lie algebra homomorphism with respect to
$[\cdot,\cdot]_{Lie}$, we have
$$
\SP{[s(X),s(Y)]_{Lie},\rho_S^*(s^*(v))}_\frakd =
\SP{s([X,Y]),v}_\frakd.
$$
Hence \eqref{eq:express} agrees with \eqref{eq:dsigma}, and this
concludes the proof.
\end{proof}

The previous proposition could also be derived from the discussion
in \cite[Sec.~2.2]{BCG}.

For a Manin pair $(\frakd,\frakg)$, we have a short exact sequence
associated with the inclusion $\frakg \hookrightarrow \frakd$,
\begin{equation}\label{eq:maninexact}
\frakg \longrightarrow \frakd \longrightarrow \gstar.
\end{equation}
Here the map on the right is the projection $\frakd \to
\frakd/\frakg$ after the identification $\frakd/\frakg\cong
\frakg^{*}$ induced by $\SP{\cdot, \cdot}_{\frakd}$. Let us choose
an isotropic splitting $j:\gstar \to \frakd$ of this sequence,
which amounts to the choice of an isotropic complement of $\frakg$
in $\frakd$. In general, such a splitting $j$ does not define a
connection on $p:D\to S$, but this happens under additional
assumptions (see also \cite{AX}).

\begin{proposition}\label{prop:jands}
An isotropic splitting $j:\frakg^*\to \frakd$ satisfying
$[\frakg,j(\gstar)]\subseteq j(\gstar)$ (i.e.,
$\Ad_g(j(\gstar))\subseteq j(\gstar)$) is equivalent to an
equivariant connection splitting $s:TS \to \frakd_S$.
\end{proposition}

\begin{proof}
The right action of $G$ on $D$ is generated by the vector fields
$u\mapsto u^l_a=dl_a(u)$, $a\in D$. So $dl_a(j(\gstar))$, $a\in
D$, defines a horizontal distribution on the bundle $p: D\to S$
(which is automatically invariant under the action of $G$ on $D$
by left multiplication). This distribution is invariant under the
right $G$-action on $D$ if and only if $j(\frakg^*)$ is
$\Ad(G)$-invariant. On the other hand, if a given connection is
left $G$-invariant, its horizontal distribution is left invariant,
and we get an $\Ad(G)$-invariant complement to $\frakg$ in
$\frakd$ by left translation of the horizontal distribution.
\end{proof}

Splittings $j$ with the additional invariance of
Prop.~\ref{prop:jands} may not exist in general, but they always
exist if e.g. $G$ is compact or semi-simple, see
Remark~\ref{rem:classF}.

Given a Manin pair $(\frakd,\frakg)$,
Propositions~\ref{prop:equivariant} and \ref{prop:jands} show that
the choice of an isotropic complement $\frakh$ of $\frakg$
satisfying $[\frakg,\frakh]\subseteq \frakh$ determines an
equivariantly closed 3-form $\sigma + \phi_S$ on $S$.

\section{$D/G$-valued moment maps via Dirac geometry}\label{sec:dirac}

In this section, we discuss a moment map theory associated with a
Manin pair $(\frakd,\frakg)$ based on the additional choice of an
isotropic connection on $p:D\to S=D/G$. A different moment map
theory \cite{AK}, based on the choice of splitting $j$ of
\eqref{eq:maninexact}, will be discussed in Section
\ref{sec:quasip}.

\subsection{The Hamiltonian category}\label{subsec:DG1}

Let us fix a connection splitting $s:TS\to \frakd_S$ of the exact
Courant algebroid $\frakd_S$.

The \textbf{Hamiltonian category} (or \textit{moment map theory})
associated with $(\frakd,\frakg)$ and $s$ is the Hamiltonian
category of the Dirac manifold $(S,L^s_S,\phi^s_S)$, in the sense
of Section \ref{subsec:hamcat}:
\begin{equation}\label{eq:hamcatdg1}
\M_s(\frakd, \frakg):= \M(S, L_{S}^s,\phi_S^s).
\end{equation}
We can similarly consider the subcategory of \textit{presymplectic
realizations}, in which Hamiltonian spaces carry 2-forms rather
than general Dirac structures:
\begin{equation}\label{eq:hamcatdgpre}
\Mp_s(\frakd, \frakg):= \Mp(S, L_{S}^s,\phi_S^s).
\end{equation}
>From Prop.~\ref{prop:explicitham}, we obtain an explicit
characterization of objects in $\Mp_s(\frakd, \frakg)$ only in
terms of the forms $\phi_S^s, \sigma_s$, with no reference to
Dirac structures.

We refer to objects in $\M_s(\frakd, \frakg)$ as
\textbf{$S$-valued Hamiltonian $\frakg$-spaces} (or
\textbf{$G$-spaces} if the action $\rho_M$ of Prop.~\ref{cor:act}
integrates to a $G$-action).

The reduction procedure for strong Dirac maps in
\cite[Sec.~4.4]{BC} immediately leads to:

\begin{proposition}\label{prop:diracred}
Consider a strong Dirac map $J:(M,L)\to (S,L_S)$ defining an
$S$-valued Hamiltonian $G$-space. Suppose $y\in S$ is a regular
value of $J$ and the action of the isotropy group $G_y$ on
$J^{-1}(y)$ is free and proper, and let $M_y:=J^{-1}(y)/G_y$.
Then:
\begin{enumerate}
\item[i)] The backward image of $L$ to
$J^{-1}(y)$ is a smooth Dirac structure;

\item[ii)] The quotient $M_y$ acquires a Poisson structure
uniquely characterized by the fact that the quotient map
$J^{-1}(y)\to M_y$ is f-Dirac;

\item[iii)] The reduced Poisson structure on $M_y$ is symplectic
if $J$ is a presymplectic realization.
\end{enumerate}
\end{proposition}

Given another connection splitting $s':TS\to \frakd_S$, we have
(see Sec.~\ref{subsec:app5}) an induced twist 2-form $B \in
\Omega^2(S)$, defined by
\begin{equation}\label{eq:B}
B(X,\rho_S(v))=\SPd{(s-s')(X),v},\;\; X\in TS,\; v\in \frakg.
\end{equation}
The 2-form $B$ relates $L_S^s$ and $L_S^{s'}$ by a gauge
transformation: $\tau_B(L_S^{s})=L_S^{s'}$. As an immediate
consequence of Prop.~\ref{prop:gauge}, we have
\begin{proposition}\label{prop:gauge2}
If $s$ and $s'$ are isotropic splittings and $B\in \Omega^2(S)$ is
as in \eqref{eq:B}, then the gauge transformation by $B$ defines
an isomorphism of categories
$$
\I_{B}:\M_s(\frakd,\frakg)\stackrel{\sim}{\longrightarrow}
\M_{s'}(\frakd,\frakg),
$$
which restricts to an isomorphism $\Mp_s(\frakd,\frakg)\cong
\Mp_{s'}(\frakd,\frakg)$.
\end{proposition}
It is immediate to check that this functor preserves reduced
spaces.

>From the general theory of Section \ref{subsec:hamcat}, we know
some canonical examples of Hamiltonian spaces associated with the
Manin pair $(\frakd,\frakg)$ and $s$. For example, the inclusion
$\mathcal{O}\hookrightarrow S$ of a dressing orbit is a
presymplectic realization with respect to the canonical 2-form
$\omega_{\mathcal{O}}$,
\[
\omega_{\mathcal{O}}(\rho(v), \rho(w))= \langle \sigma(v),
\rho(w)\rangle.
\]
On the other hand, presymplectic groupoids $\grd=(G\ltimes
S,\omega_\grd)$ integrating $L_S=\frakg\ltimes S$ define
presymplectic realizations $(\tar,\sour):\grd {\longrightarrow}
S\times \Sop$. These examples will be illustrated in concrete
situations in Sections \ref{subsec:examples} and
\ref{subsec:doubles}.

The remaining of this section presents a nontrivial object in
$\Mp_{s\oplus s}(\frakd\oplus\frakd,\frakg\oplus \frakg)$, i.e, an
$S\times S$-valued Hamiltonian space: the Lie group $D$, equipped
with the 2-form $\omega_D$ given by \eqref{eq:omega_D}.

\begin{theorem}\label{thm:omegaD}
Consider the right principal $G$-bundle $p:D\to S=D/G$ and let
$\overline{p}=p\circ \Inv$, where $\Inv:D\to D$ is the inversion
map. Then
$$
(p,\overline{p}):(D,\omega_D) \to (S\times S,L_S\times L_S)
$$
is a presymplectic realization, and the induced $\frakg\times
\frakg$-action on $D$ is given by $(u,v)\mapsto u^r-v^l$.
\end{theorem}

The proof will follow from three lemmas. First, we compare the
pull-back $p^*\phi_S$ with the Cartan 3-form on $D$,
$$
\phi_D:=\frac{1}{12}\SP{[\theta^R,\theta^R]_\frakd,\theta^R}_\frakd.
$$
As in Section \ref{subsec:connection}, we denote the connection
1-form on $p:D\to S$ associated with the connection splitting $s$
by $\theta \in \Omega^1(D,\frakg)$.
\begin{lemma}\label{lem:1} The following holds:
\[
 p^*(\phi_S)= -\phi_D+ \frac{1}{2} d\SPd{\theta^L, \theta}.
 \]
\end{lemma}

\begin{proof} Note that using the wedge product,
the Lie bracket on $\frakd$ defines a graded Lie bracket
$[\cdot,\cdot]_\frakd$ on $\Omega^\bullet(D,\frakd)$, while the
nondegenerate pairing on $\frakd$ defines a pairing
$\SPd{\cdot,\cdot}$ on $\Omega^\bullet(D,\frakd)$. The de Rham
differential $d$ also extends to $\Omega^\bullet(D,\frakd)$. Let
us consider the following operation:
\[
\Omega^\bullet(D, \frakd)\rmap \Omega^\bullet(D, \frakd), \;\;\;
\eta \rmap \widehat{\eta},
\]
where $ \widehat{\eta}_a(X_a)= \Ad_{a^{-1}}(\eta_a(X_a))$, $a\in
D$. One can check that
\begin{equation}\label{eq:properties}
\widehat{[\eta, \xi]_\frakd} = [ \widehat{\eta}, \widehat{\xi}
]_\frakd,\;\;\; \langle\widehat{\eta}, \widehat{\xi}\rangle_\frakd
= \SPd{\eta, \xi},\;\; \mbox{ and } \;\; \widehat{d\eta}=
d\widehat{\eta}- [\theta^L, \widehat{\eta}]_\frakd.
\end{equation}
To prove the last formula in \eqref{eq:properties}, one checks it
directly when $\eta=f$ is of degree zero, and similarly when
$\eta=df$ is exact of degree 1, and use the Leibniz identities to
conclude that the formula holds in general.

In order to simplify our notation, we will identify forms on $S$
with forms on $D$ via $p^*$, so we will often abuse notation and
denote $p^*\eta$ simply by $\eta$. With these conventions,
formula \eqref{eq:stheta} relating $\theta$ and $s$ becomes
\begin{equation}\label{eq:stheta2}
\theta= \theta^L- \widehat{s}.
\end{equation}
By \eqref{eq:phisconn} and the first two properties in
\eqref{eq:properties}, we have
$$
\phi_S = \frac{1}{2}\langle\widehat{ds},\widehat{s}\rangle_\frakd
+
\frac{1}{6}\SP{[\widehat{s},\widehat{s}]_\frakd,\widehat{s}}_\frakd
$$
Using \eqref{eq:stheta2} and the last property in
\eqref{eq:properties}, we can write
$$
\frac{1}{2}\langle{\widehat{ds},\widehat{s}}\rangle_\frakd=
\frac{1}{2}\left(\langle{d\theta^L,\theta^L-\theta}\rangle_\frakd
-\langle{d\theta,\theta^L-\theta}\rangle_\frakd
-\langle{[\theta^L,\theta^L-\theta]_\frakd,\theta^L-\theta}\rangle_\frakd\right).
$$
Note that $\SPd{d\theta,\theta}=0$ since $\theta$ takes values in
the isotropic subspace $\frakg\subset \frakd$. Using the
Maurer-Cartan equation
$d\theta^L=\frac{1}{2}[\theta^L,\theta^L]_\frakd$, we obtain
\begin{equation}\label{eq:part1}
\frac{1}{2}\langle{\widehat{ds},\widehat{s}}\rangle_\frakd=
\frac{3}{4}\langle{[\theta^L,\theta^L]_\frakd,
\theta}\rangle_\frakd
-\frac{1}{4}\langle{[\theta^L,\theta^L]_\frakd,\theta^L}\rangle_\frakd
-\frac{1}{2}\langle{d\theta,\theta^L}\rangle_\frakd
-\frac{1}{2}\langle{[\theta^L,\theta]_\frakd,\theta}\rangle_\frakd
\end{equation}
Similarly, we have
\begin{equation}\label{eq:part2}
\frac{1}{6}\SPd{[\widehat{s},\widehat{s}],\widehat{s}}=
-\frac{1}{2}\langle{[\theta^L,\theta^L]_\frakd,
\theta}\rangle_\frakd +
\frac{1}{6}\langle{[\theta^L,\theta^L]_\frakd,\theta^L}\rangle_\frakd
+\frac{1}{2}\langle{[\theta^L,\theta]_\frakd,\theta}\rangle_\frakd.
\end{equation}
Adding up \eqref{eq:part1} and \eqref{eq:part2}, we obtain
\begin{eqnarray*}
-\phi_D+\frac{1}{4}\langle{[\theta^L,\theta^L]_\frakd,
\theta}\rangle_\frakd-\frac{1}{2}\langle{d\theta,\theta^L}\rangle_\frakd&=&
-\phi_D+\frac{1}{2}\langle{d\theta^L,\theta}\rangle_\frakd-
\frac{1}{2}\langle{d\theta,\theta^L}\rangle_\frakd\\
&=&-\phi_D+\frac{1}{2}d\langle{\theta^L, \theta}\rangle_\frakd,
\end{eqnarray*}
proving the lemma.
\end{proof}

The next result relates $\omega_D$ and $\sigma$:
\begin{lemma}\label{aaa}
For all $u\in \frakg$, we have
\begin{equation}\label{eq:aaa}
i_{u^r}(\omega_D)= p^*(\sigma(u)),\;\; \mbox{ and }\;\;
-i_{u^l}(\omega_D)= \overline{p}^*(\sigma(u)),
\end{equation}
where $u^l, u^r \in \X(D)$ are the left, right invariant vector
fields determined by $u$.
\end{lemma}

\begin{proof}
As a first step, we prove the following expression for $\omega_D$:
\begin{equation}\label{omegaD}
\omega_D(X, Y)= \SPd{dl_a(\theta(X))- dr_a(\theta(\Inv(X)))- X,
Y}, \;\; \forall a\in D, \; X, Y\in T_aD.
 \end{equation}
To prove this formula, we use \eqref{eq:isotropy} to write
\[
\SPd{\theta^L, \theta}(X, Y)= \SPd{dl_{a^{-1}}(X),\theta(Y)}-
\SPd{dl_{a^{-1}}(Y),\theta(X)} = \SPd{X, Y}- 2
\SPd{dl_a\theta(X),Y}.
\]
Using that $dr_{a^{-1}}(X)= -dl_{a}(\Inv(X))$, we can use again
\eqref{eq:isotropy} to write
\begin{eqnarray*}
\SPd{\theta^R, \Inv^*\theta}(X, Y)&=& \SPd{dr_{a^{-1}}(X),
\theta(\Inv(Y))}-
\SPd{dr_{a^{-1}}(Y), \theta(\Inv(X))}\\
&=& -\SPd{\Inv(X), \Inv(Y)}+ \SPd{dl_{a}(\Inv(Y)),
\theta(\Inv(X))}-
\SPd{dr_{a^{-1}}(Y), \theta(\Inv(X))}\\
&=& -\SPd{X, Y}- 2\SPd{dr_{a^{-1}}(Y),\theta(\Inv(X))}.
\end{eqnarray*}
Comparing with the original expression \eqref{eq:omega_D} for
$\omega_D$, formula \eqref{omegaD} follows.

We now prove the first equation in \eqref{eq:aaa} (the second one
follows by applying $\Inv^*$). For $u\in\frakg$ and $X= X_a\in
T_aD$, we have (using \eqref{omegaD}) that $i_{u^r}\omega_D(X)$
equals
$$
- \omega_D(X, u^r) = \SPd{X- dl_a(\theta(X)), dr_a(u)}+
\SPd{dr_a\theta(\Inv(X)), dr_a(u)}.
$$
Looking at the r.h.s., we see that the last term vanishes since
$\frakg\subset \frakd$ is isotropic. By \eqref{eq:stheta}, we have
that $X- dl_a\theta(X)= dr_a s(dp(X))$, which gives us
$$
i_{u^r}\omega_D (X)= \SPd{s(dp(X)), u}= p^*\sigma(u)(X),
$$
as desired.
\end{proof}

Finally, we will need

\begin{lemma}\label{lem:inters}
At each $a\in D$, we have
\begin{align*}
&\Ker(\omega_D)\cap \Ker(dp)= \{dl_a(u)\,|\, u\in \frakg, \,
\theta(dr_{a}^{-1}(u))= 0\},\\
& \Ker(\omega_D)\cap \Ker(d\overline{p})= \{dr_a(v)\,|\, v\in
\frakg,\, \theta(dr_{a}v)= 0\}.
\end{align*}
\end{lemma}

\begin{proof}
We prove the second one here. We have $\Ker(d\overline{p})_a=
dr_a(\frakg)$. By Lemma~\ref{aaa} it follows that, for $u\in
\frakg$, $dr_a(u)\in \Ker(\omega_D)$ if and only if $p^*\sigma(u)=
0$, i.e.,
\[
\SPd{s((dp)(X)), u}= 0, \;\;\; \forall X\in T_aD.
\]
Using \eqref{eq:stheta}, the previous equation implies that
$$
\SPd{X, dr_{a}(u)}= \SPd{\theta(X), dl_{a}^{-1}dr_a(u)}
$$
and, using \eqref{eq:isotropy} to re-write the r.h.s of the last
equation, we get the identity
\[
\SPd{X, dr_{a}(u)}= \SPd{X, dr_{a}(u)}- \SPd{\theta(dr_a(u)),
dl_a(X)} \;\;\; \forall X\in T_aD,
\]
from which the statement follows.
\end{proof}

\begin{proof}(of Theorem \ref{thm:omegaD})
Since $\Inv^*\theta^L= -\theta^R$, $\Inv^*\phi_D= -\phi_D$, and
$\overline{p}^*= \Inv^*p^*$, Lemma \ref{lem:1} immediately implies
that
\begin{equation}
p^*\phi_S+ \overline{p}^*\phi_S  =  - d\omega_D.
\end{equation}

We now prove that $p$ is an f-Dirac map from $D$ into $S$ (and
this automatically implies that the same holds for
$\overline{p}$). It suffices to check that $L_S$ is contained in
the forward image of $L=\gra(\omega_D)$ under $p$ (since these
bundles have equal rank), i.e.,
\[
\{(\rho(u), \sigma(u))\,|\, u\in \frakg \}\, \subseteq\;  \{
(dp(X),\beta)\,|\, p^*\beta= i_X \omega_D \}.
\]
But this follows since, for $u\in \frakg$, $\rho(u)= dp(u^r)$ and,
from Lemma \ref{aaa}, $i_{u^r}(\omega_D)= p^*(\sigma(u))$.

In order to conclude the proof of the theorem, it remains to check
that
\begin{equation}\label{eq:inters}
\Ker(\omega_D)\cap \Ker(dp)\cap \Ker(d\overline{p})= 0.
\end{equation}
This is a consequence of Lemma \ref{lem:inters}: If $X\in T_aD$ is
in the triple intersection above, then
\[
X= dl_a(u)= dr_a(v), \; \textrm{ with } \; \theta(dr_{a}^{-1}(u))=
0\, \mbox{ and }\, \theta(dr_{a}(v))= 0.
\]
Since $\theta$ is a connection 1-form for the right $G$-action on
$D$, we obtain $u= \theta(dl_a(u))= \theta(dr_a(v))= 0$, and hence
$X= 0$.
\end{proof}

Since a strong Dirac map preserves the kernels of the Dirac
structures, it follows from Thm.~\ref{thm:omegaD} that $L_S$ is a
Poisson structure on $S$ if and only if $\omega_D$ is a symplectic
form on $D$ (which is only the case when $\sigma_s:\frakg_S\to
T^*S$ is an isomorphism).

Theorem \ref{thm:omegaD} has the following interesting
consequence: since $D$ carries principal $G$-actions on the left
and on the right (by left/right multiplication) which commute, the
fact that $(p,\overline{p}):D\rmap S\times S$ is a presymplectic
realization can be re-stated as saying that $(D,\omega_D)$ defines
a \textit{Morita equivalence} between the Dirac manifold $S$ and
its opposite $\Sop$ (i.e., a Morita equivalence of their
$\sour$-simply-connected presymplectic groupoids in the sense of
Xu \cite[Sec.~4]{Xu}).

\begin{proposition}\label{prop:morita}
Let $J:(M,\omega)\to S$ be a presymplectic realization defining an
$S$-valued Hamiltonian $G$-space. Then
\begin{enumerate}
\item The quotient $(D\times_{(\overline{p},J)}M)/G$ by the
diagonal $G$-action is a smooth manifold.

\item The pull-back of $\omega_D\oplus(-\omega)$ to the
submanifold $D\times_{(\overline{p},J)} M\hookrightarrow D\times M$
is basic with respect to the $G$-action. The quotient space
$(D\times_{(\overline{p},J)}M)/G$ equipped with the resulting 2-form
is denoted by $D\otimes_G M^{\mathrm{op}}$.

\item The map $D\otimes_G M^{\mathrm{op}}\to S$, $a\otimes x
\mapsto p(a)$ is a presymplectic realization making $D\otimes_G
M^{\mathrm{op}}$ into an $S$-valued Hamiltonian $G$-space.
\end{enumerate}
Moreover, this procedure defines a self-equivalence functor
$\mathcal{F}_D$ on the category of $S$-valued Hamiltonian
$G$-spaces satisfying $\mathcal{F}_D \circ \mathcal{F}_D\cong
\Id$.
\end{proposition}

\begin{proof}
Since $J:(M,\omega)\to S$ is a presymplectic realization, so is
$J:(M,-\omega) \to \Sop$. Since $S\stackrel{p}{\leftarrow}
(D,\omega_D) \stackrel{\overline{p}}{\to} \Sop$ is a Morita
bimodule, Xu's Morita theory for presymplectic groupoids
\cite[Sec.~4]{Xu} directly implies the assertions in parts $1$,
$2$ and $3$. The property $\mathcal{F}_D\circ \mathcal{F}_D\cong
\id$ follows from the fact that the inverse of the bimodule
$S\stackrel{p}{\leftarrow} (D,\omega_D)
\stackrel{\overline{p}}{\to} \Sop$ is the Morita bimodule
$\Sop\stackrel{\overline{p}}{\leftarrow} (D,-\omega_D)
\stackrel{p}{\to} S$, which is isomorphic to
$\Sop\stackrel{p}{\leftarrow} (D,-\omega_D)
\stackrel{\overline{p}}{\to} S$ via the inversion $\Inv:D\to D$.
Let us denote this last bimodule by $D^\mathrm{op}$. Then
$$
\mathcal{F}_D\circ \mathcal{F}_D(M)=D\otimes_G(D\otimes_G
M^\mathrm{op})^\mathrm{op}\cong D\otimes_G (D^\mathrm{op}\otimes_G
M)\cong (D\otimes_G D^\mathrm{op})\otimes_G M \cong M.
$$

\end{proof}

More generally, Proposition~\ref{prop:morita} holds for
Hamiltonian spaces given by strong Dirac maps, not necessarily
presymplectic realizations. The proposition shows that $D$ induces
an \textit{involution} in the category of $S$-valued Hamiltonian
$G$-spaces, which we illustrate in examples below.

\subsection{Examples}\label{subsec:examples}

We now discuss various concrete examples of $D/G$-valued moment
maps arising from specific choices of Manin pairs
$(\frakd,\frakg)$ and connection splittings $s:TS\to \frakd_S$.

\begin{example}[$\gstar$-valued moment maps]\label{ex:g*}
Let us consider the Manin pair $(\frakg\ltimes \frakg^*,\frakg)$
of Example \ref{dual-Lie}. In this case $S=\gstar$, and we have a
canonical equivariant connection splitting $s:T\gstar\to
(\frakg\oplus\gstar)\times \gstar$ given by
$$
s(\mu_x)=((0,\mu),x).
$$
Then $\sigma=\mathrm{Id}:\frakg \times \gstar \to T^*\gstar\cong
\frakg\times \gstar$, and $(L_S)_\mu=\{(\mathrm{ad}^*_u(\mu),u)\;|\;
u\in \frakg\}$, $\mu \in \gstar$, is just the graph of the
Lie-Poisson structure on $\frakg^*$. As we saw in Example
\ref{ex:poisson}, $\M_s(\frakg\ltimes \frakg^*,\frakg)$ is simply
the category of Poisson maps into $\frakg^*$, i.e., classical
Hamiltonian $\frakg$-spaces.
\end{example}

More generally, one can consider Manin pairs coming from Lie
bialgebras (see e.g. \cite{Luthesis}).

\begin{example}[$G^*$-valued moment maps]\label{ex:G*}
Let $(\frakg,\gstar)$ be a Lie bialgebra and $\frakd$ be its
Drinfeld double (see Section \ref{subsec:quasip}). We consider the
Manin pair $(\frakd,\frakg)$, assuming that an extra
\textit{completeness} condition \cite[Sec.~2.5]{Luthesis} holds,
as we now recall.

Let $D$ be the simply-connected Lie group integrating $\frakd$,
and $G$ and $G^*$ be the simply-connected Lie groups integrating
$\frakg$ and $\frakg^*$, respectively. The inclusion of $\frakg$
and $\gstar$ into $\frakd$ integrate to Lie group homomorphisms
$i_1: G\to D$ and $i_2:G^*\to D$, and we obtain a local
diffeomorphism
\begin{equation}\label{eq:localdiff}
G\times G^*\to D,\;\;\; (g,x)\mapsto i_1(g)i_2(x).
\end{equation}
We further assume that this map is a \textit{global
diffeomorphism}. To simplify the notation, we identify $G$ and
$G^*$ with their images in $D$ under the maps $i_1$ and $i_2$.

Any element in $ D$ can be written as $gx$ or as $x'g'$, for
unique $g,g'\in G, x,x'\in G^*$. In this case, let us write
$x'=\varphi_g(x)$. The map $\varphi: G\times G^*\to G^*$,
$(g,x)\mapsto \varphi_g(x)$ defines a left action of $G$ on $G^*$,
and it induces a diffeomorphism
$$
S=D/G \stackrel{\sim}{\rmap} G^*,\;\;\, [(g,x)]\mapsto
\varphi_g(x).
$$
Under this identification, the action of $D$ on itself by left
multiplication induces left actions of $G^*$ and $G$ on $S=G^*$:
The $G^*$-action is by left multiplication, whereas the $G$-action
(i.e., the dressing action) is $\varphi$. In particular, we have
$$
\rho_S|_{\gstar}: \gstar\to TG^*, \;\; \mu\mapsto \mu^r.
$$
It follows that there is a canonical choice of connection
splitting by
\begin{equation}\label{eq:sG*}
s:TG^*\to \frakd_S,\;\; V_x \mapsto \theta^R_{G^*}(V_x)
=dr_{x}^{-1}(V_x).
\end{equation}
Here $r_x$ denotes the right multiplication by $x$ on the Lie
group $G^*$. The induced map $\sigma_s:\frakg_S \to T^*G^*$ is
given by
\begin{equation}\label{eq:sigmaG*}
\sigma_s(v)=(dr_{x}^{-1})^*v.
\end{equation}
Note that $s(TG^*)=\gstar_S\subset \frakd_S$ is transversal to
$\frakg_S$, hence the kernel of $L^s_S$ is trivial. It follows
that the Dirac structure $L^s_S$ on $G^*$ is the graph of a
Poisson structure $\pi_{G^*}$ ($\phi_S=0$ since $\gstar\subset
\frakd$ is a subalgebra) defined by
\begin{equation}\label{eq:piG*}
\pi_{G^*}^\sharp((dr_{x}^{-1})^*v)=\rho(v), \;\; v\in \frakg.
\end{equation}
This Poisson structure makes $G^*$ into the Poisson-Lie group dual
to the one integrating $(\frakg,\gstar)$. Since $L_S$ is Poisson,
the 2-form $\omega_D$ on $D$ is symplectic, and $(D,\omega_D)$ is
the Heisenberg double (c.f. Example \ref{ex:Heisgrd}). The
Hamiltonian category $\M_s(\frakd,\frakg)$ in this example is the
category of Poisson maps into $G^*$, which are the Hamiltonian
Poisson $\frakg$-spaces in the sense of \cite{Lu}. In particular,
when the Hamiltonian space is symplectic, condition $iii)$ in
Prop.~\ref{prop:explicitham} becomes Lu's moment map condition
$$
i_{\rho_M(v)}\omega = J^*\SP{\theta^R_{G^*},v}.
$$
With the identifications $D\cong G\times G^*$ and $S\cong G^*$ (as
manifolds), the maps $p,\overline{p} : D\to G^*$ become
$$
p(g,x)= \varphi_g(x),\;\,\mbox{ and }\;\;
\overline{p}(g,x)=x^{-1}.
$$
A direct calculation shows that the  involution $\mathcal{F}_D$ of
Prop.~\ref{prop:morita} takes a Poisson map $J:M\to G^*$ to
$\Inv_{G^*}\circ J:M^\mathrm{op}\to G^*$ (where $\Inv_{G^*}$
denotes the inversion map in $G^*$).

Note that the connection \eqref{eq:sG*} is not equivariant in
general, so it does \textit{not} define an equivariant 3-form (we
will return to this issue in Section~\ref{subsec:doubles}).
\end{example}

Although a connection splitting $s$ exists even without the extra
completeness assumption made in Example \ref{ex:G*}, in general
$D/G$ will not be identified with $G^*$ and the choice of $s$ is
not canonical.

For a special class of Lie bialgebras, there is a different choice
of connection splitting which is equivariant and leads to a gauge
equivalent (in the sense of Prop.~\ref{prop:gauge2}) Hamiltonian
category:

\begin{example}[$P$-valued moment maps]\label{ex:P}
Let $G$ be a connected, simply-connected compact Lie group. We fix
an Ad-invariant, nondegenerate, symmetric bilinear form on
$\frakg$, and denote by $\Bic{\cdot,\cdot}$ the induced
complex-bilinear form on $\gc$. We view $(\frakd=\gc,\frakg)$ as a
Manin pair with respect to the pairing
$\SPd{\cdot,\cdot}=\mathfrak{Im}\Bic{\cdot,\cdot}$, given by the
imaginary part of $\Bic{\cdot,\cdot}$. By the Iwasawa
decomposition, we can write
$$
\gc=\frakg\oplus \mathfrak{a}\oplus \mathfrak{n},
$$
where $\mathfrak{a}=\sqrt{-1}\mathfrak{t}$ ($\mathfrak{t}$ is the
Lie algebra of the maximal torus $T\subset G$) and $\mathfrak{n}$
is the sum of positive root spaces. Then $\mathfrak{a}\oplus
\mathfrak{n}$ is an isotropic complement of $\frakg$ in $\frakd$.
Since $\gstar\cong \mathfrak{a}\oplus \mathfrak{n}\subset \frakd$
is a subalgebra, this defines a Lie bialgebra and we are in the
context of Example \ref{ex:G*}. At the global level, we have the
decomposition $D=\Gc=GAN$, and $G^*\cong AN$.

In the present situation, however, one has another choice of
isotropic complement of $\frakg\subset \frakd$, namely $\frakh:=
\sqrt{-1}\,\frakg$. Note that $\frakh$ is not a subalgebra, but it
satisfies $[\frakg,\frakh]\subseteq \frakh$. From
Prop.~\ref{prop:jands}, we have an induced connection splitting
which is \textit{equivariant} (hence distinct from
\eqref{eq:sG*}). In order to get a simple explicit formula for the
connection, we follow \cite[Sec.~10]{AMM} and choose a different
realization of $\Gc/G$.

Let $^c:\Gc\to \Gc$ be the involution given by exponentiating the
complex conjugation $v\mapsto \overline{v}$ on $\gc$, and consider
the map $\Gc\to \Gc$ given by $g\mapsto g^\dagger:= (g^{-1})^c$.
Let
$$
P:= \{ a\in \Gc\;|\; a=a^\dagger\}.
$$
Then the map $q:\Gc\to P$, $a\mapsto a a^\dagger$, induces a
diffeomorphism from $G^*=D/G$ to $P$, identifying the dressing
$G$-action on $G^*$ with conjugation on $P$ by $G$. Using that
$$
dq_a=
dl_{aa^\dagger}(\Ad_{a^c}(\theta^L_{\Gc}-\overline{\theta^L_{\Gc}})),
$$
and $\rho_S:\frakd_S\to TP$,
$\rho_S(u)_a=dq(dr_a(u))=dr_{aa^\dagger}(u)-dl_{aa^\dagger}(u)$,
one can find the explicit expression for the equivariant
connection splitting induced by $\frakh=\sqrt{-1}\, \frakg$:
\begin{equation}\label{eq:s'}
s':TP\to \frakd_S,\;\;\; X\mapsto \frac{1}{2}\theta^R_{P}(X),
\end{equation}
where $\theta_P^R$ is the pull-back of $\theta_{\Gc}^R$ to
$P\hookrightarrow \Gc$. The equivariant 3-form
$\sigma_{s'}+\phi_S^{s'}$ is given by
$$
\sigma_{s'}(u)= \frac{1}{2}\SP{\theta^R_P,u}_\frakd=\frac{1}{2}
\left(\frac{1}{2\sqrt{-1}}\Bic{u,\theta_P^R-\overline{\theta_P^R}}
\right)=
\frac{1}{2}\left(\frac{1}{2\sqrt{-1}}\Bic{u,\theta_P^R+{\theta_P^L}}
\right),
$$
for $u\in \frakg$, and, using \eqref{eq:phisconn} and the
Maurer-Cartan equation for $\theta^R_{\Gc}$, we get
$$
\phi_S^{s'}= -\frac{1}{2}\left( \frac{1}{12}
\mathfrak{Im}\Bic{[\theta_P^R,\theta^R_P],\theta^R_P}\right).
$$
The description of Hamiltonian spaces in
Prop.~\ref{prop:explicitham} reproduces the original definition of
$P$-valued moment maps in \cite{AMM} (up to a factor of 2). By
identifying $G^*$ and $P$ (via $q$), we get two different
connections splittings \eqref{eq:sG*} and \eqref{eq:s'} for
$D=\Gc\to G^*$. By Prop.~\ref{prop:gauge2}, the associated
Hamiltonian categories are isomorphic by a gauge transformation,
as explicitly shown in \cite[Sec.~10.3]{AMM}.
\end{example}

\begin{example}[$G$-valued moment maps]\label{ex:G}
Consider the Manin pair $(\frakg\oplus \frakg,\frakg)$ of
Example~\ref{group-valued}, where $\frakg$ sits in $\frakg\oplus
\frakg$ diagonally. The infinitesimal dressing action is
$$
\rho_S: (\frakg\oplus\frakg)\times G \to TG,\;\;\; (u,v)\mapsto
u^r-v^l.
$$
The antidiagonal in $\frakg\oplus \frakg$ gives an
ad($\frakg$)-invariant isotropic complement of $\frakg$, hence it
defines an equivariant connection splitting $s:TG \to \frakd_S$,
explicitly given by
$$
s(X_g):=\left(\frac{1}{2}dr_g^{-1}(X_g),-\frac{1}{2}dl_g^{-1}(X_g)\right).
$$
The associated equivariant 3-form $\sigma^s+\phi_S^s$ is defined
by
$$
\sigma_s(u)=\frac{1}{2}\Bi{\theta^R+\theta^L,u}, \;\; u\in \frakg,
$$
and the Cartan 3-form
$\phi_S^s=-\phi_G:=-\frac{1}{12}\Bi{[\theta^R,\theta^R],\theta^R}$
(using \eqref{eq:phisconn}). Note that $L_S^s$ is exactly the
Cartan-Dirac structure \eqref{eq:Cartan-Dirac}, and the conditions
in Prop.~\ref{prop:explicitham} reproduce the defining axioms of
quasi-Hamiltonian spaces in \cite{AMM}.

In this example, $(D=G\times G,\omega_D)$, with $p(a,b)=a b^{-1}$
and $\overline{p}(a,b)= a^{-1}b$, is easily seen to be isomorphic
to the AMM-double of \cite[Sec.~3.2]{AMM} (under $(a,b)\mapsto
(a,b^{-1})$). The involution $\mathcal{F}_D$ sends a
quasi-Hamiltonian space $(M,\omega,\rho_M,J)$ to
$(M,-\omega,\rho_M,\Inv_G\circ J)$ (c.f. \cite[Prop.~4.4]{AMM}).
\end{example}

\begin{example}[Symmetric-space valued moment maps]
The \textit{symmetric-space valued moment maps} of \cite{Lei}
naturally fit into the Dirac geometric framework of $D/G$-valued
moment maps. The Manin symmetric Lie algebras in
\cite[Sec.~2]{Lei} are examples of Manin pairs $(\frakd,\frakg)$
equipped with an ad($\frakg$)-invariant isotropic complement of
$\frakg$, which by Prop.~\ref{prop:jands} define
\textit{equivariant} connection splittings. The associated
equivariant 3-forms given by Prop.~\ref{prop:equivariant} agree
with the \textit{moment forms} of \cite[Sec.~3]{Lei}, and the
corresponding \textit{moment spaces} are exactly the objects in
$\Mp_s(\frakd,\frakg)$.
\end{example}

We will explain how all these moment map theories are related to
quasi-Poisson geometry in Section \ref{sec:equiv}.

\subsection{Presymplectic groupoids and doubles}\label{subsec:doubles}

Let $(\frakd,\frakg)$ be a Manin pair with the choice of a
connection splitting $s:TS\to \frakd_S$, and let
$$
L_S=\{(\rho(u),\sigma(u)),\; u\in \frakg\}
$$
be the associated $\phi_S$-twisted Dirac structure on $S=D/G$. In
this section, we will discuss the integration of the Dirac
manifold $(S,L_S,\phi_S)$.

As remarked in Section~\ref{subsec:equiv}, $L_S$ is isomorphic, as
a Lie algebroid, to the action algebroid $\frakg \ltimes S$ with
respect to the dressing action. Hence, as integration of $L_S$, we
can use the action groupoid $\grd= G\ltimes S$. The source and target maps are given by
$\sour(g,x)=x$, $\tar(g,x)=g.x$, and the multiplication is
$m((g,x),(h,y))=(gh,y)$. It remains to describe the 2-form $\omega
\in \Omega^2(\grd)$ making it into a $\phi_S$-twisted
presymplectic groupoid integrating $L_S$.

Consider  $\lambda \in C^\infty(\frakg,\Omega^2(S))$ given by
(c.f. Section~\ref{subsec:equiv})
\begin{equation}\label{eq:cocyc}
\lambda(v)= d\sigma(v)-i_{\rho(v)}\phi_S.
\end{equation}
Using that $\sigma$ satisfies \eqref{eq:IM2},
$$
\sigma([u,v])=\Lie_{\rho(u)}\sigma(v) - i_{\rho(v)}d\sigma(u) +
i_{\rho(u)\wedge\rho(v)}\phi_S, \;\; u,v \in \frakg,
$$
it is simple to check that $\lambda$ satisfies
$$
\lambda([u,v])=\Lie_{\rho(u)}\lambda(v)-\Lie_{\rho(v)}\lambda(u),
$$
i.e., it is a $\Omega^2(S)$-valued Lie algebra cocycle. That is,
the map $\frakg\to \frakg\ltimes \Omega^2(S)$, $v\mapsto
(v,\lambda(v))$ is a Lie algebra homomorphism.

Assume now that the cocycle
\eqref{eq:cocyc} integrates to a Lie group cocycle, i.e. $c\in
C^\infty(G,\Omega^2(S))$  satisfying
$$
c(gh)= h^*c(g)+ c(h),
$$
where the pull-back $h^*c(g)$ is with respect to the dressing
action of $G$ on $S$. This happens e.g. if $G$ is simply-connected.
It follows from Prop.~\ref{prop:equivariant}
that if $s$ is equivariant, then \eqref{eq:cocyc} vanishes and
$c\equiv 0$.

The next result follows from \cite[Sec.~6.4]{BCWZ} and gives an
explicit formula for the multiplicative 2-form integrating $L_S$:

\begin{proposition}\label{prop:2form}
The 2-form $\omega\in \Omega^2(G\ltimes S)$ integrating $L_S$ is
explicitly given by
\begin{eqnarray}
\omega_{g,x}((V,X),(V',X'))&:=& \SP{\sigma_x(\theta^L_g(V)),
\rho_x(\theta^L_g(V'))} + \SP{\sigma_x(\theta^L_g(V)),X'}-
\SP{\sigma_x(\theta^L_g(V')),X} \nonumber\\
&&  + \SP{c(g),X\wedge X'}, \label{eq:2form}
\end{eqnarray}
where $V,V'\in T_gG,\;\; X,X'\in T_xS$ and $\theta^L\in
\Omega^1(G,\frakg)$ is the left-invariant Maurer-Cartan 1-form. If
$s$ is equivariant, then $c=0$.
\end{proposition}

As mentioned in Example \ref{ex:grd}, this 2-form $\omega$ on
$G\times S$ makes it into a $S\times \Sop$-valued Hamiltonian
$G\times G$-space. This space is closely related to various
well-known notions of ``double''.

\begin{example}[Cotangent bundles]
Let us consider the Manin pair $(\frakg,\frakg\ltimes \gstar)$
with connection $s$ as in Example \ref{ex:g*}. Since the
connection is invariant, the cocycle $c$ vanishes and the 2-form
on $G\ltimes \gstar$ of Prop.~\ref{eq:2form} reads
\begin{eqnarray*}
\omega_{g,\mu}((V,X),(V',X'))&=&\SP{dl_g^{-1}(V),
\ad^*_{dl_g^{-1}(V')}(\mu)}
+ \SP{dl_g^{-1}(V),X'}- \SP{dl_g^{-1}(V'),X}\\
&=& \mu([\theta^L(V),\theta^L(V')]) + \SP{\theta^L(V),X'}-
\SP{\theta^L(V'),X},
\end{eqnarray*}
which is the canonical symplectic structure on $T^*G \cong G\times
\gstar$ (with identification via left translations).
\end{example}

\begin{example}[AMM groupoid]\label{ex:AMMgrd}
For the Manin pair $(\frakg,\frakg\oplus \frakg)$ with invariant
connection $s$ of Example \ref{ex:G}, $c=0$ and the 2-form
$\omega$ on $G\ltimes G$ can be directly computed to be
\begin{eqnarray*}
\omega_{g,x}((V,X),(V',X'))&=& \frac{1}{2} \big [
\Bi{\Ad_x\theta^L(V),\theta^L(V')}-\Bi{\Ad_x\theta^L(V'),\theta^L(V)} \\
&& +\Bi{\theta^L(V),{\theta^R}(X')}-\Bi{\theta^L(V'),\theta^R(X)}\\
&&
+\Bi{\theta^L(V),\theta^L(X')}-\Bi{\theta^L(V'),\theta^L(X)}\big
],
\end{eqnarray*}
which can be re-written more concisely as
$$
\omega_{g,x}=\frac{1}{2}\left(\Bi{\Ad_xp_1^*\theta^L,p_1^*\theta^L}+
\Bi{p_1^*\theta^L,p_2^*(\theta^L+ \theta^R)}\right ),
$$
where $p_1(g,x)=g$ and $p_2(g,x)=x$ are the natural projections
$G\times G \to G$.

The presymplectic groupoid $(\grd=G\ltimes G,\omega)$ is closely
related to the double $(D=G\times G, \omega_D)$ of
Example~\ref{ex:G}: the change of coordinates $D \to \grd$,
$(a,b)\mapsto (g=a,x=b^{-1}a)$ identifies $\omega_D$ with
$\omega$. (Under this identification, we have $p=\tar$ and
$\overline{p}=\sour^{-1}$, so $D$ and $\grd$ are not identified as
bimodules.)
\end{example}

\begin{example}[Heisenberg groupoid]\label{ex:Heisgrd}
Let us consider the case of a Lie bialgebra $(\frakg,\gstar)$ as
in Example \ref{ex:G*}, where $S=G^*$ and the connection splitting
$s$ is \eqref{eq:sG*}. This connection is not equivariant in
general, so one has to consider the 1-cocycle of \eqref{eq:cocyc}:
$\lambda(u)=d\sigma(u)$ (in this example, $\phi_S=0$). Explicitly,
we have
\begin{eqnarray*}
d\sigma(u)(\mu^r,\nu^r)&=&\Lie_{\mu^r}\SP{(dr_x^{-1})^*u,\nu^r}-\Lie_{\nu^r}\SP{(dr_x^{-1})^*u,\mu^r}-
(dr_x^{-1})^*u([\mu^r,\nu^r])\\
&=&-F(v)(\mu,\nu).
\end{eqnarray*}
Here $\mu^r,\nu^r$ are the right translations of $\mu,\nu\in
\frakg^*$ to $TG^*$, and $F:\frakg\to \frakg\wedge \frakg$ is the
co-bracket (c.f. Section \ref{subsec:quasip}). The cocycle
$\lambda\in C^\infty(\frakg,\Omega^2(G^*))$ is then given by
$$
\lambda(u)_x=-(dr_x^{-1})^*F(u) \in T_x^*G^*.
$$
We know that $F:\frakg\to \frakg\wedge \frakg$ is a 1-cocycle with
respect to the adjoint representation, and there is a unique
multiplicative bivector field $\pi_G$ on $G$ so that
$dl_g^{-1}\pi_G:G\to \frakg\wedge \frakg$ is a 1-cocycle
integrating $F$, see e.g. \cite{KoSo}. Hence the following
1-cocycle $c\in C^\infty(G,\Omega^2(G^*))$ integrates $\lambda$:
$$
c(g)_x=-(dr_x^{-1})^*dl_g^{-1}\pi_G,
$$
and we get the following expression for the 2-form on $G\ltimes
G^*$ integrating $\pi_{G^*}$:
\begin{eqnarray*}
\omega_{g,x}((V,X),(V',X'))&=&
-\SP{dr_x^{-1}\pi_{G^*},\theta^L_g(V)\wedge\theta^L_g(V')}
+ \SP{\theta^L_g(V),\theta^R_x(X')} \nonumber\\
&& - \SP{\theta^L_g(V'),\theta^R_x(X)}  - \SP{dl_g^{-1}\pi_G,
\theta^R_x(X)\wedge \theta^R_x(X')},
\end{eqnarray*}
where we have used that $\SP{(dr_x^{-1})^*(\theta^L_g(V)),
\rho_x(\theta^L_g(V'))}=-\SP{dr_x^{-1}\pi_{G^*},\theta^L_g(V)\wedge\theta^L_g(V')}$.

The 2-form $\omega$ agrees with $\omega_D$ in this example, so, as
a symplectic manifold, the groupoid $(G\ltimes G^*,\omega)$ is the
Heisenberg double $(D=G\times G^*,\omega_D)$.

\end{example}


\section{$D/G$-valued moment maps via quasi-Poisson geometry}\label{sec:quasip}

In this section we revisit the theory of $D/G$-valued moment maps
in the context of quasi-Poisson actions following \cite{AK}. This
theory has as starting point a Manin pair $(\frakd,\frakg)$
together with the choice of an isotropic complement of $\frakg$ in
$\frakd$ or, equivalently, an isotropic splitting $j$ of
\eqref{eq:maninexact}. We refer to $(\frakd,\frakg,j)$ as a
\textbf{split Manin pair}. We now observe that quasi-Poisson
spaces, just as ordinary Poisson manifolds, can be understood in
terms of Lie algebroids. This leads to refinements of results in
\cite{AK}.

\subsection{Quasi-Poisson $\frakg$-spaces}\label{subsec:quasip}

Let $(\frakd,\frakg)$ be a Manin pair. Following
Sections~\ref{subsec:app2} and \ref{subsec:app3}, we consider the
exact sequence
\begin{equation} \label{ex-seq}
\frakg\stackrel{\iota}{\rmap} \frakd\stackrel{\iota^*}{\rmap}
\frakg^{*},
\end{equation}
where $\iota:\frakg \hookrightarrow \frakd$ is the inclusion and
$\iota^*$ is the projection $\frakd \to \frakd/\frakg$ after the
identification $\frakd/\frakg= \frakg^{*}$ induced by $\SP{\cdot,
\cdot}_{\frakd}$. The choice of an isotropic splitting
$j:\frakg^*\to \frakd$ 
defines
elements
\begin{equation}\label{eq:Fchi}
\cobr_j: \frakg\rmap \wedge^2\frakg, \;\mbox{ and }\;\;  \tri_j
\in \wedge^3\frakg
\end{equation}
by the conditions
\[ \cobr_j^*(\mu, \nu)= \iota^*([j(\mu), j(\nu)]_{\frakd}),\qquad
\tri_j(\mu, \nu)= j^{*}([j(\mu), j(\nu)]_{\frakd}),\;\;\; \mu,\nu
\in \gstar,
\]
see Sec.~\ref{subsec:app2}. We will omit the subscript $j$
whenever there is no risk of confusion. Using the isometric
identification of $(\frakd,\SP{\cdot,\cdot}_{\frakd})$ with
$(\frakg\oplus \frakg^*,\SP{\cdot,\cdot}_{can})$ given by
$(\iota,j)$, the Lie bracket on $\frakd$ takes the form:
\begin{align}
&[(u,0),(v,0)]_{\frakd}=([u,v],0),\label{eq:doub1}\\
&[(v,0),(0,\mu)]_{\frakd}=(i_{\mu}(F(v)),\ad^*_v \mu),\label{eq:doub2}\\
&[(0,\mu),(0,\nu)]_{\frakd}=(\tri(\mu,\nu),\cobr^*(\mu,\nu)),\label{eq:doub3}
\end{align}
for $u,v \in \frakg$ and $\mu,\nu \in \gstar$.

As in Sec.~\ref{subsec:app3}, we say that a triple
$(\frakg,\cobr,\tri)$ is a {\bf Lie quasi-bialgebra}
\cite{BK-S,Drinf} if the bracket defined by \eqref{eq:doub1},
\eqref{eq:doub2}, \eqref{eq:doub3} is a Lie bracket, in which case
$(\frakg\oplus \frakg^*,\frakg)$ is a split Manin pair. The
resulting Lie algebra $\frakg\oplus \frakg^*$ is the
\textbf{Drinfeld double} \cite{BK-S} of the Lie quasi-bialgebra
$(\frakg,\cobr,\tri)$. A Lie quasi-bialgebra with $\tri= 0$ is
called a \textbf{Lie bialgebra}, in which case $\cobr$ defines a
Lie algebra structure on $\frakg^*$. Note that there is a 1-1
correspondence between split Manin pairs $(\frakd,\frakg,j)$ and
Lie quasi-bialgebras.

For a fixed isotropic splitting $j:\frakg^*\to \frakd$ for the
Manin pair $(\frakd,\frakg)$, and denoting by
$(\frakg,\cobr,\tri)$ the corresponding Lie quasi-bialgebra, we
define a \textbf{quasi-Poisson $\frakg$-space} \cite{AK} as a
manifold $M$ endowed with an infinitesimal action of $\frakg$,
denoted by $\rho_M: \frakg \to \mathfrak{X}(M)$, and a bivector
field $\pi \in \X^2(M)$, such that
\begin{align}
&\Lie_{\rho_M(v)}\pi = -\rho_M(\cobr(v)),\;\;\mbox{ for all }
v \in \frakg,\label{cond:q-poiss2}\\
&\frac{1}{2}[\pi,\pi]= \rho_M(\tri),\label{cond:q-poiss1}
\end{align}
where $[\cdot,\cdot]$ is the Schouten bracket. If $\cobr=0$ and
$\tri=0$, then $(M,\pi)$ is a Poisson manifold, and $\pi$ is
invariant. In general $\cobr$ controls how $\pi$ fails to be
invariant, whereas $\tri$ controls how it fails to be integrable.

\begin{remark}\label{rem:twist}
A different isotropic splitting $j'$, related to $j$ by a twist $t
\in \wedge^2 \frakg$ (i.e., $j-j'=t^\sharp$) leads, as discussed
in Sec.~\ref{subsec:app5} (see \cite{AK}), to a Lie
quasi-bialgebra defined by
\begin{equation}\label{eq:cotri}
F' = F + [t,\cdot],\;\;\; \chi'=\chi - d_F(t)+ \frac{1}{2}[t,t].
\end{equation}
If $(M,\pi)$ is a quasi-Poisson space for $(\frakg,F,\tri)$, then
$(M,\pi')$ is a quasi-Poisson space for for the Lie
quasi-bialgebra $(\frakg,F',\tri')$ \cite{AK}, where
\begin{equation}
\pi'=\pi + \rho_M(t).
\end{equation}
\end{remark}

\begin{remark}\label{rem:classF}
The cobracket $\cobr$ associated with $j$ is a cocycle with values
in $\wedge^2\frakg$ (see Q0) in Sec.~\ref{subsec:app3}), hence
determines a class in $H^1(\frakg,\wedge^2\frakg)$. It follows
from the first formula in \eqref{eq:cotri} that this class does
not depend on the splitting $j$. By \eqref{eq:doub2}, $j$ has the
property $[\frakg,j(\gstar)]\subseteq j(\gstar)$ if and only if
$\cobr=0$, and such $j$ exists if $H^1(\frakg,\wedge^2\frakg)=0$.
\end{remark}

\begin{remark}{\it (Global actions)}
One can similarly consider global quasi-Poisson $G$-actions, in
which case the Lie quasi-bialgebra is replaced by its global
counterpart, i.e., a \textbf{quasi-Poisson Lie group} \cite{K-S}
(which arise in connection with \cite{Drinf}).
Since our main constructions do not require global actions, we
will restrict ourselves to the infinitesimal picture.
\end{remark}

As shown in \cite{AK}, the choice of an isotropic splitting $j$ of
\eqref{ex-seq} induces a bivector field $\pi_S$ on $S=D/G$, which
makes $S$ into a quasi-Poisson space with respect to the dressing
$\frakg$-action. To define $\pi_S$, note that
$(\frakd_S,\frakg_S)$ is a Manin pair over the manifold $S$
(Sec.~\ref{subsec:app2}), and $j$ induces a splitting of it. As in
Sec.~\ref{subsec:app4}, we have an associated bivector field
$\pi_S$ defined by
\begin{equation}\label{eq:piS}
\pi_S^\sharp=\rho(\rho_Sj)^*: T^*S\rmap TS,
\end{equation}
where $\rho=\rho_S\iota$. If $j'$ is another isotropic splitting
and $t$ is the associated twist, then $ \pi_S'=\pi_S+\rho_S(t)$.
Let us consider the map
\begin{equation}\label{eq:sigmav}
\sigmav_j= (\rho_Sj)^*: T^*S\rmap \frakg_S.
\end{equation}
(As usual, if there is no danger of confusion, we will drop the
dependence on $j$ in the notation and write simply $\sigmav$.)


\begin{remark}\label{rem:rmatrix}
To see that $\pi_S$ agrees with the bivector defined in \cite{AK},
consider the \textbf{$r$-matrix} $\rma=\rma_j \in
\frakd\otimes\frakd$, given by $\rma(u^\vee, v^\vee)=
\SP{j\iota^*(u), v}_{\frakd}$ for $u, v\in \frakd$ (where $u^\vee,
v^\vee\in \frakd^*$ are the dual of $u, v$ with respect to
$\SP{\cdot,\cdot}_{\frakd}$). Note that $\rma$ is not
antisymmetric, since it satisfies
\begin{equation}\label{eq:rskew}
\rma(u^\vee, v^\vee)+ \rma(v^\vee, u^\vee)= \SP{u, v}_{\frakd}.
\end{equation}
But $\rho_S(\rma)$ is a bivector field, and $\pi_S=-\rho_S(\rma)$,
$$
\rho_S(\rma)^\sharp=\rho_S \rma^\sharp \rho_S^*=\rho_S j \iota^*
\rho_S^*=-\rho_S\sigmav = -\pi_S^\sharp,
$$
in accordance with \cite{AK}.
\end{remark}

Since $\rho\sigmav= -\sigmav^*\rho^*$ (which is the skew-symmetry
of $\pi_S$), it follows that $\pi_S$ can be restricted to any
orbit of the dressing action of $G$ on $S$: for any such orbit
$\mathcal{O}\subset S$ and any $\xi\in T^{*}_{x}S$ with $x\in
\mathcal{O}$  and $\xi|_{T_x\mathcal{O}}=0$, then $\rho^{*}(\xi)=
0$. Hence $\pi_{S}^{\sharp}(\xi)=0$. We denote by
\[
\pi_{\mathcal{O}}\in \Gamma(\wedge^2T\mathcal{O})
\]
the resulting bivector field. The next result follows directly
from Prop.~\ref{prop:biv}, and it was first proven in \cite{AK}.

\begin{proposition}
$(S,\pi_S)$ and $(\mathcal{O},\pi_{\mathcal{O}})$ are
quasi-Poisson spaces with respect to the dressing action of $G$ on
$S$.
\end{proposition}

Let us compute $\sigmav$ and $\pi_S$ in examples.

\begin{example}
For a Lie bialgebra $(\frakg,\gstar)$, the inclusion $\frakg^*
\hookrightarrow \frakg\oplus \frakg^*$ is an obvious choice of
isotropic splitting $j$. Since $\tri=0$, the induced bivector
$\pi_S$ on $S=D/G$ is a Poisson structure in this case. In the
context of Example \ref{ex:G*}, we have $S=G^*$ and
$\rho_S|_{\frakg^*}(\mu)= \mu^r$, so
\begin{equation}\label{eq:sigmavG*}
\sigmav(\alpha)=(dr_x)^*\alpha,\;\;\; \alpha \in T_x^*G^*
\end{equation}
and $\pi_S$ is defined by $\pi_S^\sharp(\alpha)_x=
\rho(dr_x^*\alpha)$, which agrees with \eqref{eq:piG*}. So the
graph of $\pi_S$ is $L_S$, the Dirac structure of Example
\ref{ex:G*}. The bialgebra in Example \ref{dual-Lie} is a
particular case for which $\rho_S|_{\frakg^*}=\mathrm{Id}$, so
$\sigmav=\mathrm{Id}$ and
$\pi_S^\sharp(u)(\mu)=\rho(u)(\mu)=\ad^*_u(\mu)$ is the usual
Lie-Poisson structure.
\end{example}

\begin{example}\label{ex:split}
For Example \ref{group-valued}, we consider the isotropic
splitting given by the anti-diagonal embedding,
$$
j(\mu):= \frac{1}{2}(\mu^\vee,-\mu^\vee),
$$
where $\mu \in \gstar$ and $\mu^\vee\in \frakg$ is its dual with
respect to $\Bi{\cdot,\cdot}$, i.e., $\mu=\Bi{\mu^\vee,\cdot}$.
Then a simple computation shows that $\cobr=0$ and $\tri$ is given
by
\begin{equation}\label{eq:tri}
\chi(\mu_1,\mu_2,\mu_3)=\frac{1}{4}\Bi{[\mu_1^\vee,\mu_2^\vee],\mu_3^\vee},
\end{equation}
i.e., $\chi \in \wedge^3\frakg$ is the Cartan trivector \cite{AK}.
In this case, $\sigmav:T^*G\to \frakg$ is given by
$$
\sigmav(\alpha_g)=\frac{1}{2}((dr_g^* + dl_g^*)(\alpha_g))^\vee =
\frac{1}{2}(dr_{g^{-1}} + dl_{g^{-1}})(\alpha_g^\vee),
$$
and the bivector field $\pi_S$ on $S=G$ is
$$
\pi_S(dl^*_{g^{-1}}(\mu),dl^*_{g^{-1}}(\nu))=\frac{1}{2}\Bi{(\Ad_{g^{-1}}-\Ad_g)\mu^\vee,\nu^\vee},
$$
see \cite{AKM}. Alternatively, if $e_i$ is a basis for $\frakg$
and $f_j$ is the dual basis with respect to
$\SP{\cdot,\cdot}_{\frakd}$, then $\pi_S=\frac{1}{2}\sum_i
e_i^l\wedge f_i^r$.
\end{example}


\subsection{The Lie algebroid of a quasi-Poisson $\frakg$-space}

We now present the construction of a Lie algebroid associated with
any quasi-Poisson space.

If $M$ is a manifold equipped with a bivector field $\pi$, one has
an induced bracket $[\cdot,\cdot]_\pi$ on the space of 1-forms on
$M$,
\begin{equation}\label{eq:pibrk}
[\alpha,\beta]_\pi:=\Lie_{\pi^\sharp(\alpha)}\beta -
\Lie_{\pi^\sharp(\beta)}\alpha - d\pi(\alpha,\beta),
\end{equation}
for $\alpha, \beta \in \Omega^1(M)$. Then $\pi$ is a Poisson
structure if and only if $[\cdot,\cdot]_\pi$ satisfies the Jacobi
identity, making $T^*M$ into a Lie algebroid  with anchor
$\pi^\sharp:T^*M\to TM$. The symplectic leaves of a Poisson
manifold are precisely the orbits of this Lie algebroid.

Suppose now that $M$ is equipped with a bivector field $\pi$ as
well as an infinitesimal action $\rho_M: \frakg\rmap TM$. Consider
the vector bundle $A:= \frakg \oplus T^*M $, let
\begin{equation}\label{eq:rmap}
r: A \to TM,\;\; r(v,\alpha):= \rho_M(v) + \pi^\sharp(\alpha),
\end{equation}
and consider the bracket on $\Gamma(A)=C^\infty(M,\frakg)\oplus
\Omega^1(M)$ defined by
\begin{align}
&[(u,0),(v,0)]_A=([u,v],0), \label{eq:algbr1}\\
&[(v,0),(0,\alpha)]_A=(-i_{\rho_M^*(\alpha)}(\cobr(v)),\Lie_{\rho_M(v)}\alpha),\label{eq:algbr2}\\
&[(0,\alpha),(0,\beta)]_A=(i_{\rho_M^*(\alpha\wedge \beta)}\tri,
[\alpha,\beta]_{\pi}),\label{eq:algbr3}
\end{align}
for $\alpha,\beta \in \Omega^1(M)$, and $u,v \in \frakg$,
considered as constant sections in $C^\infty(M,\frakg)$ (the
bracket is extended to general elements by the Leibniz rule). The
main result in this section is the following:

\begin{theorem}\label{thm:liealg}
$(\frakg\oplus T^*M, r, [\cdot,\cdot]_A)$ is a Lie algebroid if
and only if $(M,\pi)$ is a quasi-Poisson $\frakg$-space with
respect to the action $\rho_M:\frakg \to \X^1(M)$.
\end{theorem}

The Lie algebroid in Theorem~\ref{thm:liealg} has as special cases
the Lie algebroids previously introduced in \cite{BC} and
\cite{Lu2}, but our general proof follows a different approach.
The following result is a direct consequence of
Theorem~\ref{thm:liealg}.

\begin{corollary}
On a quasi-Poisson $\frakg$-space $(M,\pi)$, the generalized
distribution
\begin{equation}\label{eq:dist}
\{\rho_M(v)+\pi^\sharp(\alpha)\;|\; v\in \frakg, \; \alpha \in
T^*M\}\subseteq TM
\end{equation}
is integrable.
\end{corollary}

The fact that \eqref{eq:dist} defines a singular foliation was
first observed in \cite{AK,AKM}, but under the additional
assumptions of existence of a moment map and that $\frakg$ is
integrated by a compact Lie group.

In the theory of quasi-Poisson spaces, a particular role is played
by those with \emph{transitive} Lie algebroid, i.e.,
\begin{equation}\label{eq:transitive}
TM= \{\rho_M(v)+\pi^\sharp(\alpha)\;|\; v\in \frakg, \; \alpha \in
T^*M\}.
\end{equation}
Note that, if the $\frakg$-orbits are tangent to the distribution
$\pi^\sharp(T^*M)$, then this transitivity condition implies that
the bivector field $\pi$ is nondegenerate (but this is not the
case in general).

\subsubsection{The proof of Theorem~\ref{thm:liealg}}

In this subsection we present the proof of Theorem
\ref{thm:liealg}; see \cite{BCS} for an alternative discussion.


As recalled in Section~\ref{subsec:app1}, describing a Lie
algebroid structure on $A=\frakg\oplus T^*M$ is equivalent to
finding a degree-1 derivation $d_A:\Gamma(\wedge^\bullet A^*) \to
\Gamma(\wedge^{\bullet+1} A^*)$ satisfying $d_A^2=0$ \cite{Vain}:
the anchor $r$ and bracket $[\cdot,\cdot]_A$ are recovered by the
conditions
\begin{align}
&d_Af(a)=\Lie_{r(a)}f, \label{eq:anch}\\
&d_A\xi(a,b)= \Lie_{r(a)}\xi(b)-\Lie_{r(b)}\xi(a)-\xi([a,b]_A),
\label{eq:brack}
\end{align}
for $f \in C^\infty(M)$, $a,b \in \Gamma(A)$ and $\xi \in
\Gamma(A^*)$. We now present a construction of a differential
$d_A$ leading to the bracket defined by \eqref{eq:algbr1},
\eqref{eq:algbr2} and \eqref{eq:algbr3}.

Let $(\frakg,\cobr,\tri)$ be a Lie quasi-bialgebra, and let $\X =
\oplus_{q\in \mathbb{Z}}\X^q$ be any graded commutative algebra.
We consider the tensor product of graded commutative algebras
$\wedge \gstar \otimes \X$, which is itself graded commutative
with product
$$
(\mu \otimes x)\cdot (\nu \otimes y):= (-1)^{qp'}(\mu\wedge
\nu)\otimes (x \cdot y),
$$
for $\mu \otimes x \in \wedge^p\gstar \otimes \X^q$, $\nu \otimes
y \in \wedge^{p'}\gstar \otimes \X^{q'}$, and grading $(\wedge
\gstar \otimes \X)^k=\oplus_{p+q=k}\wedge^p\gstar \otimes \X^q.$
We assume that $\X$ is equipped with an operator $d:\X^\bullet \to
\X^{\bullet+1}$ which is a derivation of degree 1,
$$
 d(x\cdot y)=dx\cdot y + (-1)^{q}x\cdot dy,\;\;\;\mbox{ for }
x \in \X^q, y\in \X^{q'},
$$
but \textit{not} necessarily squaring to zero. (The example to
have in mind is when $\X^\bullet$ is given by multivector fields
on $M$, and $d=[\pi,\cdot]$ for a quasi-Poisson bivector $\pi$,
where $[\cdot,\cdot]$ is the Schouten bracket). We moreover assume
that $\frakg$ acts on $\X^\bullet$ by derivations of degree $0$,
and that we are given a $\frakg$-equivariant map
\begin{equation}\label{eq:rhomap}
\varrho:\frakg \to \X^1
\end{equation}
with respect to the adjoint action of $\frakg$ on itself. The map
$\varrho$ induces an equivariant map of graded algebras
$\wedge^\bullet \frakg \to \X^\bullet$ which we also denote by
$\varrho$.

In this framework, one can define various derivations of $\wedge
\gstar \otimes \X$.  First, the Chevalley-Eilenberg operator
(i.e., the Lie algebra differential) of $\frakg$ with coefficients
in $\X$,
\begin{equation}
\partial: \wedge^\bullet \gstar \otimes \X^\bullet \to \wedge^{\bullet+1}\gstar \otimes \X^\bullet,
\end{equation}
is a derivation of $\wedge\gstar \otimes \X$ of degree 1. Second,
we define
\begin{equation}
d: \wedge^\bullet \gstar \otimes \X^\bullet \to
\wedge^{\bullet}\gstar \otimes \X^{\bullet+1}
\end{equation}
to be the unique derivation of $\wedge\gstar \otimes \X$ extending
the operator $d: \X^\bullet \to \X^{\bullet +1}$ and vanishing on
$\wedge \gstar$. Finally, associated with $\cobr$ and $\tri$, we
define
$$
\partial_{\cobr}: \wedge^\bullet \gstar \otimes \X^\bullet \to
\wedge^\bullet \gstar \otimes \X^{\bullet+1}\;\;\mbox{ and }\;\;
\partial_{\tri}: \wedge^\bullet \gstar \otimes \X^\bullet \to
\wedge^{\bullet-1} \gstar \otimes \X^{\bullet+2}
$$
to be the unique derivations vanishing on $\X$ and defined on
$\gstar$ by the conditions
\begin{align}
&\partial_{\cobr}:\gstar \to \gstar\otimes \X^1,\;\;\;
\partial_{\cobr}\mu(u)=
- \varrho(i_\mu(\cobr(u))),\;\; u \in \frakg;\\
&\partial_{\tri}:\gstar \to \X^2,\;\;\; \partial_\tri\mu= -
\varrho(i_\mu \tri).
\end{align}
We now consider the derivation
\begin{equation}\label{eq:bigD}
\D := \partial +\partial_{\cobr} + \partial_{\tri} + d: (\wedge
\gstar \otimes \X)^\bullet \to (\wedge\gstar \otimes
\X)^{\bullet+1}.
\end{equation}
The following is a key example of this framework in which $\D$ is
a differential, i.e., $\D^2=0$.

\begin{lemma}\label{lem:double}
Consider $\X = \wedge \frakg$, equipped with the adjoint action of
$\frakg$, and let $\varrho= \id : \frakg \to \frakg$ and
$d:\wedge^\bullet\frakg \to \wedge^{\bullet+1}\frakg$ be the Lie
algebra differential of the Lie bracket
$-\cobr^*:\gstar\wedge\gstar \to \gstar$ on $\gstar$. Then, under
the canonical isomorphism $\wedge \gstar \otimes \wedge \frakg
\cong \wedge (\gstar \oplus \frakg)$, $\D$ is the Lie algebra
differential of the Drinfeld double of the Lie quasi-bialgebra
$(\frakg,-\cobr,\tri)$.
\end{lemma}

To prove the lemma, it suffices to compare how the derivation $\D$
defined by \eqref{eq:bigD} and the Lie algebra differential act on
elements in $\gstar$ and $\frakg$, and a direct computation using
the definitions shows that they agree.

Note that to check that $\D^2=0$ in general, it suffices to check
that $\D^2=0$ on $\wedge \gstar$ and $\X$ separately since, if
$\mu\otimes x \in \wedge^p \gstar \otimes \X^q$, we have
\begin{eqnarray*}
\D^2(\mu\otimes x)&=& \D^2(\mu)\cdot (1\otimes x) +
(-1)^{p+1}\D(\mu)\cdot \D(x) + (-1)^p\D(\mu)\cdot \D(x)
+\\
&& (-1)^{2p}(\mu\otimes 1)\cdot \D^2(x) \\
& = & \D^2(\mu)\cdot (1\otimes x) + (\mu\otimes 1)\cdot \D^2(x).
\end{eqnarray*}

\begin{lemma}\label{lem:cond23}
On $\X$, $\D^2=0$ is equivalent to the conditions
\begin{align}
&d\partial + \partial d + \partial_\cobr \partial = 0\label{eq:cond2}\\
&\partial_{\tri}\partial + d^2=0.\label{eq:cond3}
\end{align}
\end{lemma}

\begin{proof}
A direct computation shows that
\begin{equation}\label{eq:D2X}
\D^2=\partial d + d\partial + \partial_{\cobr} \partial +
\partial_{\tri} \partial + d^2
\end{equation}
on $\X$, where we have used that $\partial_{\cobr}$ and
$\partial_{\tri}$ vanish on $\X$ and $\partial^2=0$. Splitting
\eqref{eq:D2X} according to degrees, one obtains \eqref{eq:cond2}
and \eqref{eq:cond3}.
\end{proof}

Let us now focus on the case where $(\X,[\cdot,\cdot])$ is a
Gerstenhaber algebra (see equations \eqref{eq:sch2},
\eqref{eq:sch3} and \eqref{eq:grdjacobi} in Sec.\ref{subsec:app1}
for conventions) and that $\varrho:\frakg \to \X^1$ is a Lie
algebra homomorphism such that the $\frakg$-action on $\X$ is
given by
\begin{equation}\label{eq:Xact}
v\cdot x := [\varrho(v),x],\;\;\; v \in \frakg,\; x\in \X.
\end{equation}

\begin{lemma}\label{lem:formulas}
For all $x \in \X$ and $v \in \frakg$, the following holds:
\begin{align}
& \partial_{\tri}\partial(x) = - [\varrho(\tri), x ];\label{eq:form1}\\
& \SP{(d\partial + \partial d)(x),v}= [\varrho(v),dx] - d[\varrho(v),x] \label{eq:form2}\\
& \SP{\partial_{\cobr}\partial(x),v} = [\varrho(\cobr(v)),x].
\label{eq:form3}
\end{align}
\end{lemma}

\begin{proof}
Let $\{e_i\}$ be a basis of $\frakg$, and let $\{e^i\}$ be the
dual basis. To prove \eqref{eq:form1}, take $x\in \X$, and recall
that
\begin{equation}\label{eq:partialx}
\partial(x)= \sum_le^l\otimes [\varrho(e_l),x], \; \mbox{ and } \; \SP{\partial(x),v}=[\varrho(v),x].
\end{equation}
By definition of $\partial_{\tri}$, we get
$$
\partial_{\tri}\partial x = - \sum_l \varrho(i_{e^l}\tri) [\varrho(e_l),x]=
- 3\sum_{i,j,k}\tri_{ijk}\varrho(e_i)\varrho(e_j)[\varrho(e_k),x],
$$
where we have written $\tri = \sum \tri_{ijk} e_i\wedge e_j\wedge
e_k$. On the other hand,
$$
[\varrho(\tri),x]=\sum_{i,j,k}\tri_{ijk}[\varrho(e_i)\varrho(e_j)\varrho(e_k),x]
= 3 \sum_{i,j,k}
\tri_{ijk}\varrho(e_i)\varrho(e_j)[\varrho(e_k),x],
$$
where the last equality follows from the graded Leibniz identity
for $[\cdot,\cdot]$. This proves \eqref{eq:form1}.

>From the first formula in \eqref{eq:partialx} and the derivation
property of $d$, it follows that $d\partial x = -\sum e^i\otimes
d([\varrho(e_i),x])$, and
$$
\SP{d\partial x,v}=-\sum v_i d[\varrho(e_i),x]=-d[\varrho(v),x].
$$
The second equation in \eqref{eq:partialx} implies that
$\SP{\partial dx, v}=[\varrho(v),dx]$, so \eqref{eq:form2}
follows.

To prove \eqref{eq:form3}, note that, using \eqref{eq:partialx}
and the definition of $\partial_{\cobr}$, we have
$$
\SP{\partial_{\cobr}\partial(x),v}=\sum_l\varrho(i_{e^l}\cobr(v))[\varrho(e_l),x]=
2 \sum_{i,j,k}F_{jk}^iv_i\varrho(e_j)[\varrho(e_k),x],
$$
where we have written $F(v)=\sum_{i,j,k}F_{jk}^i v_i e_j\wedge
e_k$. On the other hand, an application of the Leibniz identity
yields
$$
[\varrho(\cobr(v)),x]=\sum_{i,j,k}F^i_{jk}v_i[\varrho(e_j)\varrho(e_k),x]=
2 \sum_{i,j,k}F^i_{jk}v_i\varrho(e_j)[\varrho(e_k),x],
$$
proving \eqref{eq:form3}.
\end{proof}

If we further assume that the derivation $d:\X^\bullet \to
\X^{\bullet+1}$ has the form $d_{\pi}=[\pi, \cdot]$  for some
fixed $\pi \in \X^2$,
then combining \eqref{eq:form2} and \eqref{eq:form3}, and using
the graded Jacobi identity for $[\;,\;]$, one checks that
\eqref{eq:cond2} is equivalent to
\begin{equation}\label{eq:ax1b}
[[\varrho(v),\pi],x]= - [\varrho(\cobr(v)),x] \;\mbox{ for all } x
\in \X.
\end{equation}
Another application of the graded Jacobi identity shows that
$d_{\pi}^2=[\frac{1}{2}[\pi,\pi],\cdot]$, so, using
\eqref{eq:form1}, one sees that \eqref{eq:cond3} can be rewritten
as
\begin{equation}\label{eq:ax2b}
\frac{1}{2} [[\pi,\pi], x]= [\varrho(\tri),x]\; \mbox{ for all } x
\in \X.
\end{equation}

Let us now consider the specific situation where $M$ is a manifold
equipped with a bivector field $\pi\in \X^2(M)$ as well as an
infinitesimal $\frakg$-action $\rho_M:\frakg \to \X^1(M)$ (playing
the role of $\varrho$ \eqref{eq:rhomap}). We use \eqref{eq:Xact}
to define a $\frakg$-action on the Gerstenhaber algebra
$\X^\bullet(M)$ of multivector fields, and consider the derivation
\eqref{eq:bigD} in this context,
\begin{equation}\label{eq:bigD2}
\D:= \partial +\partial_{\cobr} + \partial_{\tri} + d: \wedge
\gstar \otimes \X(M) \to \wedge \gstar \otimes \X(M),
\end{equation}
with $d=d_\pi=[\pi,\cdot]$.

\begin{lemma}\label{lem:quasid}
The action $\rho_M$ makes $(M,\pi)$ into a quasi-Poisson
$\frakg$-space if and only if $\D^2=0$.
\end{lemma}

\begin{proof}
To prove the claim, as we have previously remarked, it suffices to
show that $\D^2=0$ on $\wedge \gstar$ and $\X(M)$ separately.
Using Lemma \ref{lem:formulas}, we saw that \eqref{eq:cond3} is
equivalent to \eqref{eq:ax2b}, which is the same as
\begin{equation}
\frac{1}{2}[\pi,\pi]=\rho_M(\tri)
\end{equation}
when $\X=\X(M)$. Similarly, \eqref{eq:cond2} is equivalent to
\begin{equation}\label{eq:con2a}
[\rho_M(v),\pi]= - \rho_M(\cobr(v)).
\end{equation}
It then follows from Lemma \ref{lem:cond23} that $\rho_M$ is a
quasi-Poisson action if and only if $\D^2=0$ on $\X(M)$. To finish
the proof of the lemma, we will check that if $\D^2=0$ on $\X(M)$,
then it automatically happens that $\D^2=0$ on $\wedge \gstar$,
i.e., \eqref{eq:bigD2} squares to zero.

Let $\D^0$ denote the $\D$-operator \eqref{eq:bigD} in the
specific context of Lemma \ref{lem:double}, and similarly for the
derivations $\partial^0$, $d^0$, $\partial_{\cobr}^0$ and
$\partial_{\tri}^0$. Consider the homomorphism
$$
\Psi: \wedge \gstar \otimes \wedge \frakg \to \wedge \gstar
\otimes \X(M), \;\;\; \Psi(\mu \otimes v)=\mu \otimes \rho_M(v).
$$
We claim that \eqref{eq:con2a} implies that
\begin{equation}\label{eq:comm}
\Psi\circ \D^0 = \D \circ \Psi.
\end{equation}
Indeed, recall that $d^0=F$ on $\frakg$, and condition
\eqref{eq:con2a} is equivalent to $d_\pi\rho(v)=\rho_M(\cobr(v))$,
i.e., $ \Psi \circ d^0 = d_\pi \circ \Psi.$ One can immediately
check that analogous relations automatically hold for $\partial$,
$\partial_{\cobr}$ and $\partial_{\tri}$. For $\mu \otimes v \in
\wedge \gstar \otimes \wedge \frakg$, we have
$$
\partial^0(\mu \otimes v)= \partial_{\frakg}(\mu)\otimes v +
\sum_i (e^i\wedge \mu)\otimes e_i\cdot v,
$$
where $\{e_i\}$ is a basis of $\frakg$, $\{e^i\}$ is the dual
basis, and $\partial_{\frakg}$ denotes the Lie algebra
differential of $\frakg$ (with trivial coefficients). Using the
equivariance of $\rho_M$, we get
$$
\Psi(\partial^0(\mu \otimes v))=\partial_{\frakg}(\mu)\otimes
\rho_M(v) + \sum_i (e^i\wedge \mu)\otimes e_i\cdot \rho_M(v)=
\partial(\mu \otimes \rho_M(v)) = \partial(\Psi(\mu \otimes v)).
$$
To see that $\Psi \circ \partial_{\cobr}^0 = \partial_{\cobr}
\circ \Psi$, it suffices to check the equality on $\gstar$, and
this follows directly from the definitions. One checks that $\Psi
\circ \partial_{\tri}^0 = \partial_{\tri} \circ \Psi$ similarly.

To check that $\D^2=0$ on $\wedge \gstar$, note that
$$
\D(\D(\mu\otimes 1_{\X}))= \D(\D(\Psi(\mu \otimes
1)))=\D(\psi(\D^0(\mu \otimes 1))) =\Psi((\D^0)^2(\mu \otimes 1))
=0
$$
since $(\D^0)^2=0$ by Lemma \ref{lem:double}. This completes the
proof of the lemma.
\end{proof}

\begin{proof}(of Theorem \ref{thm:liealg})
Since $\wedge\gstar \otimes \X(M)$ can be identified with
$\Gamma(\wedge(\gstar \oplus TM))$, Lemma \ref{lem:quasid} shows
that $\rho_M$ defines a quasi-Poisson action on $(M,\pi)$ if and
only if $\D$ (given by \eqref{eq:bigD2}) defines a Lie algebroid
structure on $A=(\gstar\oplus TM)^*= \frakg \oplus T^*M$. To
complete the proof of Theorem \ref{thm:liealg}, it remains to
compute the anchor and bracket of $A$ using \eqref{eq:anch} and
\eqref{eq:brack}.

Since $\partial_{\cobr}$ and $\partial_{\tri}$ vanish on $\X(M)$,
it follows that, for $f \in C^\infty(M)=\X^0(M)$, we have
$$
\D f = \partial f + d_{\pi}f \in (\wedge \gstar \otimes \X(M))^1
\cong C^\infty(M,\gstar)\oplus \X^1(M) = \Gamma(A^*).
$$
For $(v,\alpha) \in C^\infty(M,\frakg)\oplus \Omega^1(M) =
\Gamma(A)$, a direct computation shows that
$$
\D f(v,\alpha)= \Lie_{\rho_M(v)}  f + \Lie_{\pi^\sharp(\alpha)}f,
$$
and, using \eqref{eq:anch}, we conclude that the anchor is given
by \eqref{eq:rmap}.

We now compute brackets. For $a, b \in \Gamma(A)$, we write
$$
[a,b]_A=([a,b]_{1},[a,b]_{2}),
$$
with $[a,b]_1 \in C^\infty(M,\frakg)$ and $[a,b]_2 \in
\Omega^1(M)$. The first case is when $a=(u,0), b=(v,0) \in
\Gamma(A)$, where $u,v \in \frakg$ (considered as constant
elements in $C^\infty(M,\frakg)$). If $\xi=(\mu,0) \in
\Gamma(A^*)$, $\mu \in \gstar$, then a direct computation shows
that $\D\xi(a,b)=\partial_{\frakg}\mu(u,v)=-\mu([u,v])$. From
\eqref{eq:brack}, it follows that $\D\xi(a,b)=-\mu([a,b]_1)$, so
$[(u,0),(v,0)]_1=[u,v]$. A similar computation shows that
$[(u,0),(v,0)]_2=0$, proving \eqref{eq:algbr1}.

Let $u \in \frakg$, $\alpha \in \Omega^1(M)$ and $\mu \in \gstar$,
and consider $a=(u,0), b=(0,\alpha)$ and $\xi=(\mu,0)$. Using the
definition of $\D$, we obtain
$$
\D\xi(a,b)=i_{\alpha}(\partial_{\cobr}\mu(u))=-i_{\alpha}\rho_M(i_\mu\cobr(u))=
\mu(i_{\rho_M^*(\alpha)}\cobr(u)).
$$
On the other hand, by \eqref{eq:brack}, we have
$\D\xi(a,b)=-\xi([a,b])=-\mu([a,b]_1)$, which implies that
$[a,b]_1= -i_{\rho_M^*(\alpha)}\cobr(u)$. To compute $[a,b]_2$,
let $\xi=(0,X)$, where $X \in \X^1(M)$. Then one checks that
$\D\xi(a,b)=\alpha([\rho_M(u),X])$. But, by \eqref{eq:brack}, we
have
$$
\D\xi(a,b)=\Lie_{\rho_M(u)}\alpha(X) - i_X([a,b]_2).
$$
So $i_X([a,b]_2)=\Lie_{\rho_M(u)}i_X\alpha -
i_{[\rho_M(u),X]}\alpha = i_X(\Lie_{\rho_M(u)}\alpha)$. This
proves \eqref{eq:algbr2}.

In order to prove \eqref{eq:algbr3}, let $a=(0,\alpha),
b=(0,\beta)$, where $\alpha, \beta \in \Omega^1(M)$. If
$\xi=(\mu,0)$, $\mu \in \gstar$, then, by \eqref{eq:brack}, we
have $\D\xi(a,b)=-\mu([a,b]_1)$. On the other hand,
$$
\D\xi(a,b)=\partial_{\tri}(\mu)(\alpha,\beta)=
-i_{\alpha\wedge\beta}(\rho_M(i_\mu\tri))=
-\mu(i_{\rho_M^*(\alpha\wedge\beta)}\tri).
$$
It follows that $[a,b]_1= i_{\rho_M^*(\alpha\wedge\beta)}\tri$. To
compute the second component of $[a,b]$, we let $\xi=(0,X)$, $X
\in \X^1(M)$. Then
$\D\xi(a,b)=d_{\pi}X(\alpha,\beta)=i_{\alpha\wedge\beta}([\pi,X])$.
By \eqref{eq:brack},
\begin{eqnarray}
i_X([a,b]_2)&=
&\Lie_{\pi^\sharp(\alpha)}i_X\beta-\Lie_{\pi^\sharp(\beta)}i_X\alpha
-i_{\alpha\wedge\beta}([\pi,X])\nonumber\\
&=&
i_X\Lie_{\pi^\sharp(\alpha)}\beta-i_X\Lie_{\pi^\sharp(\beta)}\alpha
- (i_{\alpha\wedge\beta}([\pi,X])-i_{[\pi^\sharp(\alpha),X]}\beta
+ i_{[\pi^\sharp(\beta),X]}\alpha )\label{eq:2compbr3}
\end{eqnarray}
On the other hand,
$i_Xd(\pi(\alpha,\beta))=\Lie_X(\pi(\alpha,\beta))$ equals to
$$
(\Lie_X\pi)(\alpha,\beta)+\pi(\Lie_X\alpha,\beta)+\pi(\alpha,\Lie_X\beta)
= i_{\alpha\wedge\beta}[X,\pi]+ i_{\pi^\sharp(\Lie_X\alpha)}\beta
-i_{\pi^\sharp(\Lie_X\beta)}\alpha.
$$
Using that ${\pi^\sharp(\Lie_X\alpha)}=
{[X,\pi^\sharp(\alpha)]}-{[X,\pi]^\sharp(\alpha)}$, it follows
that
$$
i_Xd(\pi(\alpha,\beta))=
i_{\alpha\wedge\beta}([\pi,X])-i_{[\pi^\sharp(\alpha),X]}\beta +
i_{[\pi^\sharp(\beta),X]}\alpha.
$$
>From \eqref{eq:2compbr3}, we see that
$[a,b]_2=[\alpha,\beta]_\pi$.
\end{proof}

\subsection{The Hamiltonian category}\label{subsec:hamcat2}

Let $(\frakd,\frakg,j)$ be a split Manin pair, so that $\frakg$
acquires the structure of a Lie quasi-bialgebra. Let
$\rho_M:\frakg \to \X(M)$ be an action making $(M, \pi)$ into a
quasi-Poisson $\frakg$-space. A \textbf{moment map} \cite{AK} for
this quasi-Poisson action is a smooth $\frakg$-equivariant map
\[
J: M\rmap S=D/G
\]
with the property that
\begin{equation}\label{eq:mmapcond}
\pi^\sharp J^* = \rho_M \sigmav,
\end{equation}
where $\sigmav=\sigmav_j:T^*S\to \frakg_S$ is defined in
\eqref{eq:sigmav}. In this case we say that $(M,\pi,\rho_M,J)$ is
a \textbf{Hamiltonian quasi-Poisson $\frakg$-space} with respect
to the split Manin pair $(\frakd, \frakg, j)$ (or, equivalently,
the Lie quasi-bialgebra $(\frakg,\cobr_j,\tri_j)$). If the
$\frakg$-action integrates to a $G$-action, we refer to a
Hamiltonian quasi-Poisson $G$-space.

Combining the equivariance of $J$ and \eqref{eq:mmapcond}, we have
that $dJ \pi^\sharp J^* = \rho \sigmav = \pi_S^\sharp$, i.e., $J_*
\pi = \pi_S$.

\begin{remark}\label{rem:twist2}
If $j'$ is another isotropic splitting and $t\in \wedge^2 \frakg$
is the associated twist, we saw in Remark \ref{rem:twist} that
$(M,\pi')$, where $\pi'=\pi+\rho_M(t)$, is a quasi-Poisson space
for $(\frakg,\cobr_{j'},\tri_{j'})$. Since
$$
\sigmav_{j'}=(\rho_S\circ j')^*=\sigmav_j + t^\sharp \circ
\rho_S^*,
$$
it follows that $ (\pi')^\sharp dJ^*=(\pi^\sharp +
\rho_M(t)^\sharp)dJ^*=(\rho_M\sigmav_j + \rho_M t^\sharp \rho_M^*
dJ^*)=\rho_M\sigmav_{j'},$ so the fact that $\rho_M$ is
Hamiltonian is independent of the choice of isotropic splitting.
\end{remark}

\begin{remark}\label{rem:adm}
The moment map condition \eqref{eq:mmapcond} is equivalent to the
one in \cite{AK} using the notion of \textit{admissibility}. An
isotropic splitting $j$ is called \textbf{admissible}
\cite[Sec.~3.4]{AK} if the restriction
$\rho_S|_{\frakh_S}:\frakh_S \to TS$ is an isomorphism, where
$\frakh=j(\gstar)$. This is equivalent to the bundle map
$\sigmav:T^*S\to \frakg_S$ being an isomorphism, in which case the
moment map condition \eqref{eq:mmapcond} can be written as
\begin{equation}\label{eq:mconad}
\pi^\sharp(J^*\sigmav^{-1}(v))=\rho_M(v),\;\;\; \forall v \in
\frakg.
\end{equation}
Since admissible sections always exist locally, a quasi-Poisson
action is Hamiltonian if and only if it satisfies, possibly after
a twist, \eqref{eq:mconad} locally, which is the original
definition in \cite[Def.~5.1.1]{AK}.
\end{remark}

\begin{example}
A canonical example of a Hamiltonian quasi-Poisson action is given
by the dressing $G$-action on $S$, in which case the moment map is
the identity $S\to S$. This restricts to Hamiltonian actions on
dressing orbits, with moment maps given by the inclusion maps
$\mathcal{O}\hookrightarrow S$.
\end{example}

We now have definitions parallel to those in
Section~\ref{subsec:DG1}: the \textbf{Hamiltonian category} (or
\textit{moment map theory}) associated with a split Manin pair
$(\frakd,\frakg,j)$ is the category $ \M_j(\frakd,\frakg)$ whose
objects are Hamiltonian quasi-Poisson $\frakg$-spaces
$(M,\pi,\rho_M,J)$ (with respect to the Lie quasi-bialgebra
$(\frakg,\cobr_j,\tri_j)$) and morphisms are smooth maps smooth
maps $f: M\rmap M'$ satisfying $f_{*}(\pi)= \pi'$ and $J'f= J$. We
also consider the subcategory $ \Mp_j(\frakd,\frakg)$ consisting
of Hamiltonian quasi-Poisson spaces with transitive Lie
algebroids, i.e., satisfying \eqref{eq:transitive}.

\begin{example}\label{ex:quasiG*}
Let us consider a Lie bialgebra as in Example \ref{ex:G*}, where
$S=G^*$ has the Poisson structure $\pi_{G^*}$. In this case, $j$
is the inclusion $\gstar \hookrightarrow \frakd$, and it is
admissible (in the sense of Remark~\ref{rem:adm}). As a result,
the moment map condition \eqref{eq:mmapcond} completely determines
the action $\rho_M$, and one can check that $\M_j(\frakd,\frakg)$
is the category of Poisson maps into $G^*$. The subcategory
$\Mp_j(\frakd,\frakg)$ consists of Poisson maps $M\to G^*$ for
which $M$ carries a \textit{nondegenerate} Poisson structure,
i.e., $M$ is symplectic. Comparing with the Hamiltonian categories
associated with the connection $s$ in Example~\ref{ex:G*}, we see
that $\sigmav_j=\sigma_s^{-1}$ and we have natural identifications
$$\M_s(\frakd,\frakg)=\M_j(\frakd,\frakg)\; \mbox{ and }\;
\Mp_s(\frakd,\frakg)=\Mp_j(\frakd,\frakg).
$$

\end{example}

\begin{example}
For the split Manin pair of Example \ref{ex:split}, the objects in
$\M_j(\frakd,\frakg)$ are the Hamiltonian quasi-Poisson
$\frakg$-manifolds of \cite[Sec.~2]{AKM}, i.e., manifolds $M$
equipped with a $\frakg$-action, an invariant bivector field $\pi$
satisfying $\frac{1}{2}[\pi,\pi]=\rho_M(\tri)$ (where $\tri$ is
the Cartan trivector \eqref{eq:tri}), and an equivariant map
$J:M\to G$ satisfying the moment map condition
\eqref{eq:mmapcond}, which reads
\begin{equation}\label{eq:Gmmmapcond}
\pi^\sharp J^*(\alpha) = \frac{1}{2}\rho_M( (\theta^R+\theta^L)
\alpha^\vee),
\end{equation}
where $\alpha \in \Omega^1(G)$ and $\alpha^\vee \in \X(G)$ is dual
to $\alpha$ with respect to the metric. Note that objects in
$\Mp_j(\frakd,\frakg)$ may still carry degenerate bivector fields,
and the relationship between Hamiltonian spaces in
$\M_j(\frakd,\frakg)$ and $\M_s(\frakd,\frakg)$ is now less
evident. Nevertheless, as proven in \cite{AKM,BC} (see also
\cite{ABM}), there is a (nontrivial) isomorphism
$\M_j(\frakd,\frakg)\cong \M_s(\frakd,\frakg)$, which will be
explained and generalized in Section \ref{sec:equiv}.
\end{example}

Just as different choices of connections give rise to gauge
transformations (see Prop.~\ref{prop:gauge2}), different choices
of isotropic splittings $j$, $j'$ define a twist $t\in
\wedge^2\frakg$ and, following Remarks~\ref{rem:twist} and
\ref{rem:twist2}, the operation $\pi \mapsto \pi +\rho_M(t)$
induces a functor $\M_j(\frakd,\frakg) \to
\M_{j'}(\frakd,\frakg)$.

\begin{proposition}\label{prop:twist}
If $j$ and $j'$ are two isotropic splittings of $\frakd$, then the
associated twist $t=j-j'$ defines a natural isomorphism:
\begin{equation}\label{eq:twist}
\I_t: \M_j(\frakd, \frakg) \cong \M_{j'}(\frakd, \frakg),
\end{equation}
which restricts to an isomorphism of subcategories
$\Mp_j(\frakd,\frakg)\cong \Mp_{j'}(\frakd,\frakg)$.
\end{proposition}

\begin{example}
The Manin pair $(\frakg^\mathbb{C},\frakg)$ of Example~\ref{ex:P}
admits two natural splittings $j$ and $j'$: one corresponds to the
Lagrangian complement $\gstar \cong \mathfrak{a}\oplus
\mathfrak{n}$ and the other to $\sqrt{-1}\frakg \subset
\frakg^\mathbb{C}$. For the splitting $j$ one has $\tri_j=0$,
whereas, for $j'$, $F_{j'}=0$. Prop.~\ref{prop:twist} establishes
an isomorphism between the corresponding Hamiltonian categories.
Notice that Hamiltonian spaces associated with $j$ are Poisson
manifolds for which the Poisson structure is generally not
$\frakg$-invariant, whereas for $j'$ Hamiltonian spaces carry
$\frakg$-invariant bivector fields which generally fail to be
Poisson.
\end{example}

We close this section with remarks about Hamiltonian reduction for
$D/G$-valued moment maps in quasi-Poisson geometry. As proven in
\cite[Thm.~4.2.2]{AKM}, if $(M,\pi)$ is a quasi-Poisson $G$-space,
then conditions \eqref{cond:q-poiss1}, \eqref{cond:q-poiss2}
directly imply that the bracket defined by $\pi$ makes the space
of $G$-invariant functions $C^\infty(M)^G$ into a Poisson algebra.
In particular, the orbit space of $M$ by the $G$-action is a
Poisson manifold whenever the action is free and proper.

Let us assume that we are in the Hamiltonian situation, i.e.,
there is a moment map $J:M\to S=D/G$.

\begin{proposition}\label{prop:qpoissonred}
Let $(M,\pi,\rho_M,J)$ be a Hamiltonian quasi-Poisson $G$-space
associated with $(\frakd,\frakg,j)$. Let $y\in S$ be a regular
value for $J$, and let $\mathcal{O}$ be the dressing orbit through
$y$. Then the $G$-invariant functions on $J^{-1}(\mathcal{O})$
form a Poisson algebra. (In particular, the quotient
$J^{-1}(\mathcal{O})/G$ is a Poisson manifold whenever the
$G$-action on $J^{-1}(\mathcal{O})$ is free and proper.)
\end{proposition}
\begin{proof}
Let $f,g \in C^\infty(J^{-1}(\mathcal{O}))^G$, and let
$\widetilde{f}, \widetilde{g}$ be arbitrary extensions of $f$ and
$g$ to $M$. It suffices to check that
$\pi(d\widetilde{f},d\widetilde{g})|_{J^{-1}(\mathcal{O})}$ does
not depend on the extensions. Since $d\widetilde{f}(\rho_M(v))=0$
over $J^{-1}(\mathcal{O})$, we use the (adjoint of the) moment map
condition, $dJ \pi^\sharp = - \sigmav^* \rho_M^*$, to see that
$\pi^\sharp(d\widetilde{f})|_{J^{-1}(\mathcal{O})} \in
TJ^{-1}(\mathcal{O})$. Hence if
$\widetilde{g}|_{J^{-1}(\mathcal{O})}\equiv 0$, we must have
$\pi(d\widetilde{f},d\widetilde{g})|_{J^{-1}(\mathcal{O})}\equiv
0$. The fact that the bracket
$\{f,g\}:=\pi(d\widetilde{f},d\widetilde{g})$ is a Poisson bracket
on $C^\infty(J^{-1}(\mathcal{O}))^{G}$ is a direct consequence of
conditions \eqref{cond:q-poiss1} and \eqref{cond:q-poiss2}.
\end{proof}

It is immediate to check that twists \eqref{eq:twist} keep reduced
spaces unchanged.

The proof of Prop.~\ref{prop:qpoissonred} shows that the
restriction $C^\infty(M)^G\to C^\infty(J^{-1}(\mathcal{O}))^G$ is
a Poisson map. If the $G$-action on $M$ is free and proper, it
follows that $J^{-1}(\mathcal{O})/G$ sits in $M/G$ as a Poisson
submanifold. We will see in Section \ref{subsec:propex} that the
symplectic leaves of $J^{-1}(\mathcal{O})/G$ are the projections
of the leaves of the Lie algebroid associated with $(M,\pi)$ (in
particular, the quotient is symplectic if the Lie algebroid of
$(M,\pi)$ is transitive).


\section{The equivalence of Hamiltonian categories}\label{sec:equiv}

We saw in Sections~\ref{sec:dirac} and \ref{sec:quasip} two
possible $D/G$-valued moment map theories associated with a Manin
pair $(\frakd,\frakg)$: one depends on the choice of an isotropic
connection $s$ on the $G$-bundle $D\to G$ and leads to Dirac
geometry (and equivariant 3-forms when $s$ is invariant), and the
other depends on the choice of an isotropic splitting $j$ of
\eqref{eq:maninexact} and leads to quasi-Poisson geometry. In this
section we will show that, when both $s$ and $j$ are chosen, there
is an isomorphism between the corresponding Hamiltonian
categories:
\begin{equation}\label{eq:iso}
\M_s(\frakd,\frakg) \stackrel{\sim}{\longrightarrow}
\M_j(\frakd,\frakg).
\end{equation}

\subsection{The linear algebra of the equivalence}\label{subsec:linequiv}

We now recall the linear algebra underpinning the isomorphism
\eqref{eq:iso}, following \cite[Sec.~1]{ABM}.

We consider the following set-up: $V$ and $W$ are vector spaces,
and $J:V\to W$ is a linear map.  The vector space $W$ is equipped
with two transversal Dirac structures, i.e., two maximal isotropic
subspaces $L_W, C_W \subset \mathbb{W}:=W\oplus W^*$ such that
$L_W\cap C_W=\{0\}$. In particular, $\mathbb{W}=L_W\oplus C_W$.
Let us consider the following two sets of additional data:
\begin{enumerate}
\item[$i)$] A Dirac structure $L$ on $V$ so that $J:V\to W$ is a
strong Dirac map with respect to $L_W$:
$$
L_W=\{ (J(v),\beta)\;|\; (v,J^*\beta) \in L \},\;\; \;\;
\ker(J)\cap L\cap V=\{0\}.
$$

\item[$ii)$] A bivector $\pi\in \wedge^2 V$ and a linear map
$\rho_V: L_W\to V$ such that
\begin{equation}\label{eq:2cond}
J\circ \rho_V=\pr_W,\;\;\; \pi^\sharp\circ J^*= \rho_V\circ
\sigmav,
\end{equation}
where $\pr_W:\mathbb{W}\to W$ is the natural projection, and
$\overline{\sigma}:W^*\to C_W^*\cong L_W$ is the linear map dual
to $\pr_W|_{C_W}:C_W\to W$.
\end{enumerate}
Recall  that, given a strong Dirac map $J:(V,L)\to (W,L_W)$, there
is an induced linear map (see Section \ref{subsec:hamcat})
\begin{equation}\label{eq:rhoV}
\rho_V:L_W\to V, \;\;\;\; \rho_V(w,\beta)=v,
\end{equation}
where $v$ in uniquely defined by the properties $J(v)=w$, and
$(v,J^*\beta)\in L$. Also, a pair of transversal Dirac structures
$L, C \subset V\oplus V^*$ defines a bivector $\pi\in \wedge^2V$
(see Section~\ref{subsec:app4}) by
\begin{equation}\label{eq:piV}
\pi^\sharp=\pr_V\circ (\pr_V|_{C})^*.
\end{equation}

\begin{theorem}\label{thm:linequiv}
Let $L_W, C_W$ be transversal Dirac structures on $W$, and let
$J:V\to W$ be a linear map. The following holds:
\begin{enumerate}
\item Consider $L$ as in i), and let $C\subset V\oplus V^*$ be the
backward image of $C_W$ by $J$. Then $L, C$ are transversal Dirac
structures on $V$, and the induced map $\rho_V$ \eqref{eq:rhoV}
and bivector $\pi$ \eqref{eq:piV} satisfy \eqref{eq:2cond}.

\item Consider $\rho_V$ and $\pi$ as in ii). Then the image of the
map
$$
L_W\oplus V^*\to V\oplus V^*,\;  (l,\alpha)\mapsto
(\pi^\sharp(\alpha)+\rho_V(l), J^*\pr_{W^*}(l) + \alpha -
J^*\pr_{W^*}\rho_V^*(\alpha))
$$
is a Dirac structure $L$ on $V$ for which $J:(V,L)\to (W,L_W)$ is
a strong Dirac map.
\end{enumerate}
Moreover, these constructions are inverses of each other.
\end{theorem}
A proof of Thm.~\ref{thm:linequiv} can be found in
\cite[Sec.~1.8]{ABM}.

The correspondence in Theorem \ref{thm:linequiv} is also
functorial:

\begin{proposition}\label{prop:funct}

Let $J_i:(V,L_i)\to (W,L_W)$ be a strong Dirac map, $i=1,2$, and
let $f:V_1\to V_2$ be a linear map which is f-Dirac and satisfies
$J_1=J_2\circ f$. Then $f\circ \rho_{V_1}=\rho_{V_2}$ and
$f_*\pi_1=\pi_2$.

Conversely, suppose that $\pi_i\in \wedge^2V_i$ and
$\rho_{V_i}:L_W\to V_i$ satisfy \eqref{eq:2cond}, $i=1,2$. Then if
$f:V_1\to V_2$ is such that $f_*\pi_1=\pi_2$, $J_2\circ f=J_1$ and
$f\circ \rho_{V_1}=\rho_{V_2}$, then $f$ is an f-Dirac map.
\end{proposition}

\begin{proof}
For the first part, the property $f\circ \rho_{V_1}=\rho_{V_2}$ is
a direct consequence of $f$ being an f-Dirac map. Let $C_i$ be the
backward image of $C_W$ by $J_i$, $i=1,2$. Then $L_2$ is the
forward image of $L_1$ by $f$ while $C_1$ is the backward image of
$C_2$ by $f$, using that $J_1=J_2\circ f$. This implies that
$f_*\pi_1=\pi_2$, see \cite[Sec.~1.8]{ABM}.

To prove the second part, we use that $f\circ (\pi_1)^\sharp \circ
f^*=(\pi_2)^\sharp$ and $f\circ \rho_{V_1}=\rho_{V_2}$ to conclude
that
\begin{equation}\label{eq:bk1}
 (\pi_2)^\sharp(\alpha) + \rho_{V_2}(l) =
f\circ (\pi_1^\sharp f^*(\alpha)+  \rho_{V_1}(l)).
\end{equation}
On the other hand, the conditions $J_1^*=f^*J_2$ and
$\rho_{V_1}^*f^*=\rho_{V_2}^*$ imply that
\begin{equation}\label{eq:bk2}
f^*(J_2^*\pr_{W^*}(l) + \alpha -
J_2^*\pr_{W^*}\rho_{V_2}^*(\alpha))= J_1^*\pr_{W^*}(l) + f^*\alpha
- J_1^*\pr_{W^*}\rho_{V_1}^*(f^*\alpha).
\end{equation}
Note that \eqref{eq:bk1} and \eqref{eq:bk2} together say that
$L_2$ is contained in the forward image of $L_1$. Since both have
the same dimension, they must coincide, so $f$ is an f-Dirac map.
\end{proof}

\subsection{Combining splittings}\label{subsec:combine}

Let $(\frakd,\frakg)$ be a Manin pair, and fix splittings $s:TS\to
\frakd_S$ and $j:\gstar\to \frakd$. We denote by
$(\frakg,\cobr,\tri)$ the Lie quasi-bialgebra determined by $j$.
 We have the induced maps
$$
\sigma=\sigma_s:\frakg_S\to T^*S, \;\; \sigma=s^*\circ \iota
\;\;\;\mbox{ and }\;\;\; \sigmav=\sigmav_j:T^*S\to \frakg,
\;\;\sigmav=j^*\circ \rho_S^*,
$$
already considered in Sections \ref{sec:dirac} and
\ref{sec:quasip}. We have another map, which depends on both $s$
and $j$, given by
\begin{equation}\label{eq:rhobar}
\rhov = \rhov_{s,j}: TS \longrightarrow \frakg_S,\;\;\;
\rhov=j^*\circ s.
\end{equation}
We represent the two short exact sequences associated with
$(\frakd,\frakg)$ and the maps defined by the splittings in the diagram below:
\begin{equation}\label{diag:gaugedual}
 \xymatrix
{ & \frakg_S \ar@<0ex>[d]_{\iota} \ar@<0.2ex>[dr]_{\rho} \ar@<0.2ex>[dl]^{\sigma} &\\
{T^*S} \ar@<0ex>[r]^{} \ar@<0.7ex>[ur]^{\sigmav}  & {\frakd_S}
\ar@<0.4ex>[r]^{\rho_S} \ar@<0.4ex>[d]_{}  &
\ar@<0.4ex>[l]^{s}\ar@<-0.9ex>[ul]_{\rhov}{TS}\\
& \ar@<0.4ex>[u]^{j} \frakg^*_S& }
\end{equation}
Decomposing $\rho_S$ and $s$ (and their duals) with respect to
$\frakd\cong \frakg\oplus \frakg^*$, we get:
\[
s= (\rhov, \sigma^*), s^{*}= (\sigma, \rhov^*), \rho_{S}= (\rho,
\sigmav^*), \rho_{S}^{*}= (\sigmav, \rho^*),
\]
where we always identify $\frakd\cong \frakd^*$ via the inner
product. Using that the vertical and horizontal sequences are short exact, one
obtains algebraic identities relating the various maps. In
particular, we have
\begin{equation}\label{eq:algid}
\sigma\rhov= -\rhov^*\sigma^*, \;\;\; \sigmav\rhov^*= -
\sigmav\rhov^*,\;\;\; \sigmav\sigma+ \rhov\rho= \Id_{\frakg},\;\;\;
\sigma\sigmav+ (\rho\rhov)^*= \Id_{T^*S}.
\end{equation}

Since $(\rho_S,s^*):\frakd_S\to \TS$ is an isomorphism of Courant
algebroids (where $\TS$ is equipped with the $\phi_S$-twisted
Courant bracket, and $\phi_S$ is given by \eqref{eq:phis2}), we
immediately obtain (besides \eqref{eq:IM2}) the differential-geometric
identities:
\begin{eqnarray}
\label{eq-comb1} \mathcal{L}_{\rho(v)}(\rhov^*\mu)-
i_{\sigmav^*\mu}d\sigma(v)+ i_{\rho(v)\wedge
\sigmav^*\mu}(\phi_S)&=& \sigma(i_{\mu}(F(v)))+
\rhov^*(\ad^*_v(\mu)),\\
 \label{eq-comb2}
\mathcal{L}_{\sigmav^*\mu}(\rhov^*\nu)-
i_{\sigmav^*\nu}d(\rhov^*\mu) + i_{\sigmav^*\mu\wedge
\sigmav^*\nu}(\phi_S)&=& \sigma(i_{\mu\wedge\nu}(\tri))+
\rhov^*(F^{*}(\mu, \nu)),
\end{eqnarray}
where $u,v \in \frakg$, $\mu, \nu \in \gstar$.

We know that $s$ defines an isomorphism $\frakd_S\cong \TS$,
and the Dirac structure $L_S$ on $S$ is just $\frakg$ under this identification.
The additional splitting $j$ defines an isotropic complement $\frakg^* \subset \frakd$
to $\frakg$, and we let $C_S$ be its image in $\TS$:
\begin{equation}\label{eq:CS}
C_S := \{((\rho_S\circ j)(\mu),(s^* \circ j)(\mu)),\; \mu \in
\gstar \}.
\end{equation}
Note that $C_S$ is an almost Dirac structure which depends both on
$s$ and $j$. The pair $L_S, C_S \subset \TS$ defines a Lie
quasi-bialgebroid structure, which determines elements $\chi_S\in
\Gamma(\wedge^3L_S)$ and $F_S^*:\Gamma(C_S)\wedge \Gamma(C_S) \to
\Gamma(C_S)$ (see Sec.~\ref{subsec:app3}). Under the
identification $\gstar_S \cong C_S$, it is clear that
$$
\tri_S(\mu_1,\mu_2,\mu_3)= \chi(\mu_1,\mu_2,\mu_3), \;\;\;
\cobr_S^*(\mu_1,\mu_2)=F^*(\mu_1,\mu_2),
$$
where $\mu_1,\mu_2,\mu_3 \in \gstar$ (viewed as constant sections
of $C_S\cong (L_S)^*$), and this completely determines $\cobr_S$
and $\tri_S$. The bivector field associated with $L_S, C_S$ (as in
Sec.~\ref{subsec:app4}) is just $\pi_S$ \eqref{eq:piS}.


\subsection{The equivalence theorem}

Let $(\frakd,\frakg)$ be a Manin pair with fixed $s:TS\to
\frakd_S$ and $j:\gstar\to \frakd$, as in Section
\ref{subsec:combine}. As we saw, $S=D/G$ inherits a Dirac
structure $L_S$ and an almost Dirac complement $C_S$. We now use
the linear construction of Section~\ref{subsec:linequiv} pointwise
to define the isomorphism \eqref{eq:iso}.

Given a Hamiltonian quasi-Poisson $\frakg$-space
$(M,\pi,\rho_M,J)$ (with respect to the Lie quasi-bialgebra
$(\frakg,\cobr,\tri)$ defined by $j$), let us consider the maps
\begin{eqnarray}
&\widehat{\rho}_M:\frakg_M=\frakg\times M \rmap \TM, &\; v\mapsto
(\rho_M(v),J^*\sigma(v)),\label{eq:rhohat}\\
&h:T^*M \rmap \TM,\; & \alpha\mapsto
(\pi^\sharp(\alpha),(1-T^*)\alpha),\label{eq:h}
\end{eqnarray}
where $T=\rho_M\overline{\rho}(dJ):TM\to TM$. The next result
generalizes \cite[Thm.~3.16]{BC}, using the techniques in
\cite{ABM}.

\begin{theorem}\label{thm:equiv1}
The following holds:
\begin{enumerate}
\item[i)] Let $J: (M,L)\to (S,L_S)$ be a strong Dirac map, and let
$\rho_M:\frakg\to \X(M)$ be the induced $\frakg$-action (as in
Prop.~\ref{cor:act}). Then the pull-back image $C\subset \TM$ of
$C_S$ under $J$ is a smooth almost Dirac structure transversal to
$L$, and the bivector field $\pi \in \X^2(M)$ associated with $L$
and $C$ is such that $(M,\pi,\rho_M,J)$ is a Hamiltonian
quasi-Poisson $\frakg$-space.

\item[ii)] Let $(M,\pi,\rho_M,J)$ be a Hamiltonian quasi-Poisson
$\frakg$-space and consider the maps $\widehat{\rho}_M$ and $h$
from \eqref{eq:rhohat} and \eqref{eq:h}. Then
\begin{equation}\label{eq:L}
L: =\{\widehat{\rho}_M(v) + h(\alpha)\;|\; v\in \frakg_M,\; \alpha
\in T^*M\} \subset \TM,
\end{equation}
is a Dirac structure for which $J:(M,L)\to (S,L_S)$ is a strong
Dirac map.
\end{enumerate}
Moreover, one construction is the inverse of the other.
\end{theorem}

\begin{proof}
Let us prove $i)$. The fact that the backward image $C$ of $C_S$
under $J$ is smooth and transversal to $L$ is shown in
\cite[Sec.~2.3]{ABM}. Hence the pair $L, C \subset \TM$ defines a
Lie quasi-bialgebroid over $M$. We denote the associated 3-tensor
by $\chi_M\in \Gamma(\wedge^3L)$ and cobracket by
$F_M:\Gamma(L)\to \Gamma(L)\wedge\Gamma(L)$, and let $d_C$ be the
degree-1 derivation on $\Gamma(\wedge L)$ defined by $C\cong L^*$
(see Sec.~\ref{subsec:app2}).

\begin{lemma}\label{lem:equ1} Let $(\frakg,\cobr,\tri)$ be the Lie quasi-bialgebra determined by
$j$. Then
\begin{equation}\label{eq:tricobr}
\widehat{\rho}_M(\tri)=\tri_M,\;\;\mbox{ and }\;\;
\widehat{\rho}_M(\cobr(v))=-d_C(\widehat{\rho}_M(v)), \;\;\forall
v\in \frakg,
\end{equation}
where $\widehat{\rho}_M:\wedge\frakg\to \Gamma(\wedge L)$ is the
extension of \eqref{eq:rhoh2} to exterior algebras.
\end{lemma}

\begin{proof}(of Lemma~\ref{lem:equ1})
The proof of the equation relating $\tri$ and $\tri_M$ can be
found in \cite[Sec.~2]{ABM}. We give here an alternative argument
which also proves the second equation in \eqref{eq:tricobr}.

Let us consider the bundle map \eqref{eq:rhoh1},
$\widehat{\rho}_M: J^*L_S \to L$, induced by $J$, and its dual
$\widehat{\rho}_M^*: C \to J^*C_S$, where we identify $L^*\cong C$
and $L_S^*\cong C_S$. It is clear from the definitions that
$(Y,\beta)\in J^*L_S$ is $J$-related to
$\widehat{\rho}_M(Y,\beta)$; similarly, given $(X,\alpha)\in C$ at
a point $x\in M$, there exists (a unique) $\mu \in \gstar$ such
that $\alpha=J^*s^*(\mu)$, and
$\widehat{\rho}_M^*(X,\alpha)=(\rho_S(\mu),s^*(\mu))$, so
$(X,\alpha)$ and $\widehat{\rho}_M^*(X,\alpha)$ are $J$-related.
Given a section $\zeta'$ of $C$ extending $(X,\alpha)$, then
$\widehat{\rho}_M^*(\zeta')$ is a section of $J^*C_S\cong
J^*\gstar_S$, but if $J$ has locally constant rank at $x$, we can
extend $\widehat{\rho}_M^*(\zeta')$ to a (local) section
$\zeta=\mu$ of $C_S\cong \gstar_S$, which is necessarily
$J$-related to $\zeta'$. It directly follows from Lemma
\ref{lem:Jrel} that
$\tri_S(\zeta_1,\zeta_2,\zeta_3)=\tri_M(\zeta_1',\zeta_2',\zeta_3')$,
which means that $\widehat{\rho}_M(\tri)=\tri_M$ at $x$.
Similarly,
\begin{eqnarray*}
\widehat{\rho}_M(\cobr(v))(\zeta_1'(x),\zeta_2'(x))
&=&\SPd{v,[\mu_1(J(x)),\mu_2(J(x))]_\frakd}\\
 &=&
\SPd{v,\Cour{\zeta_1,\zeta_2}_\frakd-\Lie_{\rho_S(\mu_1(J(x)))}\mu_2+
\Lie_{\rho_S(\mu_2(J(x)))}\mu_1}-
\SPd{\Lie_{\rho(v)}\mu_1,\mu_2}\\
&=&
\SPd{v,\Cour{\zeta_1,\zeta_2}_\frakd}-\Lie_{\rho_S(\mu_1(J(x)))}(\mu_2(v))
+ \Lie_{\rho_S(\mu_2(J(x))}(\mu_1(v)),
\end{eqnarray*}
where we used that $\gstar\subset \frakd$ is isotropic to conclude
that $\SPd{\Lie_{\rho(v)}\mu_1,\mu_2}=0$. On the other hand,
$$
d_C(\widehat{\rho}_M(v))(\zeta_1',\zeta_2') =
\Lie_{X_1}\SP{\zeta_2',\widehat{\rho}_M(v)} -
\Lie_{X_2}\SP{\zeta_1',\widehat{\rho}_M(v)}-
\SP{\widehat{\rho}_M(v),\Cour{\zeta_1',\zeta_2'}_{J^*\phi_S}},
$$
and the fact that
$\widehat{\rho}_M(\cobr(v))=-d_C(\widehat{\rho}_M(v))$ at $x$ is
again a direct consequence of Lemma \ref{lem:Jrel}. Since the
points $x\in M$ where $J$ has locally constant rank forms an open,
dense subset, we conclude that the equalities in
\eqref{eq:tricobr} hold everywhere in $M$.
\end{proof}

Let $\pi$ be the bivector field on $M$ associated with the Lie
quasi-bialgebroid defined by $L, C$ (as in
Sec.~\ref{subsec:app4}), and consider the $\frakg$-action $\rho_M$
induced by $J$. Then Lemma \ref{lem:equ1} and Prop.~\ref{prop:biv}
give
$$
\frac{1}{2}[\pi,\pi]=\pr_{TM}(\chi_M)=
\pr_{TM}(\widehat{\rho}_M(\tri))=\rho_M(\tri)
$$
and, for $v\in \frakg$,
$$
\Lie_{\rho_M(v)}\pi=\pr_{TM}(d_C(\widehat{\rho}_M(v)))=-\pr_{TM}(\widehat{\rho}(F(v)))=-\rho_M(F(v)).
$$
The moment map condition and the $\frakg$-equivariance of $J$
follow from Thm.~\ref{thm:linequiv}, part $1$. Hence
$(M,\pi,\rho_M,J)$ is a Hamiltonian quasi-Poisson $\frakg$-space,
finishing the proof of $i)$.

We now prove $ii)$. Given a Hamiltonian quasi-Poisson
$\frakg$-space $(M,\pi,\rho_M,J)$, it follows from part $2$ of
Thm.~\ref{thm:linequiv} that the subbundle $L\subset \TM$ defined
in \eqref{eq:L} is an almost Dirac structure, and $(dJ)_x:
\TM_x\to \TS_{J(x)}$ is a strong Dirac map relative to $L$ and
$L_S$ at all $x\in M$. We let $C$ be the almost Dirac structure
given by the pullback image of $C_S$ under $J$. Then $L$ and $C$
are transversal almost Dirac structures (see e.g.
\cite[Sec.~1.7]{ABM}). A direct computation shows that the bundle
map $h$ of \eqref{eq:h} agrees with the dual of $\pr_{TM}|_C:C\to
TM$,
\begin{equation}\label{eq:hproj}
h= (\pr_{TM}|_C)^*:T^*M\to C^*\cong L,
\end{equation}
and hence the bivector field associated with $L, C$, defined by
$(\pr_{TM}|_L)\circ (\pr_{TM}|_C)^*:T^*M\to TM$, is $\pi$.  To
prove $ii)$, it remains to check that $L$ is a $J^*\phi_S$-twisted
Dirac structure, i.e., that the associated 3-tensor $\chi_M'\in
\Gamma(\wedge^3C)$,
$\chi_M'(l_1,l_2,l_3)=\SP{\Cour{l_1,l_2}_{J^*\phi_S},l_3}$,
vanishes for all $l_1,l_2,l_3 \in \Gamma(L)$.

\begin{lemma}\label{lem:equ2}
Let $\Cour{\cdot,\cdot}$ denote the $J^*\phi_S$-twisted Courant
bracket on $\TM$. Then
\begin{align}
&\Cour{\widehat{\rho}_M(u),\widehat{\rho}_M(v)}=\widehat{\rho}_M([u,v]),
\;\;\mbox{ for }\; u,v\in \frakg.\label{eq:brc1}\\
&\SP{\Cour{h(\alpha_1),h(\alpha_2)},h(\alpha_3)}=0,\;\;\mbox{ for
}\;\alpha_i \in \Omega^1(M),
\;i=1,2,3. \label{eq:brc2}\\
&\Cour{\widehat{\rho}_M(v),h(\alpha)}=
-\widehat{\rho}_M(i_{\rho_M^*(\alpha)}F(v))+h(\Lie_{\rho_M(v)}\alpha),\;\;\mbox{
for }\; \alpha\in \Omega^1(M),\; v\in \frakg.\label{eq:brc3}
\end{align}
\end{lemma}

\begin{proof}(of Lemma~\ref{lem:equ2})
To prove \eqref{eq:brc1}, we have to show that
\[ J^*\sigma([u, v])  = \mathcal{L}_{\rho_{M}(u)}(J^*\sigma(v))-
                                      i_{\rho_{M}(v)}d(J^*\sigma(u))+
                                      i_{\rho_{M}(u)\wedge
                                      \rho_{M}(v)}(J^*\phi_{s}).
\]
Using the equivariance of $J$, $dJ\circ \rho_M=\rho$, we see that
this equation is just the pull-back by $J$ of condition
\eqref{eq:IM2} for $\sigma$.

To prove \eqref{eq:brc2}, we use \eqref{eq:hproj} to see that
\eqref{eq:brc2} is equivalent to the condition
$\pr_{TM}(\chi_M')=0$. A computation as in Prop. \ref{prop:biv}
(see \cite[Sec.~2]{ABM} for an alternative argument) shows that
the bivector field associated with the transversal almost Dirac
structures $L$ and $C$, which is just $\pi$, satisfies
$\frac{1}{2}[\pi,\pi]=\pr_{TM}(\chi_M) + \pr_{TM}(\chi_M')$. It
follows that $\pr_{TM}(\chi_M')=0$ since, by assumption,
$\frac{1}{2}[\pi,\pi]=\rho_M(\chi)=\pr_{TM}(\chi_M)$.

We now prove equation \eqref{eq:brc3}. The $TM$-component of this
equation gives
\begin{equation}\label{eq:claim3TM}
\Lie_{\rho_M(v)}\pi^\sharp(\alpha)=
\rho_M(-i_{\rho_M^*(\alpha)}F(v))+\pi^\sharp(\Lie_{\rho_M(v)}\alpha).
\end{equation}
Using the condition $\Lie_{\rho_M(v)}\pi=-\rho_M(F(v))$, we see
that the identity
$$
\Lie_{\rho_M(v)}(\pi(\alpha,\beta))=(\Lie_{\rho_M(v)}\pi)(\alpha,\beta)+
\pi(\Lie_{\rho_M(v)}\alpha,\beta)+\pi(\alpha,\Lie_{\rho_M(v)}\beta),
$$
can be re-written as
$$
\Lie_{\rho_M(v)}(\pi(\alpha,\beta))=
-i_\beta\rho_M(i_{\rho_M^*(\alpha)}F(v))+i_\beta\pi^\sharp(\Lie_{\rho_M(v)}\alpha)+
i_{\pi^\sharp(\alpha)}\Lie_{\rho_M(v)}\beta.
$$
Using the identity
$i_{\pi^\sharp(\alpha)}\Lie_{\rho_M(v)}\beta=-i_{\beta}{\Lie_{\rho_M(v)}\pi^\sharp(\alpha)}
+ \Lie_{\rho_M(v)}i_\beta\pi^\sharp(\alpha)$ (which is an
application of the general identity
$i_{[X,Y]}=\Lie_Xi_Y-i_Y\Lie_X$), equation \eqref{eq:claim3TM}
immediately follows.

The $T^*M$-component of equation \eqref{eq:brc3} is equivalent to
\begin{eqnarray}
T^*\mathcal{L}_{\rho_M(v)}(\alpha)-
\mathcal{L}_{\rho_M(v)}(T^*\alpha)&=&i_{\pi^{\sharp}\alpha}J^*d(\sigma(v))
-J^*\sigma(i_{\rho_{M}^{*}\alpha}F(v))- i_{\rho_M(v)\wedge
\pi^{\sharp}(\alpha)}(J^*\phi_{S})\nonumber \\
&=& J^*(-i_{\sigmav^*\mu}d\sigma(v) - \sigma(i_\mu
\cobr(v))+i_{\rho(v)\wedge\sigmav\mu}\phi_S),
 \label{eq:claim3T*M}
\end{eqnarray}
where, for the second equality, we used the $\frakg$-equivariance
of $\rho_M$, the moment map condition
$dJ\pi^\sharp=-\sigmav^*\rho_M^*$ and the notation $\mu=
\rho_{M}^{*}(\alpha)\in C^{\infty}(M, \frakg^*)$.

Evaluating the left-hand side of \eqref{eq:claim3T*M} on a vector
field $X \in \X(M)$, we obtain
\begin{eqnarray*}
-\SP{\alpha,[\rho_M(v),T(X)]} + \SP{\alpha,T([\rho_M(v),X])}&=&
-\SP{\alpha, \rho_M([v,\rhov dJ(X)]+\Lie_{\rho(v)}(\rhov
dJ(X)))}\\
&& + \SP{\mu,\rhov dJ([\rho_M(v),X])},
\end{eqnarray*}
where we have used that $\rho_M$ preserves the Lie algebroid
bracket on $\frakg_S$. So, at each point, \eqref{eq:claim3T*M}
evaluated at $X$ becomes:
\begin{align*}
\SP{\mu,-[v,\rhov (dJ(X))]_\frakg - \Lie_{\rho(v)}\rhov dJ(X) +
\rhov (dJ ([\rho_M(v),X]))} =\\
 \SP{-i_{\sigmav^*\mu}d\sigma(v) -
\sigma(i_\mu F(v))+i_{\rho(v)\wedge\sigmav(\mu)}\phi_S ,dJ(X)}.
\end{align*}
Since this equation makes sense for all $\mu$ and is
$C^\infty(M)$-linear on $\mu$, it suffices to assume $\mu\in
\gstar$ to be constant in order to prove this identity. Using
\eqref{eq-comb1}, the identity to be proven becomes:
\begin{equation}\label{eq:tobe}
\SP{\mu,-[v,\rhov dJ(X)]_\frakg - \Lie_{\rho(v)}\rhov dJ(X) +
\rhov dJ ([\rho_M(v),X])} =
\SP{\rhov^*\ad_v^*\mu-\Lie_{\rho(v)}(\rhov^*(\mu)),dJ(X)}.
\end{equation}
Let us consider the left-hand side of \eqref{eq:tobe}. Noticing
that $\rhov \in \Omega^1(S,\frakg)$, we have the identity
$$
J^*\rhov ([\rho_M(v),X])= \Lie_{\rho_M(v)}J^*\rhov (X) -
\Lie_{X}J^*\rhov(\rho_M(v)) - d(J^*\rhov)(\rho_M(v),X).
$$
Using that $dJ(\rho_M(v))=\rho(v)$, it follows from this identity
that the left-hand side of \eqref{eq:tobe} can be re-written as
$$
\SP{\mu, -[v,\rhov dJ(X)]_\frakg - \Lie_{dJ(X)}\rhov (\rho(v)) -
d\rhov (\rho(v),dJ(X))},
$$
from where it becomes clear that it depends on $dJ(X)$ only
pointwise, not locally. In particular, it makes sense to replace
$dJ(X)$ by an arbitrary vector field $V$ on $S$. So in order to
prove \eqref{eq:tobe}, it suffices to prove the identity
\begin{equation}\label{eq:tobe2}
\SP{\mu, -[v,\rhov(V)]_\frakg - \Lie_{V}\rhov (\rho(v)) - d\rhov
(\rho(v),V)}= \SP{\rhov^*\ad_v^*\mu-\Lie_{\rho(v)}(\rhov^*\mu),V},
\end{equation}
for all $V\in \mathfrak{X}(S)$. Now note that
$\SP{\rhov^*\ad_v^*(\mu),V}=\SP{\mu,[\rhov(V),v]_\frakg}$ and
\begin{eqnarray}
\SP{\Lie_{\rho(v)}(\rhov^*\mu),V}&=&\Lie_{\rho(v)}\SP{\mu,\rhov(V)}-
\SP{\rhov^*(\mu),[\rho(v),V]}\nonumber\\
&=&\SP{\mu,\Lie_{\rho(v)}\rhov(V)-\rhov([\rho(v),V])}\nonumber\\
&=&\SP{\mu, \Lie_V\rhov(\rho(v)) +d\rhov (\rho(v),V)},
\end{eqnarray}
where for the last equality we used that
$d\rhov(U,V)=\Lie_U\rhov(V)-\Lie_V\rhov(U)-\rhov([U,V])$. Now
\eqref{eq:tobe2} follows directly.
\end{proof}

To conclude that $L$ is integrable with respect to the
$J^*\phi_S$-twisted Courant bracket, we must check that
$\chi_M'(l_1,l_2,l_3)=\SP{\Cour{l_1,l_2},l_3}$ vanishes for all
$l_1,l_2,l_3 \in \Gamma(L)$. Clearly, it suffices to check this
condition when each $l_i$ is of the form $\widehat{\rho}(v_i)$ or
$h(\alpha_i)$, for $v_i\in \frakg$ and $\alpha_i \in T^*M$.

>From \eqref{eq:brc1}, we obtain that  $\SP{\Cour{l_1,l_2},l_3}=0$
if any two of the $l_i$'s are of the form $\widehat{\rho}_M(v_i)$.
Equation \eqref{eq:brc2} gives $\SP{\Cour{l_1,l_2},l_3}=0$ when
each $l_i$ is of the form $h(\alpha_i)$. The case where only two
of the $l_i's$ are of type $h(\alpha_i)$ follows from
\eqref{eq:brc3}. This concludes the proof of part $ii)$ of
Thm.~\ref{thm:equiv1}.
\end{proof}

The constructions in parts $i)$ and $ii)$ of Thm.~\ref{thm:equiv1}
are functorial as a consequence of Prop.~\ref{prop:funct}. In
particular, Thm.~\ref{thm:equiv1}, part $i)$, defines a functor
\begin{equation}\label{eq:Ifunc}
\I:\M_s(\frakd,\frakg) \to \M_ j(\frakd,\frakg),
\end{equation}
which establishes the desired isomorphism of Hamiltonian
categories; its inverse is given by the functor
$\M_j(\frakd,\frakg)\to \M_s(\frakd,\frakg)$ constructed in part
$ii)$.

We have the following characterization of the quasi-Poisson
bivector field $\pi$ constructed in Thm.~\ref{thm:equiv1}, part
$i)$ (c.f. \cite[Prop.~3.20]{BC}):

\begin{corollary} Let $J: (M,L)\to (S,L_S)$ be a strong Dirac map, and
let $\rho_M:\frakg\to \X(M)$ be the induced $\frakg$-action. The
associated quasi-Poisson bivector field $\pi$ is uniquely
determined by the following conditions: given $\alpha\in T^*M$,
then
\begin{equation}\label{eq:picond}
dJ(\pi^\sharp(\alpha))=-\sigmav^*\rho_M^*(\alpha),\;\;\;\;
(\pi^\sharp(\alpha),(\id - T^*)\alpha)\in L.
\end{equation}
\end{corollary}

\begin{proof}
The first condition in \eqref{eq:picond} is just (the dual of) the
moment map condition for the quasi-Poisson action, whereas the
second condition is just saying that $L$ contains the image of $h$
given by \eqref{eq:h}. These conditions uniquely define
$\pi^\sharp(\alpha)$ as a direct consequence of $\ker(dJ)\cap L
\cap TM =\{0\}$.
\end{proof}

\subsection{Properties and examples}\label{subsec:propex}

We keep considering a Manin pair $(\frakd,\frakg)$ together with
the choice of splittings $s$ and $j$. We now discuss several
properties of the functor $\mathcal{I}$ given in \eqref{eq:Ifunc}.

\subsubsection*{Foliations}
Given a Hamiltonian quasi-Poisson $\frakg$-space
$(M,\pi,\rho_M,J)$, its associated Dirac structure $L$ is given by
\eqref{eq:L}. The presymplectic foliation of $L$ is tangent to the
distribution
$$
\pr_{TM}(L)=\{ \rho_M(v)+\pi^\sharp(\alpha), \; v\in \frakg,\,
\alpha \in T^*M \} = \mathrm{Im}(r),
$$
where $r$, given in \eqref{eq:rmap}, is the anchor of the Lie
algebroid associated with the quasi-Poisson action. In other
words, the presymplectic foliation of $(M,L)$ coincides with the
orbit foliation of the Lie algebroid of the quasi-Poisson
structure. In particular, the functor $\mathcal{I}$ takes
presymplectic realizations to quasi-Poisson spaces with transitive
Lie algebroids:

\begin{corollary}\label{cor:I}
The functor $\mathcal{I}$ restricts to an isomorphism of
subcategories
$$
\mathcal{I}:\Mp_s(\frakd,\frakg) \stackrel{\sim}{\longrightarrow}
\Mp_j(\frakd,\frakg).
$$
\end{corollary}

\begin{example}
We saw that $S$ has a bivector field $\pi_S$,
$\pi_S^\sharp=\rho\sigmav_j$ (depending on $j$) which makes it
into a Hamiltonian quasi-Poisson $\frakg$-space with respect to
the dressing action and with $J=\id$ as moment map. It is easy to
check from the explicit formula \eqref{eq:L} that the associated
Dirac structure is just $L_S$ defined in \eqref{eq:LS}. Moreover,
the functor $\mathcal{I}$ takes each dressing orbit
$(\mathcal{O},\omega_\mathcal{O})$, viewed as a presymplectic leaf
of $L_S$, to $(\mathcal{O},\pi_\mathcal{O})$, where
$\pi_\mathcal{O}$ is the restriction of $\pi_S$ to $\mathcal{O}$.
\end{example}

\subsubsection*{Trivial equivalences}
Given a Hamiltonian quasi-Poisson $\frakg$-space
$(M,\pi,\rho_M,J)$, it may happen that the graph of $\pi$ already
defines a Dirac structure, in such a way that the functor
$\mathcal{I}$ is just the identity. From \eqref{eq:L}, we see that
this is the case if and only if the following two conditions hold:
\begin{equation}\label{eq:idcond}
\pi^\sharp\circ (J^* \sigma) = \rho_M,\;\; \mbox{ and }\;\;
\pi^\sharp \circ T^* = 0.
\end{equation}

\begin{example}
Let us consider the $G^*$-valued moment maps of
Example~\ref{ex:G*}. By \eqref{eq:sigmaG*} and
\eqref{eq:sigmavG*}, we know that $\sigma_s=\sigmav_j^{-1}$, so
the moment map condition \eqref{eq:mmapcond} is exactly the first
equation in \eqref{eq:idcond}. In this example, $\rhov=0$ (hence
$T=0$), so the second condition in \eqref{eq:idcond} is also
fulfilled. So the functor $\mathcal{I}$ produces no changes on the
geometrical structures, as already remarked in Example
\ref{ex:quasiG*}.
\end{example}

Note that the conditions in \eqref{eq:idcond} do \emph{not} hold
in the case of $G$-valued moment maps; in this case, the functor
$\mathcal{I}$ is nontrivial,  and the correspondence it
establishes recovers \cite[Thm.~10.3]{AKM} and
\cite[Thm.~3.16]{BC}.

\subsubsection*{Dependence on splittings}
The functor $\mathcal{I}$ is determined by the splittings $s$ and
$j$, and we write $\mathcal{I}^{(s,j)}$ to make this dependence
explicit. Let $s'$ be another connection splitting, and consider
the 2-form $B\in \Omega^2(S)$, defined in \eqref{eq:B}, and the
associated gauge transformation functor $\mathcal{I}_B$ of
Prop.~\ref{prop:gauge2}. Similarly, given another splitting $j'$,
let $t\in \wedge^2\frakg$ be the associated twist:
$t^\sharp=j-j':\gstar\to \frakg$. Then we have the functor
$\mathcal{I}_t$ of Prop.~\ref{prop:twist}.

\begin{proposition}
The dependence of the functor $\mathcal{I}$ on the choice of
splittings is as follows:
$$
\mathcal{I}^{(s,j)}=\mathcal{I}^{(s',j)}\circ \mathcal{I}_B,\;\;
\mathcal{I}^{(s,j')}=\mathcal{I}_t\circ \mathcal{I}^{(s,j)}.
$$
\end{proposition}

\begin{proof}
It follows from the definition of $B$ that
$\sigma^s-\sigma^{s'}=-i_{\rho_S(v)}B$, hence
$$
L_S^{s'}=\tau_B(L_S^s),\; \mbox{ and }\;
C_S^{s',j}=\tau_B(C_S^{s,j}).
$$
On the other hand, the pull-back images of $C_S^{s,j}$ and
$\tau_B(C_S^{s,j})$ under $J$, denoted by $C$ and $C'$, satisfy
$C'=\tau_{J^*B}(C)$. A direct computation shows that the bivector
field associated with $L,C$ is the same as the bivector field
associated with $\tau_{J^*B}(L), \tau_{J^*B}(C)$, and this proves
that $\mathcal{I}^{(s,j)}=\mathcal{I}^{(s',j)}\circ
\mathcal{I}_B$.

For the second identity, we note that if $t\in \wedge^2\frakg\cong
\wedge^2 L_S$ is a twist relating $C_S^{s,j}$ and $C_S^{s,j'}$,
then the twist relating their pull-back images under $J$ is
$\widehat{\rho}_M(t)$. The result now follows from part $3$ of
Prop.~\ref{lem:split}.
\end{proof}

\subsubsection*{Hamiltonian vector fields and reduction}
We now discuss the behavior of Hamiltonian vector fields and
reduced spaces under the equivalence functor $\mathcal{I}$.

Given a Dirac manifold $(M,L)$, we call a smooth function $f$ on
$M$ \textbf{admissible} \cite{Cou90} if there exists a vector
field $X\in \X(M)$ such that $(X,df)\in L$. In this case $X$ is a
\textbf{Hamiltonian vector field} for $f$, though $X$ is not
uniquely defined by this property in general. The set of
admissible functions is a Poisson algebra, with Poisson bracket
$\{f,g\}_L:=\Lie_{X_f}g$, where $X_f$ is any Hamiltonian vector
field for $f$. We now consider Hamiltonian spaces in
$\M_s(\frakd,\frakg)$.

\begin{proposition}\label{prop:hamvf}
Let $J:(M,L)\to (S,L_S)$ be a strong Dirac map, let $\rho_M$ be
the induced $\frakg$-action. Then any $\frakg$-invariant function
$f$ is admissible and has a distinguished Hamiltonian vector field
$X_f$ uniquely determined by the extra condition $dJ(X_f)=0$. In
particular, $C^\infty(M)^\frakg$ is a Poisson algebra.
\end{proposition}

\begin{proof}
Using the isotropic splitting $j$ of $(\frakd,\frakg)$, let $\pi$
be the quasi-Poisson bivector associated with $L$ and $j$ via
$\mathcal{I}$. If $f$ is $\frakg$-invariant, then $T^*df=0$, so
the vector field $X_f:=\pi^\sharp(df)\in \mathfrak{X}(M)$
satisfies $h(df)=(X_f,df)\in L$ (where $h$ is defined in
\eqref{eq:h}), i.e., it is a Hamiltonian vector field. Also,
$dJ(X_f)=-\sigmav^*\rho_M^*(df)=0$. Finally, note that there is at
most one vector field with these properties, since $\ker(L)\cap
\ker(dJ)=\{0\}$ (in particular, $X_f$ is \textit{independent} of
the splitting $j$ defining $\pi$). If $f$ and $g$ are
$\frakg$-invariant, property \eqref{cond:q-poiss2} for the
quasi-Poisson bivector field $\pi$ directly implies that the
function $\{f,g\}_L=\Lie_{X_f}g= \pi(df,dg)$ is again
$\frakg$-invariant, so $C^\infty(M)^\frakg$ is a Poisson algebra.
\end{proof}

It immediately follows from the previous proof that the Poisson
algebra of Prop.~\ref{prop:hamvf} (using Dirac geometry) agrees
with the one of Sec.~\ref{subsec:hamcat2} (using quasi-Poisson
geometry). The previous proposition recovers \cite[Prop.~4.6]{AMM}
in the case of $G$-valued moment maps.

As we have discussed, one can perform moment map reduction either
in the framework of Hamiltonian quasi-Poisson spaces or Dirac
geometry. We observe that the functor $\mathcal{I}$ preserves the
reduction procedures:

\begin{proposition}\label{prop:Icomm}
The functor $\mathcal{I}$ commutes with moment map reduction.
\end{proposition}

\begin{proof}
Let $(M,\pi,\rho_M,J)$ be the Hamiltonian quasi-Poisson $G$-space
associated with a strong Dirac map $J$ via $\mathcal{I}$. Let us
fix a dressing orbit $\mathcal{O}$ in $S$, and a point $y\in
\mathcal{O}$ which is regular for $J$. As we saw in
Prop.~\ref{prop:qpoissonred}, if $f,g\in
C^\infty(J^{-1}(\mathcal{O}))^G$, then we have a well-defined
Poisson bracket
$\{f,g\}_\pi:=\pi(d\widetilde{f},d\widetilde{g})|_{J^{-1}(\mathcal{O})}$,
independent of the extensions $\widetilde{f},\widetilde{g}$ of $f$
and $g$. Since $\pi^\sharp(d\widetilde{f})|_{J^{-1}(\mathcal{O})}$
does not depend on the extension $\widetilde{f}$ of $f$ and lies
in the kernel of $dJ$, it gives a well-defined vector field $X_f$
on $J^{-1}(\mathcal{O})$ (which is tangent to $J^{-1}(y)$).

Suppose that the isotropy subgroup of $y$, denoted by $G_y$, acts
freely and properly on $J^{-1}(y)$. We have a natural
identification $J^{-1}(\mathcal{O})/G\cong J^{-1}(y)/G_y$, which
gives an identification of $C^\infty(J^{-1}(\mathcal{O}))^G$ with
$C^\infty(J^{-1}(y))^{G_y}$. If $L'$ denotes the Dirac structure
on $J^{-1}(y)$ given by the pull-back image of $L$ under the
inclusion $\iota:J^{-1}(y)\hookrightarrow M$, then the Poisson
structure on $J^{-1}(y)/G_y$ given in Prop.~\ref{prop:diracred} is
defined by the identification of $C^\infty(J^{-1}(y))^{G_y}$ with
admissible functions of $L'$ \cite[Sec.~4.4]{BC}. Let $f\in
C^\infty(J^{-1}(y))^{G_y}\cong C^\infty(J^{-1}(\mathcal{O}))^G$
and $\widetilde{f}$ be any local extension of $f$ to $M$. Since
$\widetilde{f}$ is $\frakg$-invariant at each point on
$J^{-1}(\mathcal{O})$, it follows that
$(\pi^\sharp(d\widetilde{f}),d\widetilde{f})\in L$ over
$J^{-1}(\mathcal{O})$. By definition of backward image, it
directly follows that
$X_f=\pi^\sharp(d\widetilde{f})|_{J^{-1}(\mathcal{O})}$ (which is
tangent to $J^{-1}(y)$) satisfies
$(X_f,df=\iota^*d\widetilde{f})\in L'$. Hence $X_f$ is a
Hamiltonian vector field for $f$ with respect to $L'$. By
definition, we have
$$
\{f,g\}_{L'}=X_f.g=\{f,g\}_{\pi},
$$
which shows that we get the same reduced Poisson structure by
using Dirac reduction or quasi-Poisson reduction
\end{proof}

For $G$-valued moment maps with transitive Lie algebroids,
Prop.~\ref{prop:Icomm} recovers \cite[Prop.~10.6]{AKM}. Using
Prop.~\ref{prop:diracred}, part $iii)$, Cor.~\ref{cor:I} and
Prop.~\ref{prop:Icomm}, we see that the reduction of quasi-Poisson
spaces with transitive Lie algebroids in symplectic.

\subsubsection*{The double $(D,\omega_D)$}
Let us consider the Lie group $D$ equipped with the 2-form
$\omega_D = \omega_D^s$ given by \eqref{eq:omega_D}. As proven in
Thm.~\ref{thm:omegaD}, $(p,\overline{p}):D\to S\times S$ is a
strong Dirac map (i.e., a presymplectic realization), where
$S\times S$ is equipped with the product Dirac structure
$L_S\times L_S$. The choice of splitting $j$ of $(\frakd,\frakg)$
induces a splitting $j\times j$ of
$(\frakd\times\frakd,\frakg\times\frakg)$, and we know from
Thm.~\ref{thm:equiv1} that there is an associated bivector field
making $D$ into a Hamiltonian quasi-Poisson
$\frakg\times\frakg$-space with moment map $J=(p,\overline{p})$.
We consider the maps $\sigmav:T^*(S\times S)\to \frakg\times
\frakg$ and $\rhov: T(S\times S)\to \frakg\times \frakg$ (as in
Sec.~\ref{subsec:combine}) associated with the Manin pair
$(\frakd\times\frakd,\frakg\times\frakg)$ and the splittings
$s\times s$ and $j\times j$.

Let us consider the bivector field $\pi^j_D=\pi_D \in \X^2(D)$,
depending on $j$, given by
\begin{equation}
\pi_D(\alpha,\beta):= \SP{\alpha^\vee,\beta^\vee}_\frakd - (\rma^r
+ \rma^l)(\alpha,\beta),
\end{equation}
where $\alpha,\beta\in \Omega^1(D)$, $\alpha^\vee,\beta^\vee \in
\X(D)$ are the dual vector fields via $\SP{\cdot,\cdot}_\frakd$
and $\rma\in \frakd\otimes \frakd$ is the $r$-matrix of
Remark~\ref{rem:rmatrix}. Note that the skew symmetry of $\pi_D$
follows from \eqref{eq:rskew}.

\begin{proposition}\label{prop:piD}
The quasi-Poisson bivector field corresponding to $\omega_D$ via
$\mathcal{I}$ is $\pi_D$.
\end{proposition}
\begin{proof}
We must show that the two conditions in \eqref{eq:picond} hold,
i.e.,
\begin{eqnarray}
(dp,d\overline{p})_a(\pi^\sharp(\alpha))=-\sigmav^*\rho_D^*(\alpha),\label{eq:tbp1}\\
i_{\pi^\sharp(\alpha)}\omega_D=(1-T^*\alpha),\label{eq:tbp2}
\end{eqnarray}
$\forall \alpha \in T_aD$, where $\rho_D(u,v)=u^r-v^l$, $u, v \in
\frakg$, and $T=\rho_D \circ \rhov\circ (dp,d\overline{p}): TD\to
TD$.

It suffices to prove the equations for $\alpha=
(w^r)^\vee=\SPd{w,\theta^R_D}$, for $w \in \frakd$. Let us start
with the r.h.s. of \eqref{eq:tbp1}. To simplify the notation, we
always identify $\frakd\cong \frakd^*$ via $\SPd{\cdot,\cdot}$. A
direct computation shows that
\begin{equation}\label{eq:rhods}
\rho_D^*(\alpha_a)=(\iota^*(w),-\iota^*(\Ad_{a^{-1}}(w))), \;\;
a\in D,
\end{equation}
where $\iota^*$ is an in Section~\ref{subsec:quasip}. Using that
$\rho_S= dp \circ dr_a$, we find
\begin{equation}\label{eq:compare1}
-\sigmav^*\rho_D^*(\alpha_a)=(-dp (r_a j \iota^* (w)), dp(r_a j
\iota^*(\Ad_{a^{-1}}(w)))), \;\; \mbox{ where }\; \alpha_a=
(dr_a(w))^\vee.
\end{equation}
(We may use $r_a, l_a$ for $dr_a, dl_a$ in order to simplify the
notation.) On the other hand, a direct computation using the
definition of $\pi_D$ gives:
\begin{equation}\label{eq:pisharp}
\pi^\sharp(\alpha_a) = r_a(w) - l_aj\iota^*(\Ad_{a^{-1}}(w)) -
r_aj\iota^* w = l_a(\iota j^* \Ad_{a^{-1}}(w)) -r_a j\iota^*w,
\end{equation}
where we used that $\iota j^* + j \iota^* = 1$. Since
$dp_a(dl_a(u))=0$ and $dp_a(dr_a(u))=\rho_S(u)$ if $u\in \frakg$,
we obtain $dp_a(\pi^\sharp(\alpha_a))=-dp_a (r_a j\iota^* (w))$.
Similarly, one checks that
$d\overline{p}_a(\pi^\sharp(\alpha_a))=dp_a (r_a
j\iota^*(\Ad_{a^{-1}}w))$. Comparing with \eqref{eq:compare1},
\eqref{eq:tbp1} follows.

In order to prove \eqref{eq:tbp2}, it suffices to show that
\begin{equation}\label{eq:tbp3}
\omega_D(\pi^\sharp(\alpha_a),X_a)=\alpha_a(X_a)-\alpha_a(TX_a),
\end{equation}
for $\alpha_a=(dr_a(w))^\vee$ and $X_a=dr_a(v)$, where $w,v \in
\frakd$. Using the identity $j\iota^*=1-\iota j^*$ in
\eqref{eq:pisharp}, we get:
\begin{equation}\label{eq:pis2}
\pi_D^\sharp(\alpha_a)= l_a(\iota j^* \Ad_{a^{-1}}(w)) -r_a
j\iota^*w = r_a\iota j^* w - l_a j\iota^* \Ad_{a^{-1}}(w).
\end{equation}
Using \eqref{omegaD}, we find
\begin{eqnarray*}
\omega_D(\pi^\sharp(\alpha_a),X_a)&=&\SPd{-l_a\theta_a(X_a)+r_a(\theta_{a^{-1}}\Inv(X_a))+
X_a,\pi_D^\sharp(\alpha_a)}\\
&=&-\SPd{l_a\theta_a(r_a(v)),\pi_D^\sharp(\alpha_a)} -
\SPd{r_a\theta_{a^{-1}}l_{a^{-1}}v, \pi_D^\sharp(\alpha_a)} +
\SPd{r_a(v),\pi_D^\sharp(\alpha_a)},
\end{eqnarray*}
where we have used that $\Inv(r_a(v))=-l_{a^{-1}}v$. Using
\eqref{eq:pisharp} and that $\theta$ is isotropic, we have
$$
\SPd{l_a\theta_a(r_a(v)),\pi_D^\sharp(\alpha_a)} =
-\SPd{l_a\theta_a(r_a(v)),r_aj\iota^* w} =
-\SPd{\Ad_a\theta_a(r_a(v)), j\iota^*w},
$$
and, using \eqref{eq:pis2}, we similarly obtain
$$
\SPd{r_a\theta_{a^{-1}}l_{a^{-1}}(v), \pi_D^\sharp(\alpha_a)}
=-\SPd{\Ad_{a^{-1}}\theta_{a^{-1}}l_{a^{-1}}(v),
j\iota^*\Ad_{a^{-1}}(w)}.
$$
Using \eqref{eq:pisharp}, we get
$$
\SPd{r_a(v),\pi_D^\sharp(\alpha_a)}= -\SPd{w,\iota j^*v} -
\SPd{\iota j^* \Ad_{a^{-1}}(v),\Ad_{a^{-1}}(w)} +\SPd{v,w}.
$$
Combining the last three equations, we find that $\omega_D(\pi^\sharp(\alpha_a),X_a)$ equals
\begin{eqnarray}\label{eq:LHS}
&\SPd{\iota j^*\Ad_a\theta_a(r_a(v)), w}+
\SPd{\iota j^*\Ad_{a^{-1}}\theta_{a^{-1}}l_{a^{-1}}v, \Ad_{a^{-1}}(w)}+\\
& -\SPd{w,\iota j^*v} -\SPd{\iota j^* \Ad_{a^{-1}}(v),\Ad_{a^{-1}}(w)} +\SPd{v,w}.\nonumber
\end{eqnarray}

We now consider the r.h.s. of \eqref{eq:tbp3}. Using that
$d\overline{p}(r_a(v))=-dp(l_{a^{-1}}(v))$, we see that
$\rhov\circ dJ (X_a)= (j^*s,j^*s)\circ(dp,d\overline{p})(r_a(v))=
(j^* s dp(dr_a(v)),-j^*s dp(dl_{a^{-1}}(v)))$, so
$$
T(X_a)=\rho_D(\rhov\circ dJ (X_a))= r_a j^* s dp(r_a(v)) + l_a
j^*s dp (l_{a^{-1}}(v)).
$$
Using \eqref{eq:stheta} to express $s$ in terms of $\theta$, we get
\begin{eqnarray}\label{eq:Talpha}
\alpha_a(T(X_a))&=&(w, j^*s dp (r_a(v)) + \Ad_a(j^*s_{a^{-1}}dp(l_{a^{-1}}v)))\nonumber \\
&=& \SPd{w,j^*v} - \SPd{w,j^*\Ad_a\theta_a r_a(v)}+\SPd{\Ad_{a^{-1}}w,j^*\Ad_{a^{-1}}v} \nonumber\\
&& - \SPd{\Ad_{a^{-1}}w,j^*\Ad_{a^{-1}}\theta_{a^{-1}}l_{a^{-1}}v}.
\end{eqnarray}
Using that $\alpha_a(X_a)=\SPd{w,v}$ and \eqref{eq:Talpha}, we see that the r.h.s of \eqref{eq:tbp3}
agrees with \eqref{eq:LHS}, and this concludes the proof.
\end{proof}

In the case of $G$-valued moment maps, $\pi_D$ recovers the
quasi-Poisson structure on $G\times G$ of \cite[Ex.~5.3]{AKM}, and
the previous proposition generalizes \cite[Ex.~10.5]{AKM}.

A result analogous to Prop.~\ref{prop:piD}, relating the
presymplectic structure on the Lie groupoid $G\ltimes S$
(integrating $L_S$) to quasi-Poisson bivectors is discussed in
\cite{BIS}.



\appendix

\section{Appendix}

\subsection{Courant algebroids and Dirac structures}\label{subsec:app1}

A {\bf Courant algebroid} \cite{LWX} over a manifold $M$ is a
(real) vector bundle $E\to M$ equipped with the following
structure: a nondegenerate symmetric bilinear form
$\SP{\cdot,\cdot}$ on the bundle, a bundle map $\rho_E:E \to TM$
(called the \textit{anchor}) and a bilinear bracket
$\Cour{\cdot,\cdot}$ on $\Gamma(E)$, so that the following axioms
are satisfied:

\begin{enumerate}
\item[C1)] $\Cour{e_1,\Cour{e_2,e_3}} = \Cour{\Cour{e_1,e_2},e_3}
+ \Cour{e_2,\Cour{e_1,e_3}}$, \; $\forall \, e_1,e_2, e_3 \in
\Gamma(E)$;\label{axiom1}

\item[C2)] $\Cour{e,e}=\frac{1}{2}\mathcal{D}\SP{e,e}$, \;
$\forall e\in \Gamma(E)$, where $\mathcal{D}:C^\infty(M)\to
\Gamma(E)$ is defined by
$\SP{\mathcal{D}f,e}=\Lie_{\rho_E(e)}f$.\label{axiom2}

\item[C3)] $\Lie_{\rho_E(e)}\SP{e_1,e_2}=\SP{\Cour{e,e_1},e_2} +
\SP{e_1,\Cour{e,e_2}}$, \; $\forall e,e_1,e_2 \in
\Gamma(E)$;\label{axiom3}

\item[C4)] $\rho_E(\Cour{e_1,e_2})= [\rho_E(e_1),\rho_E(e_2)]$, \;
$\forall e_1,e_2 \in \Gamma(E)$;\label{axiom4}

\item[C5)] $\Cour{e_1,f e_2}= f\Cour{e_1,e_2} +
(\Lie_{\rho_E(e_1)}f)e_2$, \; $\forall \, e_1,e_2 \in \Gamma(E),\;
f\in C^\infty(M)$.\label{axiom5}

\end{enumerate}

Note that the bracket $\Cour{\cdot,\cdot}$ is {\it not}
skew-symmetric, but rather satisfies
\begin{equation}\label{eq:skewsym}
\Cour{e_1,e_2}=-\Cour{e_2,e_1} + \mathcal{D}\SP{e_1,e_2}
\end{equation}
as a consequence of C2). (This is the non-skew-symmetric version
of the Courant bracket studied, e.g., in \cite{Dima}; the original
notion of Courant bracket \cite{LWX} is obtained by
skew-symmetrization.) It also follows from C2) that, upon the
identification $E\cong E^*$ via $\SP{\cdot,\cdot}$, we have
\begin{equation}\label{eq:rhorho*}
\rho_E\circ \rho_E^*=0.
\end{equation}
The model example of a Courant algebroid is the following:

\begin{example}\label{ex:CA}
Consider $E=T^*M\oplus TM$ equipped with symmetric pairing
$\SP{(X,\alpha),(Y,\beta)}_{can}:=\beta(X)+\alpha(Y)$. Any closed
3-form $\phi$ on $M$ determines a Courant algebroid structure on
$E$ with bracket
$$
\Cour{(X,\alpha),(Y,\beta)}_\phi:=([X,Y], \Lie_X\beta-i_Yd\alpha +
i_Yi_X\phi).
$$
\end{example}

A detailed discussion about Courant brackets with original
references can be found in \cite{K-S2}.

A subbundle $L\subset E$ which is Lagrangian (i.e., maximal
isotropic) with respect to $\SP{\cdot,\cdot}$ is called an
\textbf{almost Dirac structure}. It is a \textbf{Dirac structure}
if it is \textit{integrable} in the sense that
$$
\Cour{\Gamma(L),\Gamma(L)}\subseteq \Gamma(L).
$$
For a Dirac structure $L$, \eqref{eq:skewsym} implies that the
restriction $[\cdot,\cdot]_L:=\Cour{\cdot,\cdot}|_{\Gamma(L)}$ is
a skew-symmetric bracket on $\Gamma(L)$, and axioms C1) and C5)
imply that this bracket makes $L$ into a \textit{Lie algebroid}
with anchor $\rho_L:=\rho_E|_L$. The bracket $[\cdot,\cdot]_L$ can
be extended to a bilinear bracket on $\Gamma(\wedge L)$,
$[\cdot,\cdot]_L:\Gamma(\wedge^p L)\times \Gamma(\wedge^q L)\to
\Gamma(\wedge^{p+q-1}L)$, by the conditions
\begin{align}
& [l_1,l_2]_L=-(-1)^{(p-1)(q-1)}[l_2,l_1]_L,\label{eq:sch2}\\
& [l_1, l_2 l_3]_L = [l_1,l_2]_L l_3 +
(-1)^{(p-1)q}l_2[l_1,l_3]_L, \label{eq:sch3}
\end{align}
for $l_1 \in \Gamma(\wedge^p L)$, $l_2\in \Gamma(\wedge^q L)$, and
$l_3\in \Gamma(\wedge^r L)$. The Jacobi identity on $\Gamma(L)$
translates into the graded Jacobi identity for the extended
bracket:
\begin{equation}\label{eq:grdjacobi}
[l_1,[l_2,l_3]_L]_L + (-1)^{(p-1)(q+r)}[l_2,[l_3,l_1]_L]_L +
(-1)^{(r-1)(p+q)}[l_3,[l_1,l_2]_L]_L =0,
\end{equation}
In other words, $\Gamma(\wedge L)$ becomes a \textit{Gerstenhaber
algebra}.

The bracket $[\cdot,\cdot]_L$ and anchor $\rho_L$ also define a
degree-1 derivation $d_L$ on the graded commutative algebra
$\Gamma(\wedge L^*)$,
\begin{equation}
d_L(\xi_1 \xi_2)=d_L(\xi_1)\xi_2 + (-1)^p \xi_1 d_L(\xi_2),
\end{equation}
for $\xi_1\in \Gamma(\wedge^p L^*)$ and $\xi_2\in \Gamma(\wedge^q
L^*)$, by the conditions:
\begin{align}
& d_Lf (l)= \Lie_{\rho_L(l)}f, \;\;\; l\in \Gamma(L), f\in
C^\infty(M);\label{eq:da1}\\
& d_L\xi
(l_1,l_2)=\Lie_{\rho_L(l_1)}\xi(l_2)-\Lie_{\rho_L(l_2)}\xi(l_1)-\xi([l_1,l_2]_L),
\;\;\; l_1,l_2 \in \Gamma(L), \xi \in \Gamma(L^*). \label{eq:da2}
\end{align}
In this case the Jacobi identity of $[\cdot,\cdot]_L$ translates
into $d_L^2=0$.

\subsection{Manin pairs over manifolds and isotropic splittings} \label{subsec:app2}

A \textbf{Manin pair over a manifold} $M$ is a pair $(E,L)$
consisting of a Courant algebroid $E$ over $M$ for which
$\SP{\cdot,\cdot}$ has signature $(n,n)$, and a Dirac structure
$L\subset E$. It follows from the signature condition that
$\mathrm{rank}(L)=n=\frac{1}{2}\mathrm{rank}(E)$. When $M$ is a
point, we recover the notion discussed in Section
\ref{subsec:manin}.

Given a Manin pair $(E,L)$ over $M$, there is an associated exact
sequence of vector bundles given by
\begin{equation}\label{eq:exact}
0\rmap L \stackrel{\iota}{\longrightarrow} E
\stackrel{\iota^*}{\longrightarrow} L^* \rmap 0,
\end{equation}
where $\iota:L \hookrightarrow E$ is the inclusion and
$\iota^*(e)(l)=\SP{e,\iota(l)}$. We consider henceforth the
identification $E \cong E^*$ induced by $\SP{\cdot,\cdot}$. The
map $\iota^*$ coincides with the projection $E \to E/L$ after the
identification $E/L\cong L^*$ induced by $\SP{\cdot,\cdot}$,
proving the exactness of the sequence \eqref{eq:exact}.

An {\bf isotropic splitting} of the exact sequence
\eqref{eq:exact} is a linear splitting $s:L^*\to E$ of
\eqref{eq:exact} whose image is isotropic in $E$, i.e.,
$\SP{\cdot,\cdot}|_{s(L^*)}=0$. A Manin pair over $M$ together
with the choice of an isotropic splitting is referred to as a
\textbf{split Manin pair} over $M$.

\begin{lemma}\label{lem:identif}
Let $(E,L)$ be a Manin pair over $M$.  Then the exact sequence
\eqref{eq:exact} admits an isotropic splitting. Moreover, any
isotropic splitting $s:L^*\to E$ defines an isomorphism
\begin{equation}\label{eq:isom}
(\iota,s):L\oplus L^* \stackrel{\sim}{\rmap} E
\end{equation}
with inverse $(s^*,\iota^*)$, which identifies the pairing
$\SP{\cdot,\cdot}$ in $E$ with the canonical symmetric pairing in
$L\oplus L^*$ given by
\begin{equation}
\SP{(l_1,\xi_1),(l_2,\xi_2)}_{can}:= \xi_2(l_1) + \xi_1(l_2).
\end{equation}
\end{lemma}

\begin{proof}
If $s:L^*\to E$ is any linear splitting of \eqref{eq:exact}, then
a direct computation shows that $ s'=s-\frac{1}{2}\iota s^* s $ is
an isotropic splitting. It is straightforward to check that
\eqref{eq:isom} is an isometric isomorphism with respect to
$\SP{\cdot,\cdot}_{can}$ and $\SP{\cdot,\cdot}$.
\end{proof}

An immediate consequence of Lemma \ref{lem:identif} is that the
following identities hold:
\begin{equation}\label{eq:idents}
s^*s=0,\;\; \iota^*s = 1,\;\; s^*\iota=1,\;\; s\iota^*+ \iota
s^*=1.
\end{equation}

Let us fix an isotropic splitting $s:L^*\to E$. Under the induced
identification $E\cong L\oplus L^*$, the maps $s^*$ and $\iota^*$
become the natural projections $\pr_L:L\oplus L^* \to L$ and
$\pr_{L^*}:L\oplus L^*\to L^*$, respectively. Then $s$ induces the
following geometrical structures:

\begin{itemize}

\item[$i)$] A cobracket
\begin{equation}\label{eq:cobr}
F_s: \Gamma(L)\to \Gamma(L)\wedge \Gamma(L),
\end{equation}

\item[$ii)$] A 3-tensor
\begin{equation}\label{eq:chi}
\chi_s \in \Gamma(\wedge^3L),
\end{equation}

\item[$iii)$] A bundle map
\begin{equation} \rho^s_{L^*}:=
\rho_E|_{s(L^*)}: L^* \to TM.
\end{equation}

\end{itemize}
We will omit the $s$ dependence in the notation whenever there is
no risk of confusion.

The cobracket $F$ is defined in terms of its dual,
$F^*:\Gamma(L^*)\wedge \Gamma(L^*)\to \Gamma(L^*)$, by
\begin{equation}\label{eq:F}
F^*(\xi_1,\xi_2):= \pr_{L^*}(\Cour{s(\xi_1),s(\xi_2)}).
\end{equation}
We also denote the skew-symmetric bracket $F^*$ on $\Gamma(L^*)$
by $[\cdot,\cdot]_{L^*}$ (the skew-symmetry of \eqref{eq:F} is a
consequence of $s(L^*)\subset E$ being isotropic).

Similarly, we define $\chi:\Gamma(L^*)\wedge \Gamma(L^*)\to
\Gamma(L)$ by the condition
$$
i_{\xi_2}i_{\xi_1}\chi = \pr_L(\Cour{s(\xi_1),s(\xi_2)}).
$$
By axiom C3) in the definition of a Courant algebroid, the
expression
$$
i_{\xi_3}\pr_L(\Cour{s(\xi_1),s(\xi_2)})=
\SP{\Cour{s(\xi_1),s(\xi_2)},s(\xi_3)}
$$
is skew-symmetric in $\xi_1,\xi_2,\xi_3$. Since it is clearly
$C^\infty(M)$-linear in $\xi_3$, it is $C^\infty(M)$-trilinear and
therefore defines \eqref{eq:chi}.

We also have an extension of $[\cdot,\cdot]_{L^*}$ to a bilinear
bracket on $\Gamma(\wedge L^*)$ satisfying \eqref{eq:sch2},
\eqref{eq:sch3} as well as a degree 1 derivation $d_{L^*}$ on
$\Gamma(\wedge L)$ defined by $[\cdot,\cdot]_{L^*}$ and
$\rho_{L^*}$ via \eqref{eq:da1}, \eqref{eq:da2}. In general,
$[\cdot,\cdot]_{L^*}$ does not satisfy the graded Jacobi identity
and $d_{L^*}$ is not a differential, as a consequence of the
failure of integrability of $L^*\subset E$.

\subsection{Lie quasi-bialgebroids} \label{subsec:app3}

Let us consider a split Manin pair, identified with $(L\oplus
L^*,L)$, where $L\oplus L^*$ is equipped with the symmetric
pairing $\SP{\cdot,\cdot}_{can}$ (as in Lemma~\ref{lem:identif}).
Fixing this identification, one obtains a formula for the Courant
bracket $\Cour{\cdot,\cdot}$ on $L\oplus L^*$ in terms of $F^*,
\chi$ and $\rho_{L^*}$:
\begin{align}
&\Cour{(l_1,0),(l_2,0)}=[l_1,l_2]_L, \label{cbr1}\\
&\Cour{(l,0),(0,\xi)}=(-i_\xi d_{L^*}l,\Lie_{l}\xi),\label{cbr2}\\
&\Cour{(0,\xi_1),(0,\xi_2)}=(\chi(\xi_1,\xi_2),F^*(\xi_1,\xi_2)),\label{cbr3}
\end{align}
where $[\cdot,\cdot]_L=\Cour{\cdot,\cdot}|_{\Gamma(L)}$,
$l,l_1,l_2 \in \Gamma(L)$, $\xi,\xi_1,\xi_2\in \Gamma(L^*)$ and
$\Lie_l=d_L i_l + i_l d_L$.

Conversely, one may start with a Lie algebroid
$(L,[\cdot,\cdot]_L,\rho_L)$ together with a skew-symmetric
bracket $F^*$ on $\Gamma(L)$, an element $\chi\in \Gamma(\wedge^3
L)$ and a bundle map $\rho_{L^*}:L^*\to TM$. This set of data is
called a \textbf{Lie quasi-bialgebroid} \cite{Dima,Dima2} if the
bracket defined by \eqref{cbr1}, \eqref{cbr2} and \eqref{cbr3}
makes $L\oplus L^*$ into a Courant algebroid with pairing
$\SP{\cdot,\cdot}_{can}$ and anchor $\rho_E=\rho_L+\rho_{L^*}$.
This requirement is equivalent to the following explicit
compatibility conditions \cite{Dima2}:
\begin{enumerate}
\item[Q0)] $d_{L^*}[l_1,l_2]_L = [d_{L^*}l_1,l_2]_L+
[l_1,d_{L^*}l_2]_L$, for all $l_1,l_2\in \Gamma(L)$.

(Using the Leibniz identity for $[\cdot,\cdot]_L$, one can check
that $d_{L^*}$ is actually a derivation of $[\cdot,\cdot]_L$ on
$\Gamma(\wedge L)$:
$d_{L^*}[l_1,l_2]_L=[d_{L^*}l_1,l_2]+(-1)^{p-1}[l_1,d_{L^*}l_2]$,
$l_1\in \Gamma(\wedge^{p}L), l_2\in \Gamma(\wedge^{q}L$).)

\item[Q1)] $\rho_{L^*}(F^*(\xi_1,\xi_2))=
[\rho_{L^*}(\xi_1),\rho_{L^*}(\xi_2)] -
\rho_L(i_{\xi_2}i_{\xi_1}(\chi))$, for all $\xi_1,\xi_2 \in
\Gamma(L^*)$.

\item[Q2)] $F^*(\xi_1,f\xi_2)=f F^*(\xi_1,\xi_2) +
\Lie_{\rho_{L^*}(\xi_1)}(f) \xi_2$, for all $\xi_1,\xi_2 \in
\Gamma(L^*), f\in C^\infty(M)$.

\item[Q3)] For all $\xi_1,\xi_2,\xi_3 \in \Gamma(L^*)$,
$$
F^*(\xi_1,F^*(\xi_2,\xi_3))+ c.p. = d_L\chi(\xi_1,\xi_2,\xi_3) +
i_{\chi(\xi_2,\xi_3)}d_L\xi_1 - i_{\chi(\xi_1,\xi_3)}d_L\xi_2 +
i_{\chi(\xi_1,\xi_2)}d_L\xi_3.
$$

\item[Q4)] $d_{L^*}\chi=0.$

\end{enumerate}
The resulting Courant algebroid  $L\oplus L^*$ is called the
\textbf{double} of the Lie quasi-bialgebroid. Hence we see that
there is a natural correspondence between split Manin pairs and
Lie quasi-bialgebroids over $M$.

\subsection{Bivector fields}\label{subsec:app4}
Given a Courant algebroid $E$ over $M$ and a pair of transversal
almost Dirac structures $L, C$, with $E=L\oplus C$, it follows
from \eqref{eq:rhorho*} that
$$
\rho_L\circ (\rho_C)^* + \rho_{C}\circ (\rho_L)^*=0,
$$
where $\rho_L:=\rho_E|_L$, $\rho_C:=\rho_E|_C$, and we identify
$C\cong L^*$ via the pairing on $E$. Hence the bundle map
$\pi^\sharp:=\rho_L\circ (\rho_C)^*:T^*M\to TM$ defines a bivector
field on $\pi$ on $M$, depending on $L$ and $C$. In particular,
any Lie quasi-bialgebroid over $M$ defines a bivector field
$\pi\in \X^2(M)$ \cite{ILX}.

\begin{proposition}\label{prop:biv}
For a given Lie quasi-bialgebroid $E=L\oplus L^*$, the bivector
field $\pi \in \X^2(M)$ defined by $\pi^\sharp= \rho_L\circ
(\rho_{L^*})^*$ satisfies
\begin{align}
&\frac{1}{2}[\pi,\pi] = \rho_L(\chi),\label{eq:bivtri}\\
&\Lie_{\rho_L(l)}\pi = \rho_L(d_{L^*}(l)),\;\; \forall l \in
\Gamma(L).\label{eq:invbiv}
\end{align}
\end{proposition}

\begin{proof}
For $f,g,h\in C^\infty(M)$, let $\mathrm{Jac}(f,g,h)=\{f,\{g,h\}\}
+ c.p.$, where $\{\cdot,\cdot\}$ is the bracket defined by $\pi$.
It then follows that (see e.g. \cite[Sec.~2.2]{BC})
\begin{equation}\label{eq:bivecid}
\frac{1}{2}[\pi,\pi](df,dg,dh)=\mathrm{Jac}(f,g,h) =
\SP{[\pi^\sharp(df),\pi^\sharp(dg)]-\pi^\sharp(d\{f,g\}),dh}.
\end{equation}
Using Q1) we see that $\rho_L(\tri)(df,dg,dh)$ equals
$$
\SP{\rho_L(i_{\rho_L^*(dg)}i_{\rho_L^*(df)}\tri),dh}=
\SP{[\pi^\sharp(df),\pi^\sharp(dg)],dh}-
\SP{\rho_{L^*}(F^*(\rho_L^*(df),\rho_L^*(dg))),dh},
$$
and, by \eqref{eq:bivecid}, this last expression equals
\begin{equation}\label{eq:interm}
\mathrm{Jac}(f,g,h) + \{\{f,g\},h\}-
\SP{F^*(\rho_L^*(df),\rho_L^*(dg)),\rho_{L^*}^*(dh)}.
\end{equation}
Using the identity \eqref{eq:da2} for the bracket $F^*$, we can
rewrite \eqref{eq:interm} as
$$
(d_{L^*}(\rho_{L^*}^*dh))(\rho_L^*df,\rho_L^*dg) + \{\{g,h\},f\} +
\{\{h,f\},g\} + \mathrm{Jac}(f,g,h) + \{\{f,g\},h\},
$$
which equals $(d_{L^*}(\rho_{L^*}^*dh))(\rho_L^*df,\rho_L^*dg).$
Hence
\begin{equation}\label{eq:interm2}
\rho_L(\tri)(df,dg,dh)=
(d_{L^*}(\rho_{L^*}^*dh))(\rho_L^*df,\rho_L^*dg) =
(d_{L^*}^2h)(\rho_L^*df,\rho_L^*dg).
\end{equation}
On the other hand, since $d_{L^*}[d_{L^*}f,g]_L=[d_{L^*}^2f,g]_L +
[d_{L^*}f,d_{L^*}g]_L$ (by Q0)), applying $\rho_L$ and using the
definiton of $\pi$ we get
$$
\Lie_{\rho_L(d_{L^*}[d_{L^*}f,g]_L)}h=\{\{f,g\},h\} =
\Lie_{\rho_L([d_{L^*}^2f,g])}h + \{f,\{g,h\}\} +\{g,\{h,f\}\}.
$$
A direct computation shows that
$\Lie_{\rho_L([d_{L^*}^2f,g])}h=-\rho_L(d^2_{L^*}f)(dg,dh)$, hence
$$
\mathrm{Jac}(f,g,h)=\rho_L(d^2_{L^*}f)(dg,dh)=d^2_{L^*}f(\rho_L^*dg,\rho_L^*dh).
$$
Using the skew-symmetry of $\mathrm{Jac}$ and \eqref{eq:interm2},
equation \eqref{eq:bivtri} follows.

To prove \eqref{eq:invbiv}, we use that
$d_{L^*}[l,f]_L=[d_{L^*}l,f]_L+[l,d_{L^*}f]_L$ for all $l\in
\Gamma(L)$. Applying $\rho_L$ to this expression, it follows that
$$
\{\Lie_{\rho_L(l)}f,g\}=\Lie_{\rho_L([d_{L^*}l,f]_L)}g +
\Lie_{[\rho_L(l),\pi^\sharp(df)]}g= \Lie_{\rho_L([d_{L^*}l,f]_L)}g
+\Lie_{\rho_L(l)}\{f,g\}-\{f,\Lie_{\rho_L(l)}g\}.
$$
Hence
$$
(\Lie_{\rho_L(l)}\pi)(df,dg)=\Lie_{\rho_L(l)}\{f,g\}-\{\Lie_{\rho_L(l)}f,g\}-
\{f,\Lie_{\rho_L(l)}g\}= - \Lie_{\rho_L([d_{L^*}l,f]_L)}g.
$$
Using the general identity
$\Lie_{\rho_L([\lambda,f])}g=-\rho_L(\lambda)(df,dg)$, for $f,g\in
C^\infty(M)$ and $\lambda\in \Gamma(\wedge^2L)$, we conclude that
$$
- \Lie_{\rho_L([d_{L^*}l,f]_L)}g=\rho_L(d_{L^*}l)(df,dg),
$$
as desired.
\end{proof}

\subsection{Twists and exact Courant algebroids}\label{subsec:app5}

Let $(E,L)$ be a Manin pair over $M$, and suppose that we have two
splittings of \eqref{eq:exact}, $s$ and $s'$. The image of the
difference $s-s':L^*\to E$ lies in $L$, hence it defines an
element $t\in \wedge^2L$, called a \textbf{twist}, by
$$
s-s'=t^\sharp : L^*\to L \subset E,
$$
where $t^\sharp(\xi_1)(\xi_2)=t(\xi_1,\xi_2)$. A direct
calculation shows the following:

\begin{proposition}\label{lem:split}
The following holds:
\begin{enumerate}
\item Let $d_{L^*}^s$ be the derivation on $\Gamma(\wedge L)$
associated with the bracket $F_s^*$ on $\Gamma(L^*)$ and bundle
map $\rho^s_{L^*}$. Then
$$
d^s_{L^*} = d^{s'}_{L^*}+[t,\cdot]_L.
$$

\item $\chi_s =\chi_{s'} + d^s_{L^*} t -\frac{1}{2}[t,t]_L$.

\item $\pi^{s'}=\pi^s + \rho_L(t)$.

\end{enumerate}
\end{proposition}

An important class of examples of Manin pairs is given by exact
Courant algebroids. Following P.~\v{S}evera \cite{Se}, a Courant
algebroid is called {\bf exact} if the sequence
\begin{equation}
0\rmap T^*M \stackrel{\rho^*_E}{\rmap} E \stackrel{\rho_E}{\rmap}
TM \rmap 0
\end{equation}
is exact (by \eqref{eq:rhorho*}, it is always true that
$\rho_E\rho_E^*=0$). Viewing $T^*M$ as a subbundle of $E$ via
$\rho_E^*$, $(E,L=T^*M)$ is a Manin pair. Using axioms C3) and C4)
in Sec.~\ref{subsec:app1}, one can check that
$[\cdot,\cdot]_L=\Cour{\cdot,\cdot}|_{\Gamma(T^*M)}\equiv 0$.
Since $\rho_L=0$, we must have $d_L=0$. Once we choose an
isotropic splitting $s$ and identify $E$ with $TM\oplus T^*M$, it
is simple to check that the bracket $F^*_s$ on
$\Gamma(L^*)=\Gamma(TM)$ is just the Lie bracket of vector fields.
>From Q4), we see that $\chi_s$ is a closed 3-form, and the general
bracket \eqref{cbr1}, \eqref{cbr2}, \eqref{cbr3} becomes the
bracket of Example \ref{ex:CA}. From Prop.~\ref{lem:split}, part
2, we see that a different splitting changes $\chi_s$ by an exact
3-form. These observations lead to the following result of
\v{S}evera \cite{Se}:

\begin{corollary}\label{cor:severa}
Exact Courant algebroids over $M$ are classified by
$H^3(M,\mathbb{R})$.
\end{corollary}

\begin{footnotesize}

\end{footnotesize}


\begin{thebibliography}{10}

\bibitem{AB} {\sc Atiyah, M., Bott, R.:}\newblock{\em The Yang-Mills equations over
Riemann surfaces}
\newblock Philos. Roy. Soc. London Ser. A {\bf 308} (1982),
523--615.

\bibitem{ABM} {\sc Alekseev, A., Bursztyn, H., Meinrenken, E.: }
\newblock{\em Pure spinors on Lie groups}, to appear in Ast\'erisque.
Arxiv: 0709.1452[math.DG].


\bibitem{AK}
{\sc Alekseev, A., Kosmann-Schwarzbach, Y.: }
\newblock {\em Manin pairs and moment maps}.
\newblock J. Differential Geom. {\bf 56} (2000), 133-165.

\bibitem{AKM}
{\sc Alekseev, A., Kosmann-Schwarzbach, Y., Meinrenken, E.: }
\newblock{\em Quasi-Poisson manifolds.}
\newblock Canadian J. Math. {\bf 54} (2002), 3--29.

\bibitem {AMM}
{\sc Alekseev, A., Malkin, A., Meinrenken, E.: }\newblock {\em Lie group valued
  moment maps}.
\newblock J. Differential Geom.  {\bf 48} (1998), 445--495.

\bibitem{AX}
{\sc Alekseev, A., Xu, P.:}\newblock{\em Derived brackets and
Courant algebroids}, unpublished manuscript.

\bibitem {BK-S}
{\sc Bangoura, M., Kosmann-Schwarzbach, Y.: }\newblock {\em The double of a Jacobian
quasi-bialgebra}.
\newblock Lett. Math. Phys.  {\bf 28} (1993), 13--29.

\bibitem{BSS}
{\sc Bott, R., Schulman, H., Stasheff, J.: }\newblock{\em On the
de Rham theory of certain classifying spaces}. \newblock Advances
in Math. \textbf{20} (1976), 43--56.

\bibitem {BCG}
{\sc Bursztyn, H., Cavalcanti, G., Gualtieri, M. : }\newblock {\em
Reduction of Courant algebroids and generalized complex
structures}.
\newblock Advances in Math.  {\bf 211} (2007), 726--765.

\bibitem {BC}
{\sc Bursztyn, H., Crainic, M.: }\newblock {\em Dirac structures,
moment maps and quasi-Poisson manifolds}. In: {\it The breadth of
symplectic and Poisson geometry. Festschrift in honor of Alan
Weinstein}, Progr. Math {\it 232}, Birkhauser, 2005, 1-40.

\bibitem {BCS}
{\sc Bursztyn, H., Crainic, M., \v{S}evera, P.: }\newblock {\em
Quasi-Poisson structures as Dirac structures}. \newblock Travaux
Mathematiques Fasc. XVI, 41--52, Univ. Luxemb., 2005.


\bibitem {BCWZ}
{\sc Bursztyn, H., Crainic, M., Weinstein, A., Zhu, C.: }\newblock {\em
  Integration of twisted Dirac brackets}. Duke Math. J. {\bf 123} (2004), 549--607.

\bibitem{BIS}
{\sc Bursztyn, H., Iglesias Ponte, D., \v{S}evera, P.:}
\newblock{\em Courant morphisms and moment maps}. \newblock To
appear in Math. Res. Letters. ArXiv:0801.1663[math.SG].

\bibitem{BuRa}
{\sc Bursztyn, H., Radko, O.:} \newblock{\em Gauge equivalence of
Dirac structures and symplectic groupoids}. Ann. Inst. Fourier
(Grenoble) \textbf{53} (2003), 309--337.

\bibitem{CaXu}
{\sc Cattaneo, A., Xu, P. :} \newblock{\em Integration of twisted
Poisson structures}. \newblock J. Geom. Phys. {\bf 49} (2004),
187--196.

\bibitem {Cou90}
{\sc Courant, T.: }\newblock {\em Dirac manifolds}.
\newblock Trans. Amer. Math. Soc.  {\bf 319} (1990), 631--661.

\bibitem {CouWe}
{\sc Courant, T., Weinstein, A.: }\newblock {\em Beyond Poisson structures}.
\newblock S\'eminaire sud-rhodanien de g\'eom\'etrie VIII.
\newblock Travaux en Cours {\bf 27}, Hermann, Paris (1988), 39-49.

\bibitem{CrFe}
{\sc Crainic, M., Fernandes, R.L. :} {\em Integability of Poisson
brackets}.\newblock J. Differential Geom. {\bf 66} (2004),
71--137.

\bibitem{Drinf}
{\sc Drinfeld, V.:} \newblock {\em Quasi-Hopf algebras}.
\newblock Leningrad Math. J.  {\bf 1} (1990), 1419--1457.

\bibitem{GS}
{\sc Guillemin, ~V., Sternberg, S.:}\newblock {\em Some problems
in integral geometry and some related problems in micro-local
analysis}, Amer. Jour. Math. \textbf{101} (1979), 915--955.

\bibitem{ILX}
{\sc Iglesias Ponte, D., Laurent-Gengoux, C., Xu, P.:} {\em
Universal lifting theorem and quasi-Poisson groupoids}. Arxiv:
math.DG/0507396.

\bibitem{KoSo}
{\sc Korogodski, L., Soibelman, Y.:}\newblock Algebras of
functions on quantum groups. Part I. {\em Mathematical Surveys and
Monographs}, 56. American Mathematical Society. Providence, RI,
1998.

\bibitem {K-S}
{\sc Kosmann-Schwarzbach, Y.: }\newblock {\em Jacobian quasi-bialgebras and quasi-Poisson
Lie groups}.
\newblock Contemporary Math. {\bf 132}, (1992), 459--489.

\bibitem {K-S2}
{\sc Kosmann-Schwarzbach, Y.: }\newblock {\em Quasi, twisted and all that... in Poisson geometry and Lie algebroid theory}.
\newblock In:
{\it The breadth of symplectic and Poisson geometry. Festschrift in honor of Alan Weinstein},
Progr. Math {\it 232}, Birkhauser, 2005, 363-389.

\bibitem{Lei}
{\sc Leingang, M.:}\newblock {\em Symmetric space valued moment
maps}.
\newblock Pacific J. Math. {\bf 212}, (2003), 103--123.

\bibitem {LWX}
{\sc Liu, Z.-J., Weinstein, A., Xu, P.:}\newblock {\em Manin triples for Lie algebroids}.
\newblock J. Differential Geom. {\bf 45}, (1997), 547--574.

\bibitem{Luthesis} {\sc Lu, J.-H.:}\newblock{Multiplicative and
affine structures on Lie groups}, Ph.D. thesis, University of
California at Berkeley, 1990.

\bibitem{Lu}
{\sc Lu, J.-H.:}\newblock{\em Momentum mappings and reduction of Poisson actions}.
\newblock In: {\em Symplectic geometry, groupoids and integrable systems (Berkeley, CA, 1989)}, 291--311.
Springer, New York, 1991.

\bibitem{Lu2}
{\sc Lu, J.-H.:}\newblock{\em Poisson homogeneous spaces and Lie
algebroids associated with Poisson actions}.
\newblock Duke Math. J. {\bf 86},(1997), 261--304.

\bibitem{Mac}
{\sc Mackenzie, K.:} \newblock{General theory of Lie groupoids and
Lie algebroids.} London Math. Soc. Lecture Notes Series {\bf 213}.
\newblock{\it Cambridge University Press, Cambridge}, 2005.

\bibitem{MW}
{\sc Marsden, J., Weinstein, A.:}\newblock {\em Reduction of
symplectic manifolds with symmetry}. Rep. Mathematical Phys. {\bf
5} (1974), 121--130.

\bibitem{MiWe}
{\sc Mikami, K., Weinstein, A.:}\newblock{\em Moments and reduction for symplectic groupoid actions}
\newblock Publ. RIMS, Kyoto Univ. {\bf 24} (1988), 121-140.

\bibitem{Dima}
{\sc Roytenberg, D.:}\newblock{\em Courant algebroids, derived
brackets and even symplectic supermanifolds}
\newblock Phd. Thesis, Berkeley, 1999. ArXiv: math.DG/9910078.

\bibitem{Dima2}
{\sc Roytenberg, D.:}\newblock{\em Quasi-Lie bialgebroids and
twisted Poisson manifolds.} \newblock Lett. Math. Phys. {\bf 61}
(2002), 123--137.

\bibitem{Se}
{\sc \v{S}evera, P.:}\newblock{ Letters to A. Weinstein}.
available at
http://sophia.dtp.fmph.uniba.sk/$\sim$severa/letters/.


\bibitem {SeWe01}
{\sc \v{S}evera, P., Weinstein, A.: }\newblock {\em Poisson geometry with a
  $3$-form background}.
\newblock Prog. Theo. Phys. Suppl.  {\bf 144} (2001), 145--154.

\bibitem{XS}
{\sc Stienon, M., Xu, P.:}\newblock{\em Reduction of generalized
complex structures}, arXiv: math.DG/0509393.


\bibitem{Ter}
{\sc Terashima, Y.:}\newblock{\em On Poisson functions}. \newblock
J. Symplectic Geom. {\bf 6} (2008), 1--7.

\bibitem{Vain}
{\sc Vaintrob, A.:}\newblock {\em Lie algebroids and homological
vector fields}, \newblock Russian Math. Surveys {\bf 52} (1997),
428-429.

\bibitem{Wecat}
{\sc Weinstein, A:} Lectures on symplectic manifolds. \newblock
CBMS Regional Conference Series in Mathematics, 29. {\em American
Mathematical Society, Providence, R.I.}, 1979.

\bibitem{Wesym}
{\sc Weinstein, A:} {\em Symplectic groupoids and Poisson
manifolds}. \newblock Bull. Amer. Math. Soc. (N.S.) {\bf 16}
(1987), 101--104.

\bibitem {We}
{\sc Weinstein, A.: }\newblock {\em The geometry of momentum}.
\newblock G\'eom\'etrie au XX\`eme Si\`ecle, Histoire et
Horizons, Hermann, Paris, 2005. \newblock Arxiv: Math.SG/0208108.


\bibitem {Xu}
{\sc Xu, P.: }\newblock {\em Morita equivalence and momentum maps}.
\newblock J. Differential Geom. {\bf 67} (2004), 289-333.

\end{thebibliography}
\end{document}